\documentclass[11pt,twoside]{article}

\usepackage{amsmath}
\usepackage{amssymb}
\usepackage{mathrsfs}
\usepackage{amsthm}

\usepackage{latexsym}

\usepackage{indentfirst}
\usepackage{color}
\usepackage{txfonts}

\usepackage{anysize}

\allowdisplaybreaks

\usepackage[colorlinks=true,
  linkcolor=red,
  citecolor=blue,
  urlcolor=magenta]{hyperref}

\def\R{{\mathbb R}}
\def\rn{{{\R}^n}}

\def\Q{\mathcal{Q}}
\def\B{\mathcal{B}}
\def\S{\mathcal{S}}
\def\A{\mathcal{A}}
\def\d{\mathrm{d}}
\def\M{\mathcal{M}}
\def\F{\mathcal{F}}
\DeclareMathOperator*{\esssup}{ess\,sup}
\def\scale{\operatorname{scale}}

\newtheorem{theorem}{Theorem}[section]
\newtheorem{lemma}[theorem]{Lemma}
\newtheorem{corollary}[theorem]{Corollary}
\newtheorem{proposition}[theorem]{Proposition}

\theoremstyle{definition}
\newtheorem{remark}[theorem]{Remark}
\newtheorem{definition}[theorem]{Definition}

\numberwithin{equation}{section}

\begin{document}
\title{\bf\Large Equivalent norms and $\varphi$-transform of matrix-weighted anisotropic Besov-type and Triebel-Lizorkin-type spaces
\footnotetext{\hspace{-0.35cm} 2020 {\it
Mathematics Subject Classification}. Primary 46E35; Secondary   42B25, 42B35. \endgraf
{\it Key words and phrases: matrix weight, anisotropic Besov-type space, anisotropic Triebel-Lizorkin-type space, $\varphi$-transform}. 
\endgraf
The work is supported by the Natural Science Foundation of Guangxi (No. 2026GXNSFBA00640069),  the Science and Technology Project of Guangxi (Guike AD25069086), the National Natural Science Foundation of China (Grant No.
12561002) and Hainan Provincial Natural Science Foundation of China (Grant No. 126MS0135).
}}
\date{}
\author{}
\author{Tengfei Bai, Pengfei Guo and Jingshi Xu\footnote{Corresponding author,
E-mail: \texttt{jingshixu@126.com}}}
\maketitle

\vspace{-0.8cm}

\begin{center}
\begin{minipage}{13cm}
{\small {\bf Abstract:}\quad
We introduce matrix-weighted anisotropic Besov-type and Triebel-Lizorkin-type spaces associated with an expansive matrix $A$. 
Inspired by the $\A_p$-dimensions of matrix weight of Bu et al. (2025), we study the properties of matrix weight associated with $A$. Using the nice properties of matrix weight, we obtain that these spaces are equivalent with their corresponding averaging spaces and establish the discrete $\varphi$-transform of these spaces.
Finally, we introduce matrix-weighted anisotropic  Triebel-Lizorkin spaces for the limiting case $p=\infty$ and obtain their  $\varphi$-transform. The relation between matrix-weighted anisotropic  Triebel-Lizorkin spaces and matrix-weighted anisotropic Besov-type and Triebel-Lizorkin-type spaces is also studied.
}
\end{minipage}
\end{center}

\vspace{0.2cm}

\tableofcontents

\section{Introduction}
Many fields in analysis   require the study of specific function spaces. After a long history of development,
Besov  and Triebel-Lizorkin spaces provide a unified framework for many function spaces.
Indeed in 1951, Nikol'ski\u i \cite{Ni51} introduced the Nikol'ski\u i-Besov spaces $B^s_{p,\infty} (\rn)$. Besov  \cite{Be59,Be61} introduced  Besov spaces by adding the third index $q$. Around 1970, Lizorkin \cite{Li72,Li74} and Triebel \cite{Tr73} independently studied Triebel-Lizorkin spaces. Peetre \cite{Pe73, Pe75, Pe76} extended the ranges of admissible parameters $p,q$ to values less than one. 
Frazier and Jawerth \cite{FJ85, FJ90}  studied   the $\varphi$-transform on Besov and Triebel-Lizorkin spaces.
We refer to monographs \cite{Tr83, Tr92, Tr06, Tr20} of Triebel  as well as the monograph \cite{Sa18} of Sawano for the  theory of these spaces.

The study of Euclidean spaces equipped with non-isotropic dilation structures is a  direction of extending classical function spaces arising in harmonic analysis. 
In \cite{Bo03}, Bownik first introduced  anisotropic Hardy spaces. Then he  \cite{Bo05} developed the theory of weighted anisotropic Besov spaces (homogeneous and inhomogeneous) associated with expansive dilations with the use of the discrete $\varphi$-transform of Frazier and Jawerth.  Weighted anisotropic Triebel-Lizorkin spaces were introduced and studied  with the use of the discrete $\varphi$-transform  by Bownik  and Ho in \cite{BH06}.
In \cite{Bo07}, Bownik dealt with the analysis  of  anisotropic Triebel-Lizorkin spaces on $\rn$ coming with $A$-doubling measure where $A$ is a  dilation matrix. The real and complex interpolation of the above spaces and their duality can be found in \cite{Bo08, LBYY12}.

The Besov-type and Triebel-Lizorkin-type spaces were introduced by Yang et al. \cite{SYY10,YY08,YY10},  which include many well-known function spaces, such as Besov spaces, Triebel-Lizorkin spaces, BMO spaces, $Q$ spaces, Morrey spaces, Campanato spaces, Besov-Morrey spaces, Triebel-Lizorkin-Morrey spaces. We refer to the monograph \cite{YSY05} by Yuan et al. for  many characterizations and  applications of these spaces.

In the last three decades, the theory of matrix-weighted function spaces has been developed. 
Indeed, in 1997, Treil and Volberg \cite{TV97} introduced Muckenhoupt $\mathcal A_{2}$ matrix weights.
Later, 
Goldberg \cite{Go03} established the $\mathcal  A_{p}$ condition in
terms of matrix weights for $p \in (1,\infty)$. He also obtained that the matrix $\mathcal A_{p}$ condition
implies $L^{p}$-boundedness of the Hardy-Littlewood maximal operator.
In \cite{CMR16},  Cruz-Uribe et al.  studied degenerate Sobolev spaces where the degeneracy is controlled by a
matrix $\mathcal A_p$ weight. As applications, they obtained the weak solutions of degenerate $p$-Laplace equations and mappings of finite distortion.
In \cite{FR04,FR21,Ro03}, Frazier and Roudenko
introduced the matrix weight class $\mathcal A_{p}$ for $p\in(0,1]$ and matrix-weighted Besov and Triebel-Lizorkin spaces. For matrix weight $W\in \mathcal A_{p},$
they showed that homogeneous (and inhomogeneous) matrix-weighted Triebel-Lizorkin
spaces and Besov spaces are equivalent with their discrete spaces and the corresponding averaging spaces. For $p\in(1,\infty)$ and $W\in \mathcal  A_{p}$, they showed that
matrix-weighted Lebesgue space $L^{p}(W)=\dot{F}_{p}^{0,2}(W)$ and
that the matrix-weighted Sobolev space $L_{k}^{p}(W)=F_{p}^{k,2}(W)$. 
The trace operator, an extension operator and duality of matrix-weighted Besov space were obtained in \cite{ FR08, Ro04}.

After this, matrix-weighted function spaces  received more and more attention.
In \cite{BYY23}, Bu, Yang and Yuan established wavelet and molecular characterizations of matrix-weighted Besov spaces defined on spaces of homogeneous type (in the sense of R. Coifman and G. Weiss).
In \cite{WYZ24}, Wang, Yang and Zhang studied the Peetre maximal characterization, the Lusin-area function characterization, the $g_\lambda^\ast$-function characterization of homogeneous matrix weighted Triebel-Lizorkin spaces. 
The matrix-weighted Besov-Triebel-Lizorkin spaces with logarithmic smoothness were studied in \cite{LYY24}.

In \cite{BHYY24,BHYY251, BHYY252},  Bu et al. introduced the matrix-weighted
Besov-type spaces and Triebel-Lizorkin-type spaces and the related sequence spaces $\dot{b}_{p,q}^{s,\tau}(W)$
and $\dot{f}_{p,q}^{s,\tau}(W)$. They proposed a new concept of $\mathcal A_{p}$-dimension of matrix weights. The boundedness of $\varphi$-transform,   pseudo-differential operators, trace operators and Calder\'{o}n--Zygmund operators was obtained. The optimal characterizations of molecules and wavelets, trace theorems and the sharp boundedness of almost diagonal operators were also proved.
The new characterizations of matrix weights classes $\mathcal A_{p,\infty}$ were studied in \cite{BYYH26}
and the  concepts of variable scalar 
weights  $A_{ p(\cdot), \infty } $  and variable matrix  weights $\mathcal A_{ p(\cdot) , \infty } $ were introduced in \cite{YYZ26}.
We refer to \cite{BHYY26} for the inhomogeneous Besov-type and Triebel-Lizorkin-type spaces with matrix weights classes $\mathcal A_{p,\infty}$, 
to \cite{BYYZ26, YYZ25} for generalized matrix-weighted Besov-Triebel-Lizorkin-type spaces with matrix $\mathcal A_{p,\infty}$ weights,
to \cite{CYYZ26} for the variable Hardy space with the  $\mathcal A_{ p(\cdot), \infty } $ matrix weight by means of the matrix-weighted grand maximal function,
to \cite{BCYY25} for matrix-weighted Hardy spaces via the matrix-weighted grand non-tangential maximal function,
to \cite{BX242, BX24, BX25, MX25,  WGX25, WGX252} for the theory of other matrix-weighted function spaces,
to the survey \cite{BYYZ25} for the development of matrix-weighted function spaces,
to \cite{Cz25} for recent results on matrix weighted norm inequalities.

In \cite{LYY26}, Li, Yang and Yuan introduced the vector-valued Haj\l asz gradient sequences and established some related matrix-weighted Poincar\'e-type inequalities on a space of homogeneous type. As an application, they introduced the matrix-weighted logarithmic Besov spaces on a space of homogeneous type
and established their pointwise characterization via Haj\l asz gradient sequences.
In \cite{BCHYY26}, Bu et al. developed the  theory for a general class of multi-parameter function spaces of Besov-Triebel-Lizorkin type with a matrix weight. They proved the equivalence of different quasi-norms, molecular and wavelet characterisations,  Sobolev-type embedding theorems, the boundedness of almost diagonal operators and multi-parameter singular integrals.

Recently, Liu and Wang \cite{LW25}  studied the matrix-weighted homogeneous (inhomogeneous) anisotropic Besov spaces and the boundedness of $\varphi$-transform. 
 Nielsen \cite{Ni25} studied  matrix $\A_p$ weights  relative to a family
of anisotropic balls in $\rn$ defined by a pseudo-metric. Later, he \cite{Ni26} introduced and researched anisotropic matrix-weighted smoothness spaces in both continuous and discrete settings.

Motivated by above literature, we introduce and study matrix-weighted anisotropic Besov-type and Triebel-Lizorkin-type spaces associated with a dilation.
The main results of this paper are their equivalent norms and $\varphi$-transform characterizations.
This paper is organized as follows. In Section \ref{preliminaries},
we recall some results about the expansive dilations and some lemmas.
 In Section \ref{sec matrix weight}, we study the matrix weight associated with an expansive matrix and a key estimate via reducing operators. 
In Section \ref{sec BF spaces}, we introduce matrix-weighted anisotropic Besov-type and Triebel-Lizorkin-type space and obtain their equivalent norms. Moreover, we establish the $\varphi$-transform characterization of these spaces. As applications, we obtain that these spaces are independent of the choice of $\varphi$ and the boundedness of the lifting operator.
In Section \ref{sec p infty}, we introduce the averaging matrix-weighted Triebel-Lizorkin space and the corresponding sequence space and obtain the $\varphi$-transform characterization of these spaces.

Throughout this paper,  let $c, C$ denote constants that are independent of the main parameters involved but whose value may differ from line to line.
For $A,B>0$, by $A\lesssim B$, we mean that $A\leq CB$ with some positive constant $C$ independent of appropriate quantities. By $ A \approx B$, we mean that $A\lesssim B$ and $B\lesssim A$.
For any measurable function $f$ on $\rn$ and any measurable set $E \subset \rn$ with $0< |E| <\infty$, define $\fint_E f(x) \d x := |E|^{-1} \int_E f(x) \d x$ where $|E|$ is the Lebesgue measure of $E$.
For $a,b\in \mathbb R$, let $a \vee b := \max\{ a,b\} $  and $a \wedge b : =  \min\{ a,b\} $.
Let $\mathbb N:= \{1,2,\ldots\}$, $\mathbb Z_+ := \mathbb N \cup \{0\}$ and $\mathbb Z$ be the set of all integers.

\section{Preliminaries} \label{preliminaries}
\subsection{Basic facts about expansive dilations}
For $m \in \mathbb N$,
let $M_m (\mathbb C)$ be the set of all $m \times m$ complex-valued matrices. For a matrix $A :=[ a_{ij} ] \in M_m (\mathbb C) $, the conjugate of $A$, denoted by $\bar A$, is the matrix in $M_m (\mathbb C)$ whose $(i,j)$ entry is $\bar a_{ij} $, the transpose of $A$, denoted by $A^T $, is the matrix in $M_m (\mathbb C)$ whose $ (i,j)$ entry is $a_{ji}$, and the conjugate transpose of $A$ denoted by $A ^\ast : = \overline {A^T }$.

A real $n \times n$ matrix $A$ is an expansive matrix, sometimes called shortly a dilation, if $\min _{\lambda  \in \sigma (A)}  |\lambda| >1  $, where $\sigma (A)$ is the set of all eigenvalues (the spectrum) of $A$.
A basic notion is a quasi-norm $\rho_A$  associated with $A$, which induces a quasi-distance making $\rn$ a space of homogeneous type.

\begin{definition}\label{def rho_A}
	A quasi-norm associated with an expansive matrix $A$ is a measurable	mapping  $\rho_A :\rn \to [0,\infty) $ satisfying
	
	(i) $\rho_A (x) >0$ for $x \neq 0$,
	
	(ii) $\rho _A ( Ax) = |\det A| \rho_A (x)$ for $x\in \rn$,
	
	(iii) $\rho_A (x +y) \le C ( \rho_A (x) + \rho_A (y)) $ for $x,y \in \rn$ where $C\ge 1$ is a constant.
\end{definition}

In the standard dyadic case $A = 2I_n$  where $I_n$ is the $n\times n$ identity matrix, a quasi-norm $\rho_A$  satisfies $\rho_A (2x) = 2^n \rho_A(x)$  instead of the usual scalar homogeneity.  In particular, $\rho_A (x) = |x|^n$  is an example for a quasi-norm for $A= 2 I_n$, where $|\cdot| $ is the Euclidean norm in $\rn$.

We refer the reader to \cite{Bo03, Ho02} for a list of properties of quasi-norms associated with  an expansive matrix. Recall some basic facts needed in this work. From \cite[Lemma 2.4]{Bo03}, all quasi-norms associated to a fixed dilation $A$ are equivalent. Moreover, there always exists a quasi-norm $\rho_A$, which  is $C^\infty$ on $\rn$ except the origin; see \cite{LP94}. For our purposes it is enough to restrict to a quasi-norm $\rho_A$ given by 
\begin{equation} \label{eq rho_A baisic}
	\rho_A (x)  = \sum _{k= -\infty}^\infty |\det A|^k \chi_{ O _k} (x),
\end{equation}
where $O_k = A^k ( B(0,1) )  \backslash \bigcup_{j =-\infty }^{k-1} A^j ( B(0,1)) $, and $B(x,r) = \{y\in \rn : |x-y| <r   \} $. Equivalently,
\begin{equation*}
	\rho_A (x) = |\det A|^k, \quad \operatorname{where} \; k =\inf \{ j \in \mathbb Z :A^{-j}x \in B(0,1) \}
\end{equation*}
for $x \neq 0$ and $\rho_A (x) = 0$ for $x=0$. It is clear that $\rho_A $ given by (\ref{eq rho_A baisic}) satisfies (iii) of Definition \ref{def rho_A} with the constant $C = |\det A|^{j_0}$, where $j_0$ is the smallest integer such that $ \bigcup_{j\le 0} A^{j } ( B (0,2) )  \subset A^{j_0}(  B (0,1) ) $. Moreover, the 
above quasi-norm satisfies 
\begin{equation} \label{eq  rho A B approx r}
	| \{ x \in \rn : \rho_A (x) < r\} | \approx r  \quad \operatorname{for \; any} \; r>0.
\end{equation}
Since all quasi-norms associated a fixed dilation $A$ are equivalent, (\ref{eq  rho A B approx r}) holds for any quasi-norm $\rho_A$ associated with $A$.

It should be remarked that the quasi-norm $\rho_A$ given by (\ref{eq rho_A baisic}) might produce $\rho_A$-balls $\{x\in \rn: \rho_A (x)<r\}$, which are not convex. However, it is possible to modify the above construction to guarantee that $\rho_A$-balls are convex; see \cite[p. 5]{Bo03}. Therefore, we will simply assume that $\rho_A$-balls are convex.

\begin{definition}
	Let $\B$ be the collection of all  $\rho_A$-balls	\begin{equation*}
		B_{\rho_A} (x, r) = \{ y \in \rn : \rho_A (x-y)  <r\} ,  \quad  x\in \rn , r>0.
	\end{equation*}
For $  B_{\rho_A} (x, r) \in \B $ and $a >0$, let $ a B_{\rho_A} (x, r) = B_{\rho_A} (x, ar) $. For a $\rho_A$-ball $B$, let $c_B$ be its center and $r_B$ be its radius.
	
	For $j \in \mathbb Z$ and $k \in \mathbb Z^n$, let $Q_{j,k} = A ^{-j}  ( [0,1) ^n + k )$ be the dilated cube, and $x_{ Q_{j,k} }  =A ^{-j} k  $  be its lower-left corner. For $Q=Q_{j,k} \in \Q$, define its center $c_Q := A ^{-j}  ( (1/2, \ldots, 1/2) + k )$.
	Let $\mathcal Q_j   = \{Q_{j,k} : k \in \mathbb Z^n  \}  $ for $j \in \mathbb Z$.
	Let $\mathcal Q  =\mathcal Q _A = \{Q_{j,k} : j \in \mathbb Z, k \in \mathbb Z^n  \}  $ be the set of all dilated cubes.
	
	The scale of a ball $B =B_{\rho_A} (x, r)  $ is defined  as $\scale (B) = \lfloor \log_{ |\det A| } r  \rfloor$. The scale of dilated cube $Q \in \Q$ is defined as $\scale (Q) = \log_{ |\det A| }  |Q|$.
\end{definition}

By renormalizing $\rho_A$, it is convenient to assume that  $|B_{\rho_A} (x,1) | =1 $. Consequently,
\begin{equation*}
|B_{\rho_A} (x,  |\det A|^j  ) | = |\det A|^j \quad \operatorname{for \; any}  \; j \in \mathbb Z
\end{equation*}
Therefore,
\begin{equation*}
	|\det A|^{\scale (B) } \le |B| \le 	|\det A|^{\scale (B) +1} 
\end{equation*}
and for any $Q\in \Q$, $B\in  \B$ with $\scale (Q) = \scale (B)$,
\begin{equation*}
	|Q| \le |B| \le |\det A| |Q|.
\end{equation*}

\begin{lemma}[Lemma 2.2, \cite{BH06}]  \label{lemma rho A and |x|}
	Suppose $A$ is an expansive matrix, and $\lambda_- $ and $ \lambda_+$ are any positive real numbers such that  $1< \lambda_{-} <\min_{\lambda  \in \sigma(A)} |\lambda|   \le \max_{\lambda \in \sigma (A) } <\lambda_+ $.  
	Let 
	\begin{equation*}
	\zeta_+ := \frac{\log_e \lambda_+}{\log_e |\det A|}  \quad  \operatorname{and} \quad \zeta_- := \frac{\log_e \lambda_-}{\log_e |\det A|}.
	\end{equation*}
 Then for any quasi-norm $\rho_A$, there exists a constant $C$ such that
	\begin{equation*}
		C^{-1} \rho_A (x) ^{\zeta_-}  \le |x| \le C \rho_A (x) ^{\zeta_+}, \quad {\rm if} \; \rho_A (x) \ge 1
	\end{equation*}
and 
	\begin{equation*}
	C^{-1} \rho_A (x) ^{\zeta_+}  \le |x| \le C \rho_A (x) ^{\zeta_-},\quad {\rm if}\; \rho_A (x) \le 1.
\end{equation*}
Furthermore, if $A$ is diagonalizable over $\mathbb C$, then we may take $\lambda_- =  \min_{\lambda  \in \sigma(A)} |\lambda|$ and $\lambda_+ =  \max_{\lambda  \in \sigma(A)} |\lambda|$.
\end{lemma}

\begin{lemma}[Lemma 2.9, \cite{Bo07}]  \label{lemma basic cube ball cover}
	Let $A$ be an expansive matrix. Let $\B$  be the set of all $\rho_A$-balls. Then there exist constants $C_\B$,  $C_{\Q} >0$	such that:

{\rm (i)} for any $Q \in \Q $, we have
	\begin{equation*}
		B_0 \subset Q \subset B_1, 
	\end{equation*}
where 
\begin{equation*}
	B_0 :=B_{\rho_A} (c_Q, |Q| |\det A|^{-C_{\B} }) , B_1 :=B_{\rho_A} (c_Q, |Q| |\det A|^{C_{\B} });
\end{equation*}
	
	{\rm (ii)} for any $B\in \B$, the collection
	\begin{equation*}
		\Q_B = \{ Q\in \Q: Q \cap B  \neq  \emptyset, \scale (Q) = \scale (B) \}
	\end{equation*}
	has at most  $C_\Q$ elements. 
\end{lemma}

\begin{definition}
	We say that a cube $Q\in \Q$  is stacked below the cube $P\in\Q$, and write $Q \preceq P$, if there is a chain of cubes  $Q = Q_0,  Q_1 ,\ldots, Q_s=P \in \Q $ such that
	\begin{equation*}
		\scale (Q_i) < \scale (Q_{i+1})  \;\operatorname{and }\; |Q_i \cap Q_{i+1} | >0 \quad \operatorname{for \; all}\;  i= 0, \ldots , s-1.
	\end{equation*}
	The relation  $\preceq$ induces a partial order in $\Q$.	
	For a dilated cube $Q\in \Q$,  let $Q ^\sharp : = \{P\in\Q :  P \preceq Q, | P \cap Q| >0  \} $.
\end{definition}

\begin{remark}
	
	Suppose that $Q^\prime$ is a subfamily of $\Q$. Let $ \max(Q^\prime)$ be the set of maximal elements
	in $Q^\prime$ with respect to the relation $\preceq$. If a subfamily $Q^\prime$ does not contain arbitrary large cubes, i.e., $\sup_{Q\in  \Q ^\prime} \scale (Q)<\infty$,
	then for any cube $Q \in \Q^\prime $, there is always a cube $P\in  \max(Q^\prime) $ with $Q \preceq P$.	
	In general, a maximal cube   $P$ is not unique unless, for example, the dilation  $A = 2I_n $ and we work with nicely nested dyadic cubes.

\end{remark}

\begin{lemma} [Lemma  6.5, \cite{Bo07}] \label{lemma stacked cube}
	There is a universal constant $\eta\in \mathbb N$ such that whenever we have two cubes $Q, P \in \Q$ with $Q \preceq P = A^{j_0} ( [0,1)^n+k_0 )$, then 
	\begin{equation*}
		Q\subset \bigcup_{|k-k_0 |< \eta }  A^{j_0} ( [0,1)^n+k ) .
	\end{equation*}
\end{lemma}

\begin{lemma} \label{lemma use covered Q}
	For any $Q\in \Q$, there exists a finite set of disjoint  dilation cubes $\{ Q_\ell\}_{\ell =1}^{C_\Q} $ such that 
	$ Q^\sharp \subset \bigcup_{ \ell=1 }^{C_\Q}  Q_\ell $, $ |Q| \approx \sum_{\ell=1}^{C_\Q  }|Q_\ell|$ and $Q_\ell $ is  near $Q$ for each $\ell \in  \{ 1, \ldots, C_\Q \}$.
\end{lemma}

\begin{proof}	
	Fix $Q = A^{j_0} ( [0,1)^n+k_0 ) $ for some $j_0 \in \mathbb Z$ and $k_0 \in \mathbb Z^n$.
	Since $Q ^\sharp : = \{P\in\Q :  P \preceq Q, | P \cap Q| >0  \} $, by Lemma \ref{lemma stacked cube}, there is a universal constant $\eta\in \mathbb N$ such that
	\begin{equation*}
	Q ^\sharp  \subset \bigcup_{|k-k_0 |< \eta }  A^{j_0} ( [0,1)^n+k ).
	\end{equation*}
 We rearrange these cubes and denote them as $Q_1, \ldots, Q_{C_\Q} $. Thus we finish the proof.
\end{proof}

\subsection{Discrete wavelet transforms}
Let $\S(\rn)$ be the spaces of all Schwartz functions on $\rn$  and let $\S^\prime (\rn)$ be its dual.
Suppose that $\varphi,\psi \in \mathcal S(\rn)$  such that 
\begin{equation} \label{eq varphi supp}
	\operatorname{supp} \F \varphi , \operatorname{supp}  \F \psi \subset [-\pi,\pi]^n \backslash \{0\},
\end{equation}
\begin{equation} \label{eq varphi > 0}
	\sup_{j \in \mathbb Z}  | \F \varphi (  (A^\ast) ^j \xi ) | >0 , 
\end{equation}
\begin{equation} \label{eq varphi psi =1}
	\sum_{j\in \mathbb Z} \overline{\mathcal F \phi (  (A^\ast) ^j \xi )} \F \psi ( (A^\ast) ^j \xi ) =1, 
\end{equation}
for all $\xi \in \rn \backslash\{0\}$ where $A^\ast$ is the adjoint (transpose) of $A$. Here, supp $\F\phi = \overline{\{  \xi \in \rn: \F\phi (\xi )  \neq 0\}}$ and the Fourier transform of $f$ is 
\begin{equation*}
	\F f (\xi) = \int_\rn f (x) e^{-2\pi x \cdot \xi} \d x.
\end{equation*} 
The inverse Fourier of $f$  is defined by 
\begin{equation*}
		\F ^{-1} f (\xi) = \int_\rn f (x) e^{2\pi x \cdot \xi} \d x.
\end{equation*}
For $\varphi \in \mathcal S (\rn)$, define
\begin{equation*}
	\varphi_j (x) = | \det A|^j \varphi (A^j x) \quad  \operatorname{for} \; j \in \mathbb Z,
\end{equation*}
and 
\begin{equation*}
	\varphi_Q (x) =  | \det A|^{j/2} \varphi (A^j x -k) = |Q|^{1/2} \varphi_j (x- x_Q) \quad  \operatorname{for} \; Q = Q_{j,k}  \in \Q.
\end{equation*}
Hence $\F \varphi_j (\xi) = \F \varphi ( (A^\ast)^{-j} \xi )$ and supp $\F \varphi_j \subset (A^\ast)^{j} [ -\pi,\pi]^n$.

The closed subspace $\S_\infty (\rn)$ of the Schwartz class $\S (\rn)$ is given by 
\begin{equation*}
	\S_\infty (\rn) :=\left \{ \phi\in \S(\rn) :\int_\rn \phi (x) x^\alpha  \d x = 0 \; {\rm for \; any \; }  \alpha \in \mathbb N_0^n \right \}.
\end{equation*}
Denote by $\S_\infty (\rn) ^\prime$ the space of all continuous linear functionals on $\S_\infty (\rn) $ equipped with the weak-$\ast$ topology. It is well known that $\S' / \mathcal P :=\S' / \mathcal P (\rn) $ can be identified with $\S_\infty (\rn) ^\prime$ where $\mathcal P$ is the class of all polynomials in $\rn$.

\begin{lemma}[Lemma 2.8, \cite{BH06}] \label{lemma identity}
	Suppose that $A$ is an expansive matrix. If $g\in \S ' (\rn)$, $h\in \S(\rn)$ and supp $\F g, \F h  \subset (A^\ast)^j  [-\pi,\pi]^n$ for some $j \in \mathbb Z$, then
	\begin{equation*}
		g*h (x) = \sum_{k\in \mathbb Z} |\det A|^{-j} g (A^{-j} k)  h (x - A^{-j} k),
	\end{equation*}
with convergence in $\mathcal S' (\rn)$.
Consequently, if $\varphi, \psi \in \S^{\prime} (\rn)$ satisfy (\ref{eq varphi supp}), (\ref{eq varphi psi =1}), then  for any $f \in \S_\infty ^\prime (\rn)$,
\begin{equation*}
	f = \sum_{Q\in\Q} \langle f , \varphi_Q \rangle \psi_Q
\end{equation*}
where the convergence of the above series, as well as the equality, is in $ \S_\infty ^\prime (\rn)$. More precisely, there exists a sequence of polynomials $\{P_k\}_{k=1}^\infty  \subset \mathcal P$  and $P \in \mathcal P$  such that 
\begin{equation*}
	f = \lim_{k\to \infty}   \left(  \sum_{Q\in \Q, |\det A|^{-k} \le |Q| \le  |\det A|^{k}    }  \langle f , \varphi_Q \rangle \psi_Q + P_k \right) +P,
\end{equation*}
with convergence in $\S ^\prime (\rn)$.
\end{lemma}

\subsection{Hardy-Littlewood maximal operator}
For a function $f\in L^1_{\operatorname{loc}}$, the Hardy-Littlewood maximal operator $\M_{\rho_A}$ is defined by 
\begin{equation*}
	\M_{\rho_A} f (x) = \sup_{x\in B \in \B} \frac{1}{|B|}  \int_B |f(y) | \d y.
\end{equation*}



The following estimate is contained in the proof of \cite[Theorem 3.6]{Bo03}.

\begin{lemma} \label{lemma a >1 less M}
	For any expansive matrix $A$, any $j\in \mathbb Z$, $f\in L_{\operatorname{loc}} ^1$  and $ x \in \rn$, if $a>1$, then for $x\in \rn $,
	\begin{equation*}
		\int_\rn \frac{ |\det A|^j   | f(y)| }{ (1 + \rho_A  ( A^j (x-y) ) ) ^a } \d y \lesssim \M_{\rho_A} f (x).
	\end{equation*}
\end{lemma}


\begin{lemma}[Theorem 2.4, \cite{BH06}]  \label{lemma M weak 1,1 strong p,p}
	Let $A$ be an expansive matrix. Then $\M_{\rho_A}$ is of weak type (1,1) on $L^1$ and is bounded on $L^p$ for $1<p \le \infty$.
\end{lemma}

\begin{lemma}[Theorem 2.5, \cite{BH06}] \label{lemma M rho A FS}
	Let $A$ be an expansive matrix. 
	Let $1<p <\infty$ and $ 1<q\le \infty$. Then 
	\begin{equation*}
		\left\|  \left( \sum_{j \in \mathbb Z} ^\infty ( \M_{\rho_A} f_j )^q   \right) ^{1/q}   \right\|_{L^p }   \lesssim  \left\|  \left( \sum_{j \in \mathbb Z} ^\infty  |f_j| ^q   \right) ^{1/q}   \right\|_{L^p } .
	\end{equation*}
\end{lemma}


%
%
Let $j\in \mathbb Z$.
Then for  $y \in A ^{-j} [0,1)^n$,  $Q, R \in \Q_j$, and any $x\in Q$, we have
\begin{equation} \label{eq baisic geo}
	1 +   |\det A|^j \rho_A (x - x_R )   \approx 	1 +   |\det A|^j  \rho_A (x_Q - x_R )   \approx 	1 +   |\det A|^j  \rho_A (x_Q - x_R + y ) .
\end{equation}
 We will often use (\ref{eq baisic geo}) without explicit reference.

\begin{definition}
	Given a sequence $ u = \{  u_Q \}_{Q\in \Q} $, $ 0< r<\infty $ and $\lambda > 0 $, define the sequence $ u ^\ast_{ r, \lambda}  = \{ (u ^\ast_{ r, \lambda}  )_Q \}_{ Q \in \Q} $ by 
	\begin{equation*}
		(u ^\ast_{ r, \lambda}  )_Q :=  \left( \sum_{P\in \Q, |P|=|Q| }  \frac{|u_P |^r   }{ (1+ |Q|^{-1} \rho_A (x_Q -x_P )  ) ^\lambda }   \right)^{1/r} .
	\end{equation*}
\end{definition}
Clearly, we always have $ |u_Q | \le  (u ^\ast_{ r, \lambda}  )_Q $  for any $Q \in \Q$.

\begin{lemma}[Lemma 6.2, \cite{BH06}] \label{lemma seq HL pointwise}
	Suppose $0 <a \le r <\infty$, $\lambda > r/a$, and $i,j \in \mathbb Z$. Then for any sequence $s = \{ s_P \}_{P \in \Q}$ and for each cube $Q \in \Q$ with $|Q| =  |\det A|^{-j}$ and each $x \in Q$, we have 
	\begin{equation*}
		\left( \sum_{P  \in \Q_i}  \frac{|s_P|^r}{ \left ( 1 +  \frac{ \rho_A (x_Q - x_P)}{  |P| \vee |Q|  }  \right ) ^\lambda }   \right) ^{1/r} \le C |\det A|^{ (i-j)_+  /a }   \left(  \M_{\rho_A}   \left( \sum_{ P \in \Q_i}  |s_P|^a \chi_P  \right) (x)    \right)^{1/a}
	\end{equation*}
	where $C$ depends only on $ \lambda - r/a $.	
	Moreover, if $i=j$, then 
	\begin{equation*}
		\sum_{Q \in \Q_j}   (s ^\ast_{r, \lambda } )_Q \chi_Q   \le C    \left(  \M_{\rho_A}   \left( \sum_{ Q \in \Q_i}  |s_Q| \chi_Q  \right) ^a     \right)^{1/a} .
	\end{equation*}	
\end{lemma}

\section{Matrix weights associated with an expansive matrix} \label{sec matrix weight}
For any $A \in M_m (\mathbb C)$, let 
\begin{equation*}
	\| A\|:= \sup_{\vec z \in \mathbb C^m , |\vec z|= 1} |A \vec z|
\end{equation*}
where $ |\vec z| = ( \sum_{i=1}^m ||z_i|^2 )^{1/2}  $.
The matrix $ A \in M_m (\mathbb C)$ is called a Hermitian matrix if $A^\ast  =A$  and called a unitary matrix if $A A^\ast = I_m$, where $I_m$  is the identity matrix in $ M_m (\mathbb C)$. The diagonal matrix $\operatorname{diag} ( \lambda_1, \ldots, \lambda_m )  \in  M_m (\mathbb C)$ is defined by 
\begin{equation*}
	\operatorname{diag} ( \lambda_1, \ldots, \lambda_m )  : = \begin{pmatrix}
		\lambda_1 & 0 & \dots & 0 & 0\\
		0	 & \lambda_2 & \dots & 0 & 0\\
		\vdots	 & 	\vdots & \ddots & 	\vdots& 	\vdots\\
		0	 & 	0 & \dots & 		\lambda_{m-1} & 	0\\
	0	 & 	0 & \dots & 0& 		\lambda_m
	\end{pmatrix} .
\end{equation*}

A matrix $ A \in  M_m (\mathbb C) $ is said to be positive definite if,  for any $\vec z \in \mathbb C^m \backslash \{ \vec 0\}$, $\vec z ^\ast A\vec z >0 $, and $A$  is called nonnegative definite if, $\vec z ^\ast A\vec z \ge 0 $ for any $\vec z \in \mathbb C^m$. From \cite[Theorem 4.1.4]{HJ13}, the  nonnegative definite matrix is always Hermitian.

Let $B \in M_m (\mathbb C)$ be a positive definite matrix and have eigenvalues $ \{ \lambda_i \}_{i=1}^m$. From \cite[Theorem 2.5.6 (c)]{HJ13}, there exists a unitary matrix $U \in M_m (\mathbb C)$ such that 
\begin{equation*}
	B = U \operatorname{diag} (\lambda_1, \ldots, \lambda_m ) U^\ast .
\end{equation*} 
By \cite[Theorem 4.1.8]{HJ13}, for each $i\in \{ 1, \ldots, m\}$, the eigenvalue $ 0<\lambda_i <\infty $. Hence, for any $\alpha\in \mathbb R$, we can define
\begin{equation*}
	B^\alpha :=  U \operatorname{diag} (\lambda_1 ^\alpha, \ldots, \lambda_m  ^\alpha ) U^\ast .
\end{equation*}

Then we recall the definition of matrix weights.

\begin{definition} \label{def matrix weight}
	A matrix-valued function $W :\rn \to M_m (\mathbb C)$ is called a matrix weight if the following holds:
	
	(i) for any $x\in \rn$, $W(x)$  is nonnegative definite;
	
	(ii) for a.e. $x\in \rn$, $W(x)$ is invertible;
	
	(iii) the entries of $W$ are all locally integrable.
\end{definition}

\begin{lemma}[Lemma 2.3, \cite{BHYY25}] \label{lemma matrix norm AB=BA}
 Let $ A, B \in M_m (\mathbb C) $ be two nonnegative definite matrices. Then 
$\|AB\| =\|BA\|$.
\end{lemma}

\subsection{Basic properties of matrix $\A_p$ weights}

\begin{definition}
	Let $ p \in (0,\infty)$. A matrix weight $W$ is called an $\A_p (\rn , \mathbb C ^m, A)$ matrix weight  associated with a expansive matrix $A$ if $W$ satisfies that, when $p \in (0,1]$,
	\begin{equation*}
		[W]_{\A_p (\rn , \mathbb C ^m, A) }:= \sup_{Q \in \Q} \esssup_{y\in Q} \fint_Q  \| W^{1/p} (x)  W^{-1/p} (y) \|^p \d x <\infty,
	\end{equation*}
or that, when $p \in(1,\infty)$,
	\begin{equation*}
	[W]_{\A_p (\rn , \mathbb C ^m, A) }:= \sup_{Q \in \Q} \fint_Q  \left(   \fint_Q  \| W^{1/p} (x)  W^{-1/p} (y) \|^{p'} \d y  \right) ^{p/ p'} \d x <\infty,
\end{equation*}
where $p' = p/ (p-1)$.
\end{definition}

\begin{remark}
	If we replace dilated cubes $\{Q\}_{Q\in \Q}$ by  $\rho_A$-balls $B$, then we obtain the equivalent class.
	For instance, for $ p \in (0,1]$, let
	\begin{equation*}
		[W]_{\A_p (\rn , \mathbb C ^m, A, \B) }   := \sup_{B \in \B} \esssup_{y\in B} \fint_B  \| W^{1/p} (x)  W^{-1/p} (y) \|^p \d x <\infty.
	\end{equation*}
By  Lemma \ref{lemma basic cube ball cover} (i), for any  $Q\in \Q$, 
\begin{equation*}
	\fint_Q  \| W^{1/p} (x)  W^{-1/p} (y) \|^p \d x \lesssim 	\fint_B \| W^{1/p} (x)  W^{-1/p} (y) \|^p \d x.
\end{equation*}
From this, we obtain $ [W]_{\A_p (\rn , \mathbb C ^m, A) } \lesssim [W]_{\A_p (\rn , \mathbb C ^m, A, \B) } $.
By  Lemma \ref{lemma basic cube ball cover} (ii), for any $\rho_A $ ball $B$, there exists at most $C_\Q$ dilated cube $Q$ with $\scale (Q) = \scale (B) $ such that $B \subset \bigcup_{i=1}^{C_\Q} Q_i$. Then 
\begin{equation*}
	\fint_B  \| W^{1/p} (x)  W^{-1/p} (y) \|^p \d x \le \sum_{i=1}^{C_\Q} \frac{1}{|B|} \int_{Q_i} \| W^{1/p} (x)  W^{-1/p} (y) \|^p \d x \lesssim  \sum_{i=1}^{C_\Q}  \fint_{Q_i} \| W^{1/p} (x)  W^{-1/p} (y) \|^p \d x .
\end{equation*}
From this, we obtain $[W]_{\A_p (\rn , \mathbb C ^m, A, \B) }  \lesssim  [W]_{\A_p (\rn , \mathbb C ^m, A) } $. The proof of  case $p \in (1,\infty)$ is similar.
\end{remark}

In what follows, if there exists no confusion, we denote $\A_p (\rn , \mathbb C ^m, A)$ simply by $\A_p$.

If $m=1$, matrix $\A_p (\rn , \mathbb C ^m, A)$ weights become the usual scalar $A_{p\vee 1}$ weight. 
We refer to \cite[Chapter 7]{Gra14} for the theory of scalar weights.

\begin{definition}
	Let  $p\in (0,\infty)$,  $W$ be a matrix weight, and  $E \subset \rn$  a bounded
	measurable set satisfying  $ |E | \in (0,\infty) $. The matrix  $A_E \in M_m (\mathbb C) $ is called a reducing 	operator of order $p$ for $W$ if $A_E$ is positive definite and, for any $\vec z \in \mathbb C^m$,
	\begin{equation} \label{eq reducing matrix}
		|A_E \vec z| \approx \left( \fint_E  |W^{1/p}  (x) \vec z|^p \d x \right)^{1/p},
	\end{equation}
where the positive equivalence constants depend only on $m$ and $p.$
\end{definition}
The existence of $A_E$ is guaranteed by \cite[Proposition 1.2]{Go03} and \cite[p. 1237]{FR04}.

\begin{lemma}[Lemma 2.10, \cite{BHYY25}] \label{lemma reduce replaced z by M}
	Let $p\in (0,\infty)$, $W$ be a matrix weight, and $E \subset \rn$ a bounded measurable set with $|E| \in (0,\infty)$. If $A_E$ is a reducing operator of order $p$ for $W$, then, for any matrix $M \in M_m (\mathbb C)$, 
	\begin{equation*}
		\| A_E M \| \approx \left( \fint_E  \|W^{1/p}  (x) M \|^p \d x \right)^{1/p}
	\end{equation*}
where the positive equivalence constants depend only on $m$ and $p.$
\end{lemma}

Using Lemma \ref{lemma reduce replaced z by M}, we obtain an equivalent characterization of $\A_p$-matrix weights for $p \in (0,1]$. The proof is similar to \cite[Proposition 2.11]{BHYY25}.
\begin{corollary}
	Let $p \in (0,1]$ and $A$ be a dilation. Then there exists a constant $C>0$, depending only on $m$  and $p$, such that for any matrix weight $W$,
	\begin{equation*}
		[W]_{\A_p } \le [W]_{\A_p}^\ast \le C [W]_{\A_p },
	\end{equation*}
where
\begin{equation*}
	 [W]_{\A_p}^\ast : = \sup_{Q \in \Q}  \fint_Q   \esssup_{y\in Q} \| W^{1/p} (x)  W^{-1/p} (y) \|^p \d x .
\end{equation*}
\end{corollary}
Similar to \cite[Proposition 2.12]{BHYY25}, \cite[Lemma 2.14, Corolary 2.16]{BHYY25}  and \cite[Lemma 5.4]{FR04} respectively, we have the following three lemmas.

\begin{lemma}
	Let $0 <p <q<\infty$. Then $\A_p \subset \A_q$. Moreover, there exists a constant $C>0$, depending only on $m,p,q$, such that for any matrix weight $W$,  $[W]_{\A_q}  \le C [W]_{\A_p}$.
\end{lemma}

\begin{lemma} \label{lemma p> 1 A_Q -1  approx W -1/p}
	Let $p \in(1,\infty) $, $p' = p / (p-1)$ and $W \in \A_p$. Then $\widetilde W : = W ^{-1/ (p-1)}  \in \A_{p'}$.
	Furthermore, if $\{A_Q\}_{Q\in \Q} $ and $\{ \tilde A_Q\}_{Q\in \Q} $ denote the  reducing operators, respectively, of order $p$ for $W$ and of order $p'$ for $\widetilde W$, then 
	\begin{equation*}
		[W]^{1/p}_{\A_p} \approx [\widetilde W]^{1/p'}_{\A_{p'}}  \approx \sup_{Q\in \Q} \| A_Q \tilde A_{Q} \|,
	\end{equation*}
where the positive equivalence constants depend only on $m$ and $p$. Moreover, for any $\vec z \in \mathbb C^m$ and any matrix $M \in M_m (\mathbb C)$, 
\begin{equation*}
	|A_Q ^{-1} \vec z| \approx |\tilde A_Q  \vec z| \approx \left( \fint_Q |W^{-1/p} (x) \vec z |^{p'} \d x \right)^{1/p'}  ,
\end{equation*}
and 
\begin{equation*}
		\|A_Q ^{-1} M \| \approx \|\tilde A_Q  M\| \approx \left( \fint_Q \|W^{-1/p} (x) M \|^{p'} \d x \right)^{1/p'},
\end{equation*}
where the positive equivalence constants depend only on $m$, $p$ and $[W]_{\A_p}$.
\end{lemma}

\begin{lemma} \label{lemma p le 1 A_Q ^-1 z W^ -1/p}
	Let $p \in (0,1]$, $ W \in \A_p$, $\{A_Q \}_{Q\in \Q}$ be a reducing operator of order $p$ for $W$. Then, for any $\vec z \in \mathbb C^m$, any $M\in M_m (\mathbb C)$,
	\begin{equation*}
		|A_Q ^{-1} \vec z | \approx \esssup_{x\in Q}  |W^{-1/p} (x) \vec z | 
	\end{equation*} 
and 
\begin{equation*}
	\|A_Q ^{-1} M \| \approx \esssup_{x\in Q}  \|W^{-1/p} (x) M \| 
\end{equation*}
where the positive equivalence constants depends only on $m$, $p$ and $[W]_{\A_p}$.
\end{lemma}

\subsection{The $\A_p$-dimension of matrix weights associated with a dilation}
\begin{definition} \label{def A_p dimension}
	Let $ p\in(0,\infty)$,  $d \in \mathbb R$  and $W $ be a matrix weight. Then $W $ is said
	to have the $\A_p$-dimension $d$ associated with a dilation $A$, denoted by $ W \in \mathbb D_{p,d,A}$, if there exists a positive
	constant $C $ such that, for any balls $B \in \B $ and  any  $ i \in \mathbb Z_+$, when  $p \in (0,1]$,
	\begin{equation*}
		\esssup_{ y \in  |\det A|^i B } \fint_{B }  \|  W^{1/p} (x)  W^{-1/p} (y) \| ^p \d x \le C |\det A|^{id},
	\end{equation*}
or that,  when $p \in (1,\infty)$,
\begin{equation*}
 \fint_B   \left(  \fint_{|\det A|^i B  }   \|  W^{1/p} (x)  W^{-1/p} (y) \| ^{p'} \d y \right) ^{p/p'}  \d x \le C |\det A|^{id},
\end{equation*}
where $p' = p/(p-1)$.
\end{definition}
Next we have the following basic properties of $\A_p$-dimension.
\begin{proposition}\label{prop baisic dimension}
	Let $p \in (0,\infty)$. Then the following statements hold. 
	
	{\rm (i)}
 For any $d \in (-\infty,0)$,  $\mathbb D_{p,d,A} = \emptyset$.
	
	{\rm (ii)} For any $d \in [0,1)$,  $\mathbb D_{p,d,A} \subset \A_p$.
	
		{\rm (iii)} For any $d \in [1,\infty)$,  $\mathbb D_{p,d,A} = \A_p$.
	
	{\rm (iv)} For any $0 \le d_1 < d_2 < \infty $,  $\mathbb D_{p,d_1,A} \subset \mathbb D_{p,d_2,A}$.
	
	{\rm (v)} For any $ p< q<\infty$ and $d \in [0,\infty) $, $\mathbb D_{p,d,A} \subset \mathbb D_{q,d,A}$ .
\end{proposition}

\begin{proof}
	By the definition of $\A_p$-dimensions, we  obtain  (ii), (iii) and  (iv). 
Applying an argument similar to that used in the proof of  \cite[Proposition 2.12]{BHYY25}, we  obtain (v).

Now we prove (i). 
Case $p \in (0,1]$. 
\begin{align*}
	0  & < \esssup _{y \in B_{\rho_A}(0,1)}  \fint_{B_{\rho_A} (0,1)}  \| W^{1/p} (x) W^{-1/p} (y) \|^p \d x \\
	& \le \sup_{B \in \B} \esssup_{ y \in  |\det A|^i B } \fint_{B }  \|  W^{1/p} (x)  W^{-1/p} (y) \| ^p \d x \le C |\det A|^{id},
\end{align*}
which contradicts $d \in (-\infty,0 )$, and hence $\mathbb D_{p,d,A} = \emptyset$.

Case $p \in (1,\infty)$. If there exists $W\in \mathbb D_{p,d,A}$, then, by Definition  \ref{def matrix weight} (iii), we have that $ \|W^{1/p }  A_{B_{\rho_A}(0,1)} ^{-1} \| \in L^1_{\operatorname{loc}} $, where $A_{B_{\rho_A}(0,1)}$ is the reducing operator of order $p$ for $W$. By the Lebesgue differentiation theorem,
Lemma \ref{lemma p> 1 A_Q -1  approx W -1/p}, for a.e. $x_0 \in \rn$ with $\rho_A (x_0) <1$,
\begin{align*}
	 \|W^{1/p } (x_0)  A_{B_{\rho_A}(0,1)} ^{-1} \| ^p   &= \lim_{i \to \infty}  \fint_{ B_{\rho_A} (x_0 , |\det A|^{-i})}   \|W^{1/p }(x)  A_{B_{\rho_A}(0,1)} ^{-1} \| ^p \d x \\
	 & \approx  \lim_{i \to \infty}  \fint_{ B_{\rho_A} (x_0 ,|\det A|^{-i} )}  \left( \fint_{B_{\rho_A }(0,1) }  \|W^{1/p }(x) W^{-1/p} (y) \|^{p'} \d y     \right)^{p/p'}   \d x \\
	 & \lesssim  \lim_{i \to \infty}  \fint_{ B_{\rho_A} (x_0 , |\det A|^{-i} )}  \left( \fint_{B_{\rho_A}(x_0, C ) }  \|W^{1/p }(x) W^{-1/p} (y) \|^{p'} \d y     \right)^{p/p'}   \d x \\
	 &\lesssim \lim_{i \to \infty} | \det A |^{id} =0
\end{align*}
and hence all entries of $W(x_0)$ are $0$, which contradicts Definition  \ref{def matrix weight} (ii). Thus  $\mathbb D_{p,d,A} = \emptyset$.
This finishes the proof of (i) and hence Proposition \ref{prop baisic dimension}.
\end{proof}

\begin{proposition}
	Let $ p\in(0,\infty)$, $W \in \A_p$,  $\{ A_B\}_{B \in \B} $ be a family of reducing
	operators of order $p$ for $W$, and $d \in [0,\infty) $. Then $W \in \mathbb D_{p,d, A}$  if and
	only if there exists a positive constant $C $ such that, for any $ B \in \B$ and any $i\in \mathbb Z_+$,
	$  \| A_B A_{ |\det A|^iB } ^{-1} \|^p  \le C |\det A|^{id}$.
\end{proposition}

\begin{proof}
	The proof is similar to \cite[Proposition 2.25]{BHYY25} and we omit it here.
\end{proof}

\begin{lemma} \label{lemma W 1/p e is scalar weight}
	Let $ p\in(0,\infty)$, $W \in \A_p$. Then for any $\vec e \in \mathbb C^m$,  $|W^{1/p} (x) \vec  e| ^p $ is the scalar $A_{ p \vee 1  } $ weight.
\end{lemma}
\begin{proof}
	The proof of case $ p \in (1,\infty)$ is  similar to \cite[Corollary 2.2]{Go03} and the proof of case $p \in (0,1] $  is  similar to  \cite[Lemma 2.1]{FR04}. We omit it here.
\end{proof}

Now we have the following reverse H\"older inequality.

\begin{lemma} \label{lemma reverse holder matrix}
	Let $ p\in(0,\infty)$, $W \in \A_p$.  Then there exist $r (W) \in (1,\infty) $  and a
	positive constant $C$ such that, for any  $r \in [1,r (W)]$, any $B \in \B$, and any	nonnegative definite matrix  $M$,
\begin{equation*}
	\left(  \fint_B \|W^{1/p} (x) M\|^{pr} \d x \right) ^{1/(rp)} \le C \left(  \fint_B \|W^{1/p} (x) M\|^{p} \d x\right)^{1/p} .
\end{equation*}
\end{lemma}
\begin{proof}
	Let $\{\vec e_i\}_{ i=1}^m  $ be an orthonormal basis on $\mathbb C^m$.  Then for any matrix $B \in \M_m (\mathbb C)$, $ \|B\| \approx \sum_{i=1}^m |B \vec e_i | $.  By Lemma \ref{lemma W 1/p e is scalar weight},  for all $\vec y \in \mathbb C^m$, the scalar weight $|W^{1/p} (x) \vec y |^p $ are uniformly in  $A_{p \vee 1}$. 
	Hence, the following reverse H\"older condition holds: there exists $\gamma >0 $ such that 
	\begin{equation*}
		\left( \fint_Q |W^{1/p} (x)  \vec y |^{p (1+\gamma)} \d x  \right)^{1/ (1+\gamma)} \lesssim  \fint_Q |W^{1/p} (x)  \vec y |^{p } \d x,
	\end{equation*}
	where the the implicit positive constant is independent of $\vec y $ and $Q \in \Q$. Let $r = r(W) = 1+\gamma$ and 
	 we have
\begin{align*}
		\left(  \fint_B \|W^{1/p} (x) M\|^{pr} \d x \right) ^{1/(rp)}  & \approx \sum_{i=1}^m 	\left(  \fint_B |W^{1/p} (x) M \vec e_i|^{pr} \d x \right) ^{1/(rp)} \\
		& \lesssim \sum_{i=1}^m 	\left(  \fint_B |W^{1/p} (x) M \vec e_i|^{p} \d x \right) ^{1/p}\\
		& \approx 	\left(  \fint_B \|W^{1/p} (x) M \|^{p} \d x \right) ^{1/p}.
\end{align*}
Thus the proof is complete.
\end{proof}

\begin{proposition}
	Let $A$ be a dilation.
		Let $ p\in(0,\infty)$, $W \in \A_p$. Then there exists  $d_W \in [0,1 )$ such that $W \in \mathbb D_{p,d_W, A}$.
\end{proposition}
\begin{proof}
	Let $r = r(W) >1$ is the same as in Lemma \ref{lemma reverse holder matrix}. Next we consider two cases on $p$.
	
	Case $p \in (0,1]$.
	By H\"older's inequality and Lemma \ref{lemma reverse holder matrix}, we have
	\begin{align*}
		\esssup_{y \in |\det A|^i B} \fint_B \| W^{1/p} (x) W^{-1/p} (y) \| ^p \d x  & \lesssim  \esssup_{y \in |\det A|^i B}   \left( \fint_B \| W^{1/p} (x) W^{-1/p} (y) \| ^{pr} \d x \right)^{1/r}  \\
		& \lesssim   \esssup_{y \in |\det A|^i B}     \left( |\det A|^i  \fint_{ |\det A|^i B } \| W^{1/p} (x) W^{-1/p} (y) \| ^{pr} \d x \right)^{1/r}  \\
		& \lesssim |\det A|^{i/r} \esssup_{y \in |\det A|^i B}     \fint_{ |\det A|^i B } \| W^{1/p} (x) W^{-1/p} (y) \| ^{p} \d x  \\
		&\le |\det A|^{i/r}  [W]_{\A_p} .
	\end{align*}
Hence $W$ has the $\A_p$-dimension $1/r \in [0,1)$ associated with the dilation $A$.

Case $ p\in (1,\infty).$ For any $\rho_A$-balls $B$, let $A_B $  be a reducing operator of order $p$ for $W$.
From Lemmas  \ref{lemma matrix norm AB=BA}, \ref{lemma p> 1 A_Q -1  approx W -1/p}, \ref{lemma reverse holder matrix}, \ref{lemma reduce replaced z by M},  and H\"older's inequality, we obtain

\begin{align*}
	\fint_B   \left(  \fint_{|\det A|^i B}   \|  W^{1/p} (x)  W^{-1/p} (y) \| ^{p'} \d y \right) ^{p/p'}  \d x &
	=  \fint_B    \left(  \fint_{|\det A|^i B}   \|    W^{-1/p} (y) W^{1/p} (x) \| ^{p'} \d y \right) ^{p/p'}  \d x \\
	&\approx  \fint_B       \|   A_{ |\det A|^i B}^{-1}  W^{1/p} (x) \|  ^{p}  \d x  \\
	&=\fint_B       \|   W^{1/p} (x)   A_{ |\det A|^i B}^{-1}  \|  ^{p}  \d x \\
	&\le \left(  \fint_B       \|   W^{1/p} (x)   A_{ |\det A|^i B}^{-1}  \|  ^{p r}  \d x\right)^{1/r}\\
	& \lesssim |\det A|^{i/r}  \left(  \fint_{|\det A|^i B}       \|   W^{1/p} (x)   A_{ |\det A|^i B}^{-1}  \|  ^{p r}  \d x\right)^{1/r}\\
	& \lesssim |\det A|^{i/r}  \fint_{|\det A|^i B}       \|   W^{1/p} (x)   A_{ |\det A|^i B}^{-1}  \|  ^{p }  \d x\\
	& \approx |\det A|^{i/r} .
\end{align*}
Hence $W$ has the $\A_p$-dimension $1/r \in [0,1)$ associated with the dilation $A$. Thus we finish the proof.
\end{proof}
\begin{proposition}
	Let $A$ be a dilation.
	Let $ p\in(0,\infty)$, $W \in \A_p$.
	
	{\rm (i)} If $p\in (0,1] $, then for any $i \in \mathbb Z_+$,
		\begin{align*}
		\esssup_{y \in  B} \fint_{|\det A|^i B} \| W^{1/p} (x) W^{-1/p} (y) \| ^p \d x \lesssim 1.
	\end{align*}

{\rm (ii)} If  $ p\in(1,\infty)$ and $d\in \mathbb R$, then for any $i \in \mathbb Z_+$,
\begin{equation} \label{Bi Q le det B id}
 	\fint_{ |\det A|^i B}    \left(  \fint_{B }   \|  W^{1/p} (x)  W^{-1/p} (y) \| ^{p'} \d y \right) ^{p/p'}  \d x  \lesssim |\det A|^{id}
\end{equation}
if and only if the dual wight $ \widetilde W= W ^{-1/(p-1)}  \in \A_{p'} $ has the   $\A_{p'}$-dimension $d / (p-1)$ associated with the dilation $A$.
\end{proposition}
\begin{proof}
	(i)
	Note that 
	\begin{align*}
		\esssup_{y \in  B} \fint_{|\det A|^i B} \| W^{1/p} (x) W^{-1/p} (y) \| ^p \d x \le \esssup_{y \in |\det A|^i B} \fint_{|\det A|^i B} \| W^{1/p} (x) W^{-1/p} (y) \| ^p \d x \le  [W]_{\A_p} .
	\end{align*}
(ii)
  Let  $\{ A_B\}_{B\in\B} $ be the reducing operator of order $p$ for $W$  and $\{ \widetilde A_B\}_{B\in\B} $ be the reducing operator of order $p'$ for $\widetilde W= W ^{-1/(p-1)}$. Then for any $\rho_A$-balls $R,B$, we have 
\begin{align*}
	\left(  	\fint_R    \left(  \fint_{B }   \|  W^{1/p} (x)  W^{-1/p} (y) \| ^{p'} \d y \right) ^{p/p'}  \d x \right)^{1/p}  \approx \left(  	\fint_Q    \left(  \fint_{R}   \|  \widetilde W^{1/p'} (y)  \widetilde{W}^{-1/p'} (x) \| ^{p} \d x \right) ^{p'/p}  \d y \right)^{1/p'} .
\end{align*}

For $R = |\det A|^i B$, equation (\ref{Bi Q le det B id}) is equivalent to the boundedness of the left-hand side above by $C |\det A |^{ id /p }  $. On the other hand, the condition that $\widetilde{W}$ has the   $\A_{p'}$-dimension $\tilde d$ associated with the dilation $A$ is equivalent to the boundedness of the right-hand side above by  $C  |\det A |^{ i  \tilde d /p' } $. Since both sides are comparable to each other, it follows that (\ref{Bi Q le det B id}) holds with dimension $ d$ if
and only if $ \widetilde{W} $ has the   $\A_{p'}$-dimension $\tilde d$ associated with the dilation $A$ such that $ d/p = \tilde d / p' $.
\end{proof}

\begin{lemma} \label{lemma A B_0  AB_1 estimate}
	Let $p\in (0,\infty)$.
	Let $W \in \A_p$ have the $\A_{p}$-dimension $ d \in [0,1)$, and let $\{A_B\}_{B\in \B}$ be a family of reducing operators of order $p$ for $W$.
	If $p\in (1,\infty)$, let further $\widetilde W : =W^{-1/(p-1)}$ (which belongs to $\A_{p'}$) have the $\A_{p'}$-dimension $ \tilde d \in [0,1)$, while if $p\in (0,1] $, let $ \tilde d =0 $. 
	Let $\Delta : =d/p + \tilde d/p' $. Then there exists a  constant $C>0$ such that, for any $\rho_A$-balls   $B_0, B_1 \in \B $,
	\begin{align*}
		\| A_{B_0} A_{B_1}  ^{-1} \|  \le C \max\left\{ \left( \frac{ r_{B_1}  }{   r_{B_0}   }\right) ^{d/p}  , \left( \frac{ r_{B_0} }{   r_{B_1} }\right) ^{ \tilde d/p'} \right\}  \left( 1 +  \frac{ \rho_A ( c_{B_0} -  c_{B_1}  ) }{ r_{B_0}  \vee r_{B_1} } \right)^{ \Delta } .
	\end{align*}
\end{lemma}
\begin{proof}
	First consider $B_0 \cap B_1 \neq \emptyset$. Then $B_1 \subset \lambda B_0 $ where $ \lambda \approx  \max \{   r_{B_1} / r_{B_0} , 1 \} $.
	
	If $ p \in (0,1]$, for  a.e. $y \in B_1$, we have
		\begin{equation} \label{eq AQ AR p le 1}
		\| A_{B_0} A_{B_1}  ^{-1} \|^p  \le 	\| A_{B_0} W^{-1/p} (y)  \|^p  \|  W^{1/p} (y)A_{B_1}  ^{-1} \|^p.
	\end{equation}
	By Lemma \ref{lemma reduce replaced z by M} and Definition \ref{def A_p dimension}, we obtain
	\begin{equation*}
		\|A_{B_0}  W^{-1/p} (y)  \|^p \approx \fint_{B_0} \| W^{1/p} (x)  W^{-1/p} (y) \|^p \d x \le C \lambda^d  
	\end{equation*}
since $y \in B_1 \subset \lambda B_0$.
	Taking an integral average of (\ref{eq AQ AR p le 1}) over $y \in B_1$ and using Lemma \ref{lemma reduce replaced z by M}, we get
	\begin{equation*}
			\| A_{B_0} A_{B_1}  ^{-1} \|^p \lesssim \lambda^d   \fint_{B_1} \|W^{1/p} (y) A_{B_1} ^{-1} \|^p \d y  \approx \lambda^d  .
	\end{equation*}
	
	If $ p \in (1,\infty)$, using H\"older's inequality, we have
	\begin{align*}
		\| A_{B_0} A_{B_1}  ^{-1} \| & \le \fint_{B_0} 	\| A_{B_0}  W^{-1/p} (y)  \|  \|  W^{1/p} (y) A_{B_1}  ^{-1}  \| \d y \\
		& \le \left(   \fint_{B_0} 	\| A_{B_0} W^{-1/p} (y)  \|^{p'}  \d y  \right)^{1/p'} \left(   \fint_{B_0} 	\|  W^{1/p} (y) A_{B_1}^{-1} \| ^{p}  \d y  \right)^{1/p} .
	\end{align*}
From Lemmas \ref{lemma matrix norm AB=BA} and \ref{lemma p> 1 A_Q -1  approx W -1/p}, we have	
	\begin{equation*}
		\left(   \fint_{B_0}	\| A_{B_0}  W^{-1/p} (y)  \|^{p'}  \d y  \right)^{1/p'} \lesssim 1.
	\end{equation*}
	Using Lemmas \ref{lemma matrix norm AB=BA} and \ref{lemma p> 1 A_Q -1  approx W -1/p}, we obtain 
	\begin{align} \label{eq A_B 1 A B2 -1 approx}
	I_2:=	\left(   \fint_{B_0} 	\|  W^{1/p} (y) A_{B_1}^{-1} \| ^{p}  \d y  \right)^{1/p} &\approx \|A_{B_0} A_{B_1}^{-1} \| \approx \left(  \fint_{B_0}  \left(   \fint_{B_1}	\|  W^{1/p} (y) W^{-1/p} (x) \| ^{p'}  \d x  \right)^{p/p'}  \d y  \right)^{1/p}   .
	\end{align}
If $r_{B_1} \ge r_	{B_0} $, then $ B_1 \subset  a B_0$ and $|B_1| \approx |aB_0|$ where $a\approx   r_{B_1} / r_{B_0}$. Hence 
\begin{align*}
	I_2  & \lesssim  \left(  \fint_{B_0}  \left(   \fint_{ a B_0 }	\|  W^{1/p} (y) W^{-1/p} (x) \| ^{p'}  \d x  \right)^{p/p'}  \d y  \right)^{1/p}  \\
	&\lesssim  a^{d/p } \approx \left( \frac{ r_{B_1} }{ r_{B_0}}\right) ^{d/p} .
\end{align*}
If $r_{B_1} \le r_	{B_0} $, then $ B_0 \subset  b B_1$   and $|B_0 | \approx |b B_1|  $ where $b \approx   r_{B_0} / r_{B_1}$. Hence 
\begin{align*}
	I_2  & \lesssim  \left(  \fint_{b B_1}  \left(   \fint_{ B_0 }	\|  W^{1/p} (y) W^{-1/p} (x) \| ^{p'}  \d x  \right)^{p/p'}  \d y  \right)^{1/p}  \\
	&\lesssim  b^{ \tilde d/p' }= \left( \frac{ r_{B_0} }{ r_{B_1}}\right) ^{\tilde d/p'} .
\end{align*}

	Thus, for $p\in (0,\infty)$, $B_0 \cap  B_1 \neq \emptyset$,  we have
	\begin{equation*}
		\| A_{B_0} A_{B_1}  ^{-1} \| \lesssim \begin{cases}
			\left( \frac{ r_{B_1} }{ r_{B_0}}\right) ^{d/p} & r_{B_1} \ge r_	{B_0} ,\\
			\left( \frac{ r_{B_0} }{ r_{B_1}}\right) ^{\tilde d/p'}  & r_{B_1} \le r_	{B_0}   .
		\end{cases}
	\end{equation*}

Case $B_0 \cap  B_1 =\emptyset$.  Chose the smallest $\rho_A$-ball $B_2$ such that $B_0 \subset B_2 $ and $B_1 \subset B_2 $. Then $r _{B_2} \approx  r_{B_1} +r_{B_0} + \rho_A ( c_{B_0} -  c_{B_1}  )$ from some geometrical observations. From this, 
\begin{align*}
	\| A_{B_0} A_{B_1}  ^{-1} \| & \le \| A_{B_0} A_{B_2}  ^{-1} \| 	\| A_{B_2} A_{B_1}  ^{-1} \| \\
	&\lesssim \left( \frac{ r_{B_2} }{ r_{B_0}}\right) ^{d/p} \left( \frac{ r_{B_2} }{ r_{B_1}}\right) ^{\tilde d/p'} \\
	& =   \left( \frac{ r_{B_0}  \vee r_{B_1}  }{   r_{B_0}   }\right) ^{d/p}  \left( \frac{ r_{B_0}  \vee r_{B_1}  }{   r_{B_1} }\right) ^{ \tilde d/p'} \left( \frac{ r_{B_2} }{   r_{B_0}  \vee r_{B_1} }\right) ^{d/p + \tilde d/p'} \\
	& \approx \max\left\{ \left( \frac{ r_{B_1}  }{   r_{B_0}   }\right) ^{d/p}  , \left( \frac{ r_{B_0} }{   r_{B_1} }\right) ^{\tilde d/p'} \right\}  \left( 1 +  \frac{ \rho_A ( c_{B_0} -  c_{B_1}  ) }{ r_{B_0}  \vee r_{B_1} } \right)^{ d/p + \tilde d/p' } .
\end{align*}
Thus the proof is completed.
\end{proof}

\begin{lemma} \label{lemma cube AQ AR-1}
		Let $p\in (0,\infty)$.
	Let $W \in \A_p$ have the $\A_{p}$-dimension $ d \in [0,1)$, and let $\{A_B\}_{B\in \B}$ be a family of reducing operators of order $p$ for $W$.
	If $p\in (1,\infty)$, let further $\widetilde W : =W^{-1/(p-1)}$ (which belongs to $\A_{p'}$) have the $\A_{p'}$-dimension $ \tilde d \in [0,1)$, while if $p\in (0,1] $, let $ \tilde d =0 $. 
		Let $\Delta : =d/p + \tilde d/p' $. Then there exists a  constant $C>0$ such that, for any dilated cubes $Q, R \in \Q$,
	\begin{align*}
		\| A_{Q} A_{R}  ^{-1} \|  \lesssim \max\left\{ \left( \frac{ |R|  }{   |Q|   }\right) ^{d/p}  , \left( \frac{ |Q| }{   |R| }\right) ^{ \tilde  d/p'} \right\}  \left( 1 +  \frac{ \rho_A ( c_{Q} -  c_{R}  ) }{ |Q|  \vee |R| } \right)^{ \Delta } .
	\end{align*}
\end{lemma}

\begin{proof}
	For $Q,R\in \Q$,  by Lemma \ref{lemma basic cube ball cover}, let $B_0, B_1\in \B $ such that 
	$Q \subset B_0$, $R\subset B_1$,   $|Q| \approx |B_0| \approx r_{B_0}$ and  $|R| \approx |B_1| \approx r_{B_1}$.
	
	Case $p \in (1,\infty)$.
	Then by (\ref{eq A_B 1 A B2 -1 approx}) and  Lemma \ref{lemma A B_0  AB_1 estimate}, we obtain
	\begin{align*}
		\| A_{Q} A_{R}  ^{-1} \|  & \approx \left(  \fint_{Q}  \left(   \fint_{R}	\|  W^{1/p} (y) W^{-1/p} (x) \| ^{p'}  \d x  \right)^{p/p'}  \d y  \right)^{1/p} \\
		& \lesssim \left(  \fint_{B_0}  \left(   \fint_{B_1}	\|  W^{1/p} (y) W^{-1/p} (x) \| ^{p'}  \d x  \right)^{p/p'}  \d y  \right)^{1/p} \\
		& \approx \|A_{B_0} A_{B_1}^{-1} \| \\
		& \lesssim \max\left\{ \left( \frac{ |R|  }{   |Q|   }\right) ^{d/p}  , \left( \frac{ |Q| }{   |R| }\right) ^{ \tilde  d/p'} \right\}  \left( 1 +  \frac{ \rho_A ( c_{Q} -  c_{R}  ) }{ |Q|  \vee |R| } \right)^{ \Delta }  .
	\end{align*}
Case $p \in (0,1]$.
By Lemmas \ref{lemma matrix norm AB=BA}, \ref{lemma p le 1 A_Q ^-1 z W^ -1/p}, \ref{lemma A B_0  AB_1 estimate}, we get
\begin{align*}
\| A_{Q} A_{R}  ^{-1} \| & = \| A_{R}  ^{-1} A_{Q}  \| \approx \esssup_{x\in R} \|W^{-1/p} (x) A_Q\| \\
&\approx \esssup_{x\in R}  \left(  \fint_Q   \|W^{-1/p} (x) W^{1/p} (y)\|^p \d y \right)^{1/p} \\
&\lesssim \esssup_{x\in B_1}  \left(  \fint_{B_0}   \|W^{-1/p} (x) W^{1/p} (y)\|^p \d y \right)^{1/p} \\
&\approx \| A_{B_0} A_{B_1}  ^{-1} \|  \\
& \lesssim \max\left\{ \left( \frac{ |R|  }{   |Q|   }\right) ^{d/p}  , \left( \frac{ |Q| }{   |R| }\right) ^{ \tilde  d/p'} \right\}  \left( 1 +  \frac{ \rho_A ( c_{Q} -  c_{R}  ) }{ |Q|  \vee |R| } \right)^{ \Delta }  .
\end{align*}
Thus we finish the proof.
\end{proof}

\begin{lemma} \label{lemma reverse A_Q W}
	Let $ p\in (0,\infty) $, $W \in \A_p$, and $\{ A_Q\}_{Q \in \Q}$ be a sequence of reducing operators of order $p$ for $W$.
	
	{\rm (i)}
If $ p \in (0,1]$, then 
	\begin{equation*}
		\sup_{Q \in \Q} \esssup_{x\in \Q} \|A_Q W ^{-1/p} (x) \| \approx [W]_{\A_p}^{1/p} .
	\end{equation*}
	
		{\rm (ii)}
  If $ p \in (0,  \infty) $, then there exists  $\delta_W >0$  such that for all $r \in (0, p+\delta_W)$
	\begin{equation*}
		\sup_{Q\in \Q}	\fint_Q \| W^{1/p} (x) A_Q ^{-1} \|^{r} \d x   <\infty.
	\end{equation*}
	
		{\rm (iii)}   If $ p \in (0,1]$, then there exists $\delta_W >0$  such that for all $r \in (0, p+\delta_W)$
	\begin{equation*}
		\sup_{Q \in \Q}  \fint_Q  \sup_{P\in Q ^\sharp, x\in P}	\| W^{1/p} (x) A_P^{-1} \| ^r \d x  <\infty.
	\end{equation*}
	
	{\rm (iv)} If $ p \in (1,\infty)$, then there exists a positive constant $\delta_W$, depending only on $n,m,p$ and $[W]_{\A_p}$, such that 	
	for $r \in (0, p' +\delta_W)$
	\begin{equation*}
		\sup_{Q}  \left(  \fint_Q    \|A_Q W ^{-1/p} (x) \| ^r \d x \right)^{1/r}   \lesssim [W]_{\A_p}^{1/p} .
	\end{equation*}

		{\rm (v)}   If $ p \in (1,\infty)$, then there exists $\delta_W >0$   such that for all $r \in (0, p+\delta_W)$
	\begin{equation*}
		\sup_{Q \in \Q}  \fint_Q  \sup_{P\in Q ^\sharp, x\in P}	\| W^{1/p} (x) A_P^{-1} \| ^r \d x  <\infty .
	\end{equation*}
	
	{\rm (vi)} If $ p \in (1,\infty)$, then there exists a positive constant $\delta_W$, depending only on $n,m,p$ and $[W]_{\A_p}$,    such that 	
	for $r \in (0, p' +\delta_W)$
	\begin{equation*}
		\sup_{Q}  \left(  \fint_Q  \sup_{P\in Q ^\sharp, x\in P}  \|A_P W ^{-1/p} (x) \| ^r \d x \right)^{1/r}   <\infty .
	\end{equation*}	
\end{lemma}

\begin{proof}
	Let $\{ \vec  e_i\}_{i=1}^m  $ are the standard unit Euclidean basis vectors in $\mathbb C^m$. Then for any matrix $B \in \M_m (\mathbb C)$, $ \|B\| \approx \sum_{i=1}^m |B \vec e_i | $.
	
	(i) For a.e. $x\in Q $,  we have
	\begin{align*}
		\|A_Q W ^{-1/p} (x) \| & \approx \sum_{i=1}^m |A_Q W ^{-1/p} (x) \vec e_i|  \\
		&\approx \sum_{i=1}^m   \left( \fint_Q  | W^{1/p} (y) W ^{-1/p} (x)  \vec e_i| ^p \d y   \right) ^{1/p} \\
		& \approx \left( \fint_Q  \| W^{1/p} (y) W ^{-1/p} (x) \| ^p \d y   \right) ^{1/p} .
	\end{align*}
	Taking supremum over $Q\in \Q$, we obtain (i).
	
	(ii)   By Lemma \ref{lemma W 1/p e is scalar weight}, for all $\vec y \in \mathbb C^m$, the scalar weight $|W^{1/p} (x) \vec y |^p $ are uniformly in  $A_{p \vee 1 }$. Hence, the following reverse H\"older condition holds: there exists $\gamma >0 $ such that 
	\begin{equation*}
		\left( \fint_Q |W^{1/p} (x)  \vec y |^{p (1+\gamma)} \d x  \right)^{1/ (1+\gamma)} \lesssim  \fint_Q |W^{1/p} (x)  \vec y |^{p } \d x,
	\end{equation*}
	where the  implicit positive constant is independent of $\vec y $ and $Q \in \Q$. Applying the above inequality with $\vec y = A_Q^{-1} \vec e_i$, 
	\begin{align*}
		\fint_Q \| W^{1/p} (x) A_Q ^{-1} \|^{p (1+\gamma)} \d x &\lesssim 	\fint_Q \left(  \sum_{ i =1 }^m |  W^{1/p} (x) A_Q ^{-1} \vec e_i|  \right)^{p (1+\gamma)} \d x \\
		&\lesssim \sum_{ i =1 }^m	\fint_Q \left(   |  W^{1/p} (x) A_Q ^{-1} \vec e_i|  \right)^{p (1+\gamma)} \d x \\
		&\lesssim \sum_{ i =1 }^m	 \left(  \fint_Q \left(   |  W^{1/p} (x) A_Q ^{-1} \vec e_i|  \right)^{p (1+\gamma)} \d x \right)^{1+\gamma} \\
		&\lesssim \sum_{ i =1 }^m	 \left(  \fint_Q \left(   | A_Q A_Q ^{-1} \vec e_i|  \right)^{p (1+\gamma)} \d x \right)^{1+\gamma} <\infty  .
	\end{align*}
	Letting $\delta_W := p \gamma$, we obtain (ii).
	
	(iii) 
	Let $Q, P\in\Q$ and $ P \preceq Q $.
	By Lemmas \ref{lemma p le 1 A_Q ^-1 z W^ -1/p} and  \ref{lemma matrix norm AB=BA},
	we obtain
	\begin{align*}
		\| A_Q A_P^{-1} \| \approx \sum_{i=1}^m | A_P^{-1} A_Q \vec e_i| \approx \sum_{i=1}^m  \esssup_{x\in P} | W^{-1/p } (x)  A_Q \vec e_i| \approx  \esssup_{x\in P} \| W^{-1/p}   (x) A_Q \| .
	\end{align*}
	By Lemma \ref{lemma use covered Q}, $Q^\sharp \subset \bigcup_{ \ell=1 }^{C_\Q}  Q_\ell  $ where $ |Q_\ell |= |Q| $ and $Q_\ell$ is near $Q$.
	Hence by Lemma \ref{lemma cube AQ AR-1} and    (i),
	\begin{align*}
		\| A_Q A_P^{-1} \| & \approx \esssup_{x\in P} \| W^{-1/p}   (x) A_Q \|  \le \esssup_{x\in \bigcup_{ \ell=1 }^{C_\Q}  Q_\ell } \| W^{-1/p}   (x) A_Q \|  \\
		&\lesssim \sup_{\ell \in \{ 1, \ldots, C_\Q\} } \esssup_{x\in  Q_\ell } \| W^{-1/p}   (x) A_{Q_\ell} \| \| A_{Q_\ell}^{-1} A_Q \| 
		\lesssim   \esssup_{x\in  Q_\ell } \| W^{-1/p}   (x) A_{Q_\ell} \| \lesssim   [W]_{\A_p} ^{1/p} .
	\end{align*}
	Hence,
	for a.e. $x\in P$,
	\begin{align*}
		\| W^{1/p} (x) A_P^{-1} \| \le 	\| W^{1/p} (x) A_Q ^{-1}  \| \|A_Q A_P^{-1} \| \lesssim  \| W^{1/p} (x) A_Q ^{-1}  \| .
	\end{align*}	
	Thus,
	\begin{align*}
		\sup_{P\in Q ^\sharp, x\in P}  	\| W^{1/p} (x) A_P^{-1} \| \lesssim \| W^{1/p} (x) A_Q ^{-1}  \| .
	\end{align*}
	Hence, by  (ii), we obtain
	\begin{equation*}
		\sup_{Q \in \Q}  \fint_Q  \sup_{P\in Q ^\sharp, x\in P}	\| W^{1/p} (x) A_P^{-1} \| ^r \d x \lesssim C_r, 
	\end{equation*}
	for $r \in (0, p+\delta_W)$.

	(iv)  This is proven in an identical manner with the starting point that  $ W^{-1/(p-1)} $ is an $\A_{p'}$ matrix weight.
	
	(v)
	Let $r\in (p +\delta_W)$. Assume that $ \int_Q  \sup_{P\in Q ^\sharp, x\in P}	\| W^{1/p} (x) A_P^{-1} \| ^r \d x \le B |Q |$ with some finite  $B $. Then we derive an a priori bound for $B$.
	Let  $D <\infty$ be a large constant to be specified later. Denote by  $\{ R_j\}$ the set of maximal cubes with respect to the partial order $\preceq$ such that $\|A_Q A_{R_j}^{-1} \|> D $. Then for $x\in Q\backslash \bigcup_j  R_j $, 
	\begin{equation*}
		\sup_{P\in Q ^\sharp, x\in P}	\| W^{1/p} (x) A_P^{-1} \| \le D \| W^{1/p} (x) A_Q^{-1} \| .
	\end{equation*}
	Thus by   (ii), for $r \in (0, p  +\delta_W)$, we have 
	\begin{align*}
		&  \left (  \frac{1}{|Q|} \int_{Q\backslash \bigcup R_j}  \sup_{P\in Q ^\sharp, x\in P}	\| W^{1/p} (x) A_P^{-1} \| ^r \d x \right)^{1/r} \\
		& \le \left (  D \frac{1}{|Q|} \int_{Q\backslash \bigcup R_j} \| W^{1/p} (x) A_P^{-1} \| ^r \d x  \right)^{1/r}  \lesssim D^{1/r}.
	\end{align*}
	We claim that $ \sum_j |R_j|< |Q|/2 $ if $D $ is sufficiently large. From Lemma \ref{lemma p> 1 A_Q -1  approx W -1/p},
	\begin{equation*}
		\| A_Q A_{R_j}^{-1} \| =	\|A_{R_j}^{-1} A_Q  \|    \approx \left( \fint_{R_j} |W^{-1/p} (x) A_Q   |^{p'} \d x \right)^{1/p'}.
	\end{equation*}
	That is  $ |R_j|  	\| A_Q A_{R_j}^{-1} \| ^{p'} \approx  \int_{R_j} |W^{-1/p} (x) A_Q   |^{p'} \d x $.
	Since $\{R_j\}$ is  the maximal cubes with respect to the partial order  $ \preceq$,  $\{R_j\}$  is  disjoint.
	And by Lemma \ref{lemma use covered Q}, $\bigcup_j  R_j  \subset  \bigcup_{ \ell=1 }^{C_\Q}  Q_\ell$  where $ |Q_\ell |= |Q| $ and $Q_\ell$ is near $Q$.
	Hence by Lemma \ref{lemma cube AQ AR-1},
	\begin{align*}
		\sum_j   	D^{p'} |R_j|  \lesssim \sum_j \int_{R_j} \|W^{-1/p} (x) A_Q   \|^{p'} \d x \le \sum_{\ell=1}^{C_\Q} \int_{Q_\ell}   \|W^{-1/p} (x) A_{Q_\ell}   \|^{p'} \d x \lesssim |Q| .
	\end{align*}
	This estimate shows that for  $D$ large enough, $\sum_j   	 |R_j| \le |Q|/2$, and the value
	of $D$ may be chosen independently of  $Q$.
	
	Inside each cube $R_j$ , we may assume that
	\begin{equation*}
		\sup_{P\in Q ^\sharp, x\in P}	\| W^{1/p} (x) A_P^{-1} \| = \sup_{P\in {R_j} ^\sharp, x\in P}	\| W^{1/p} (x) A_P^{-1} \| ,
	\end{equation*}
	otherwise the bound  $	\sup_{P\in Q ^\sharp, x\in P}	\| W^{1/p} (x) A_P^{-1} \| \le D \| W^{1/p} (x) A_Q^{-1} \|$ still holds. Then
	\begin{align*}
		\int_{\bigcup_j R_j  }  \sup_{P\in Q ^\sharp, x\in P}	\| W^{1/p} (x) A_P^{-1} \| ^r \d x &= \sum_j \int_{R_j} \sup_{P\in {R_j} ^\sharp, x\in P}	\| W^{1/p} (x) A_P^{-1} \| ^r \d x \\
		& \le B \sum_j |R_j| < \frac{1}{2} B |Q|.
	\end{align*}
	Putting these pieces together, we would discover that $B \le  C+ \frac{1}{2} B $ where
	$C <\infty $ is determined by the constants in the reverse H\"older inequality.
	
	(vi) Consider the dual weight $\widetilde W:= W^{-1/(p-1) }   \in \A_{p'}$. Let $\{ \tilde{A}_R \}_{R\in \Q}$ be the reducing operatorof order  $p'$  for $\tilde W$. Then for any $R\in \Q$  and a.e. $x\in \rn$,
	\begin{equation} \label{eq A_R W -1/p le tilde W}
		\| A_R W^{-1/p} (x) \| \le \| A_R  \tilde{A}_R \| \| \tilde{A}_R^{-1} \widetilde  W^{1/p'} (x) \| \lesssim [W]_{\A_p} ^{1/p}  \|  \widetilde  W^{1/p'} (x)\tilde{A}_R ^{-1} \|  
	\end{equation}
	due to Lemma \ref{lemma p> 1 A_Q -1  approx W -1/p}. We apply    (v)  to $p'$  and $\widetilde W \in \A_{p'}  $   in place of $p$  and $W\in  \A_p$. This shows that
	\begin{equation*}
		\sup_{Q\in \Q}  \left(  \fint_Q  \sup_{R\in Q ^\sharp, x\in R}  \|  \tilde  W^{1/p'} (x)\tilde{A}_R ^{-1} \|   ^r \d x \right)^{1/r}   \lesssim C_r .
	\end{equation*}
	for $r\in ( 0, p' + \delta_W)$.  Combine this and (\ref{eq A_R W -1/p le tilde W}), we obtain (vi).
	
	This finishes the proof of Lemma \ref{lemma reverse A_Q W}.
\end{proof}

\section{Matrix-weighted anisotropic Besov-type and Triebel-Lizorkin-type spaces}\label{sec BF spaces}
In this section, we introduce matrix-weighted anisotropic Besov-type and Triebel-Lizorkin-type spaces.

\begin{definition}
	Let $A$ be a dilation. Let $\tau \in [0,\infty)$,   $p, q \in (0,\infty]$,   and $ \{f_j\}_{j\in \mathbb Z} $ be a sequence of measurable functions on $\rn$.
	For $Q \in \Q$, we abbreviate $\widehat Q := Q \times \{ - \scale (Q), - \scale (Q) +1, \ldots \}$ so that
	\begin{equation*}
		\| \{f_j\}_{j\in \mathbb Z } \|_{ L \dot B _{ p,q }  (\widehat Q)} = \left( \sum_{j \ge - \scale (Q)} ^\infty  \|f_j\|_{L^p(Q)}^q    \right)^{1/q}
	\end{equation*}
and
	\begin{equation*}
	\| \{f_j\}_{j\in \mathbb Z } \|_{ L \dot F _{ p,q }  (\widehat Q)} = \left\| \left(   \sum_{j \ge - \scale (Q)} ^\infty |f_j|^q \right)^{1/q} \right\|_{L^p(Q)}    .
\end{equation*}

We may drop $\hat Q $ from these symbols when $\hat Q = \rn \times \mathbb Z$.

Define 
\begin{equation*}
	\| \{f_j\}_{j\in \mathbb Z } \|_{ L \dot A _{ p,q }^\tau }  := \sup_{Q \in \Q} |Q|^{-\tau } \| \{f_j\}_{j\in \mathbb Z } \|_{ L \dot A _{ p,q }  (\widehat Q)}
\end{equation*}
for both choices of $ L \dot A _{ p,q }^\tau  \in \{ L \dot B _{ p,q }^\tau , L \dot F _{ p,q }^\tau  \}$.
\end{definition}

Next we define the following eight function spaces.
\begin{definition} \label{def 8 func spaces}
	Let $A $ be a dilation and let  $W$ be a matrix weight. 
	Let $s\in \mathbb R$, $\tau \in [0,\infty) $, $p\in  (0,\infty)$,
	$q \in (0,\infty]$.
	Let $\varphi\in \mathcal S (\rn)$ satisfy (\ref{eq varphi supp}) and (\ref{eq varphi > 0}).
	
	Suppose that for each $ Q \in \Q $,  $A_Q$  is an $m\times m$ non-negative definite matrix.
	
	For each $j \in \mathbb Z$, we define 
	\begin{equation} \label{eq def A_j}
		\mathbb A_j := \sum_{Q\in \Q_j} A_Q \chi_Q.
	\end{equation}
	
	(i) The homogeneous matrix weighted anisotropic Triebel-Lizorkin-type space $ \dot F_{p,q}^{s,\tau} (W,A)$ 
	and the homogeneous matrix weighted anisotropic Besov-type space $ \dot B_{p,q}^{s,\tau} (W,A)$ 
	 are the sets of all $\vec f \in  ( \mathcal S_\infty ^\prime (\rn) )^m$ such that
	\begin{equation} \label{eq def A W A}
		\|\vec f\|_{\dot A_{p,q}^{s,\tau} (W,A) } := 	\left\| \left\{ |\det A|^{js} |W^{1/p}  \varphi_j * \vec f|\right \}_{j\in \mathbb Z } \right\|_{ L \dot A _{ p,q }^\tau  } <\infty ,
	\end{equation}
where $ \dot A_{p,q}^{s,\tau} (W,A) \in \{ \dot F_{p,q}^{s,\tau} (W,A), \dot B_{p,q}^{s,\tau } (W,A) \}$  and $ L \dot A _{ p,q }^\tau \in \{ L \dot B _{ p,q }^\tau, L \dot F _{ p,q }^\tau  \}$.

(ii)	The homogeneous $\{A_Q\}_{Q\in \Q}$ anisotropic Triebel-Lizorkin-type space $ \dot F_{p,q}^{s,\tau} (\{A_Q\}_{Q\in \Q},A)$ 
and the homogeneous $\{A_Q\}_{Q\in \Q}$ anisotropic Besov-type space $ \dot B_{p,q}^{s, \tau} (\{A_Q\}_{Q\in \Q}, A)$ 
are the sets of all $\vec f \in  ( \mathcal S_\infty ^\prime (\rn) )^m$ such that
\begin{equation*}
	\|\vec f\|_{\dot A_{p,q}^{s,\tau} (\{A_Q\}_{Q\in \Q} ,A) } := 	\left\| \left\{ |\det A|^{js} |\mathbb A_j  \varphi_j * \vec f|\right \}_{j\in \mathbb Z } \right\|_{ L \dot A _{ p,q }^\tau } <\infty ,
\end{equation*}
where $ \dot A_{p,q}^{s, \tau} (\{A_Q\}_{Q\in \Q},A) \in \{ \dot F_{p,q}^{s,\tau} (\{A_Q\}_{Q\in \Q},A), \dot B_{p,q}^{s,\tau} (\{A_Q\}_{Q\in \Q},A) \}$  and $ L \dot A _{ p,q }^\tau \in \{ L \dot B _{ p,q }^\tau, L \dot F _{ p,q }^\tau \}$.

For a sequence $ \vec s :  =  \{ \vec s_Q  \}_{Q\in \Q} \subset \mathbb C $, set $ \vec s_{Q,j} = \{ \vec s_Q \}_{Q\in \Q_j}  $.

(iii)
The discrete  anisotropic Besov-type space $ \dot b_{p,q}^{s,\tau} (W, A)$ 
and the discrete  anisotropic Triebel-Lizorkin-type space $ \dot f_{p,q}^{s,\tau} (W, A)$ 
are the sets of all sequences $\vec s :  =  \{ \vec s_Q  \}_{Q\in \Q} \subset \mathbb C $ such that
\begin{equation*}
	\|\vec s\|_{\dot a_{p,q}^{s,\tau} (W,A) } := 	\left\| \left\{ |\det A|^{j (s +1/2)} |W^{1/p}  \vec s_{Q,j} |\right \}_{j\in \mathbb Z } \right\|_{ L \dot A _{ p,q }^\tau } <\infty ,
\end{equation*}
where   $\dot a_{p,q}^{s,\tau} (W,A) \in \{ \dot b_{p,q}^{s,\tau} (W,A), \dot f_{p,q}^{s,\tau} (W,A) \}$  and $ L \dot A _{ p,q }^\tau \in \{ L \dot B _{ p,q }^\tau, L \dot F _{ p,q }^\tau \}$.

(iv)
 The discrete  anisotropic Besov-type space $ \dot b_{p,q}^{s,\tau} (\{A_Q\}_{Q\in \Q}, A)$ 
and the discrete  anisotropic Triebel-Lizorkin-type space $ \dot f_{p,q}^{s,\tau} (\{A_Q\}_{Q\in \Q}, A)$ 
are the sets of all sequences $\vec s :  =  \{ \vec s_Q  \}_{Q\in \Q} \subset \mathbb C $ such that
\begin{equation*}
	\|\vec s\|_{\dot a_{p,q}^{s,\tau} (\{A_Q\}_{Q\in \Q},A) } := 	\left\| \left\{ |\det A|^{j (s +1/2)} |\mathbb A_j  \vec s_{Q,j} |\right \}_{j\in \mathbb Z } \right\|_{ L \dot A _{ p,q }^\tau } <\infty ,
\end{equation*}
where   $\dot a_{p,q}^{s,\tau} (\{A_Q\}_{Q\in \Q},A) \in \{ \dot b_{p,q}^{s,\tau} (\{A_Q\}_{Q\in \Q},A), \dot f_{p,q}^{s,\tau} (\{A_Q\}_{Q\in \Q},A) \}$  and $ L \dot A _{ p,q }^\tau \in \{ L \dot B _{ p,q }^\tau, L \dot F _{ p,q }^\tau \}$.
\end{definition}

\begin{remark}
	(i)
	Let $ m =1$, $W \equiv 1$ and $ A_Q \equiv 1$ for $Q\in \Q$. Then we denote $ \dot a_{p,q}^{s,\tau} (A) := \dot a_{p,q}^{s,\tau} (\{A_Q\}_{Q\in \Q},A) = \dot a_{p,q}^{s,\tau} (W,A) $ for brevity.
	
	(ii) To emphasize the dependence on $\varphi$, we will use the notation $\dot A_{p,q}^{s,\tau} (W,A, \varphi)$ for (\ref{eq def A W A}). Later we will show that this definition is independent of $\varphi$.

\end{remark}

In what follows, the symbol $\hookrightarrow $ stands for continuous embedding. 
\begin{lemma}
		Let  $W$ be a matrix weight and  let $A $ be a dilation. 
	Let $s\in \mathbb R$, $\tau \in [0,\infty) $, $p\in  (0,\infty)$,
$q \in (0,\infty]$. Then
	\begin{equation*}
\dot B_{p, p \wedge q}^{s,\tau} (W,A)	 \hookrightarrow 	\dot F_{p,q}^{s,\tau} (W,A) \hookrightarrow \dot B_{p, q\vee q}^{s,\tau} (W,A),
	\end{equation*}
and 
	\begin{equation*}
	\dot b_{p, p \wedge q}^{s,\tau} (W,A)	 \hookrightarrow 	\dot f_{p,q}^{s,\tau} (W,A) \hookrightarrow \dot b_{p, q\vee q}^{s,\tau} (W,A).
\end{equation*}
\end{lemma}
\begin{proof}
	The proof is similar to \cite[p. 47]{Tri83} and we omit it here.
\end{proof}

\subsection{Equivalent norms}

For any sequence  $\mathbb A := \{A_Q\} _{Q\in \Q} $ of matrices, 
any $\varphi \in \S_\infty(\rn)$, and any $f \in ( \mathcal S_\infty ^\prime (\rn) )^m$, define
\begin{equation*}
	\sup_{ \mathbb A , \varphi} \vec f := \left\{  \sup_{ \mathbb A , \varphi, Q  } \vec f \right\}_{Q \in \Q}, 
\end{equation*}
where for each $Q\in \Q$,
\begin{equation*}
	\sup_{ \mathbb A , \varphi, Q  } \vec f := \sup_{y\in Q} |A_Q  \varphi_Q * \vec f (y) | =  |Q|^{1/2} \sup_{y\in Q} |A_Q  \varphi_{ -\scale (Q) } * \vec f (y) | . 
\end{equation*}

The following theorem is the main result of this subsection.

\begin{theorem}\label{theorem W app a app A_Q}
		Let $A $ be a dilation. 	Let $s\in \mathbb R$, $\tau \in [0,\infty) $, $p\in  (0,\infty)$,
		$q \in (0,\infty]$.
	Let  $W \in \A_p$ and $\mathbb A:= \{ A_Q\}_{Q\in \Q}$ be a sequence of reducing operators of order $p$ for $W$. 
	Let $\varphi\in \mathcal S (\rn)$ satisfy (\ref{eq varphi supp}) and (\ref{eq varphi > 0}).
	Then $\vec f \in \dot A_{p,q}^{s,\tau} (W,A)$ if and only  if $ \vec f \in \dot A_{p,q}^{s,\tau} (\{A_Q\}_{Q\in \Q} ,A)$. Moreover, for any $\vec f\in  (\S_\infty^\prime (\rn))^m $, 
	\begin{equation*}
		\| \vec f \|_{ \dot A_{p,q}^{s,\tau} (W,A) }  \approx  \left\|  \sup_{ \mathbb A , \varphi} \vec f      \right \|_{\dot a_{p,q}^{s,\tau} ( A) }  \approx   \|\vec f\|_{\dot A_{p,q}^{s,\tau} (\{A_Q\}_{Q\in \Q} ,A) } ,
	\end{equation*}
where the positive equivalence constants are independent of $\vec f$.
\end{theorem}
To prove Theorem \ref{theorem W app a app A_Q}, we need some preparation.
\begin{lemma} \label{lemma r trick all  tau}
	Let $A $ be a dilation. 	Let $s\in \mathbb R$, $\tau \in [0,\infty) $, $p\in  (0,\infty)$,
$q \in (0,\infty]$ and $M \in (1,\infty)$. 
Suppose that two sequences $\{g_j\}_{j\in\mathbb Z} $ and $\{h_j\}_{j \in \mathbb Z}$  of measurable functions on $\rn$ satisfy: there exist $r\in (0, \min\{p,q\})$ and a positive constant $c$ such that, for any $j\in \mathbb Z$ and $x\in \rn$,
	\begin{equation*}
	|g_j (x) |^r \le c  |\det A|^j \int_{\rn}  |h_j (z) |^r   \frac{1}{  (1 + |\det A|^j \rho_A  (x- z ) ) ^M }  \d z .
\end{equation*}
Then there exists a constant $C>0$, independent with $\{g_j\}_{j\in\mathbb Z} $ and $\{h_j\}_{j \in \mathbb Z}$, such that
\begin{equation*}
	\left\| 	\{ |\det A|^{js} g_j  \}_{j\in\mathbb Z}   \right\|_{L \dot A_{p,q}^\tau }   \le C \left\| 	\{ |\det A|^{js} h_j  \}_{j\in\mathbb Z}   \right\|_{L \dot A_{p,q}^\tau }  .
\end{equation*}
\end{lemma}

\begin{proof}
Fix $P\in \Q$. For any $j \in \{ - \scale (P), - \scale (P) + 1, \ldots  \}$ and $x\in P$,
\begin{align*}
	\left| |\det A|^{js} g_j (x) \right|^r \lesssim   |\det A|^j \int_{\rn}   \frac{|\det A|^{jsr}  |h_j (z) |^r }{  (1 + |\det A|^j \rho_A  (x- z ) ) ^M }  \d z .
\end{align*}
Set  
\begin{equation*}
	E_i = \{ R \in \Q_{- \scale (P)   } :  i  <   \rho_A  (x_P -  x_R ) / |P|  \le   (i+ 1)  \} \quad \operatorname{for} \;  i\in \mathbb N,
\end{equation*}
and 
\begin{equation*}
	E_0 = \{ R \in \Q_{ - \scale (P)   } :      \rho_A  (x_P -  x_R )  \le  |P|  \} \quad \operatorname{for} \;  i\in \mathbb N .
\end{equation*}
By Lemma \ref{lemma rho A and |x|}, $\sharp E_i \le C_A < \infty$ for each $i\in \mathbb Z_+$. Indeed, suppose that $P = A^{\scale (P)   }  ([0,1)^n + k_0 ) $ for some fixed $k_0\in \mathbb Z^n$.  For $E_i, i\in\mathbb N$, $ \rho_A(  A^{\scale (P)   } ( k_0 -k_1)  ) /|P| =  \rho_A( k_0 -k_1) \ge 1$ where $k_1  \in \mathbb Z^n$. Hence,  $  C^{-1}  i^{\zeta_-} \le  |k_0 -k_1| \le C (i+1)^{\zeta_+}$ where $\zeta_-, \zeta_+$  are  the same as in Lemma \ref{lemma rho A and |x|}. From this,we obtain $\sharp E_i \le C_A < \infty$ for each $i\in  \mathbb Z_+$.

Thus
\begin{align*}
	\left| |\det A|^{js} g_j (x) \right|^r  & \lesssim  \sum_{i=0}^\infty  \sum_{Q \in E_i}   |\det A|^j \int_{Q}   \frac{|\det A|^{jsr}  |h_j (z) |^r }{  (1 + |\det A|^j \rho_A  (x- z ) ) ^M }  \d z  \\
	&\lesssim  \sum_{i=0}^\infty  \frac{1}{ (1+i)^M } \sum_{Q \in E_i}   \frac{1}{|Q|} \int_{Q}  |\det A|^{jsr}  |h_j (z) |^r   \d z  \\
	&\lesssim \sum_{i=0}^\infty  \frac{1}{ (1+i)^M }   \sum_{Q\in E_i} \M_{\rho_A} \left( |\det A|^{jsr}  |h_j  |^r  \chi_Q \right) (x_Q) .
\end{align*}
By Minkowski's inequality, Lemma \ref{lemma M rho A FS} and $M>1$, we obtain
\begin{align*}
	\left\| \left\{ \left| |\det A|^{js} g_j  \right|^r   \right \}_{j \in\mathbb Z}    \right\|_{ L \dot F_{p/r, q/r } (\hat P)  }  &\lesssim 	\left\| \left \{ \sum_{i=0}^\infty   \frac{1}{(1 +  i ) ^M}  \sum_{Q\in E_i} \M_{\rho_A} \left( |\det A|^{jsr}  |h_j  |^r  \chi_Q \right)   \right\}_{j \in\mathbb Z}    \right\|_{ L \dot F_{p/r, q/r } (\hat P)   }  \\
	& \lesssim   \sum_{i=0}^\infty   \frac{1}{(1 +  i ) ^M}  \sum_{Q\in E_i} \left\| \left \{ \M_{\rho_A} \left( |\det A|^{jsr}  |h_j  |^r  \chi_Q \right)  \right\}_{j \ge -\scale (P)}    \right\|_{ L \dot F_{p/r, q/r }    }  \\
	& \lesssim   \sum_{i=0}^\infty   \frac{1}{(1 +  i ) ^M}  \sum_{Q\in E_i} \left\| \left \{ \left( |\det A|^{jsr}  |h_j (z) |^r  \chi_Q \right)  \right\}_{j \ge -\scale (P)}    \right\|_{ L \dot F_{p/r, q/r }    }  \\
	&\lesssim \left(  |P|^\tau \left\|  \{ |\det A|^{js} h_j  \}_{j\in\mathbb Z}   \right\|_{L \dot F_{p,q}^\tau  } \right) ^r .
\end{align*}
Taking the supremum over $P\in \Q$, we obtain
\begin{equation*}
 \left\| 	\{ |\det A|^{js} g_j  \}_{j\in\mathbb Z}   \right\|_{L \dot F_{p,q}^\tau }   \lesssim \left\| 	\{ |\det A|^{js} h_j  \}_{j\in\mathbb Z}   \right\|_{L \dot F_{p,q}^\tau }  .
\end{equation*}
 The proof of case of $ L \dot B_{p,q}^\tau$ is similar. Thus we finish the proof.
\end{proof}

First we establish the relations between $\|  \sup_{ \mathbb A , \varphi} \vec f       \|_{\dot a_{p,q}^{s,\tau} ( A) } $ and $\| \vec f\|_{ \dot A_{p,q}^{s,\tau} (\{A_Q\}_{Q\in \Q} ,A)  }$.

\begin{theorem} \label{theorem sup A_Q less inf}
	Let $A $ be a dilation. 	Let $s\in \mathbb R$, $\tau \in [0,\infty) $, $p\in  (0,\infty)$,
	$q \in (0,\infty]$.
		Let  $W \in \A_p$ and $\mathbb A:= \{ A_Q\}_{Q\in \Q}$ be a sequence of reducing operators of order $p$ for $W$. 
	Let $\varphi\in \mathcal S (\rn)$ satisfy (\ref{eq varphi supp}) and (\ref{eq varphi > 0}).
	Then for any $\vec f \in (\S_\infty^\prime (\rn))^m $, 
	\begin{align*}
		\| \vec f\|_{ \dot A_{p,q}^{s,\tau} (\{A_Q\}_{Q\in \Q} ,A)  }  \le \left\|  \sup_{ \mathbb A , \varphi} \vec f      \right \|_{\dot a_{p,q}^{s,\tau} ( A) }  \lesssim 	\| \vec f\|_{ \dot A_{p,q}^{s,\tau} (\{A_Q\}_{Q\in \Q} ,A)  } .
	\end{align*}
\end{theorem}

\begin{proof}
	The first  inequality comes from the definition, so let us show the second inequality.
	Let $\gamma \in \mathcal S(\rn) $  such that $\F \gamma (\xi) = 1$ on the support of $\F \varphi$ and supp $\F \gamma  \subset  [-\pi, \pi]^n $. (Here we may suppose that $\F\varphi$  is compactly supported in $[-\pi, \pi]^n \backslash\{0\}$.) Let 
	$\gamma_j (x) = |\det A|^{j} \gamma( A ^j x) $. Then  $\F \gamma_j (\xi) = \F \gamma ( (A^\ast)^{-j}  \xi)$ and supp $\F \gamma_j \subset (A^\ast)^{j} [-\pi,\pi]^n  $. 
	
	Hence, for any $\vec g = (g_1, \ldots, g_m)^T$ with supp $\F  g_i \subset( A^\ast) ^{j} [-\pi, \pi]^n $  for $i\in \{1,\ldots,m\}$, we have
	$
		\vec g = \vec g * \gamma_j.
	$	
	By Lemma \ref{lemma identity}, we have the identity
		\begin{equation*}
		\vec g (x)  = \vec g * \gamma_j (x) = \sum_{k\in \mathbb Z} |\det A|^{-j} g (A^{-j} k)  h (x - A^{-j} k).
	\end{equation*}
	We apply this identity with $\vec g (x) =\varphi_j * \vec f (x+y) $, for an arbitrary $y\in \rn$,  to
	obtain
	
		\begin{equation*}
	\varphi_j * \vec f (x +y ) =  \sum_{k\in \mathbb Z} |\det A|^{-j}  \varphi_j * \vec f (A^{-j} k +y) \gamma_j (x - A^{-j} k) .
	\end{equation*}
Leting $z =x+y$ in above identity, we obtain for any $z,y \in \rn$,
	\begin{align} \label{eq identity varphi_j f}
		\nonumber
	\varphi_j * \vec f (z  ) & =  \sum_{k\in \mathbb Z} |\det A|^{-j}  \varphi_j * \vec f (A^{-j} k +y) \gamma_j (z - A^{-j} k -y) \\
	& = \sum_{R \in \Q_j }  |\det A|^{-j} \varphi_j * \vec f (x_R  +y )  \gamma_j  (z - x_R -y),
\end{align}
where $ x_R =A^{-j} k $ is the left corner of $R $.

Let $ r \in (0,\min\{ 1,p,q \}) $ and $M > \Delta + 1 /r $ where $\Delta$ is the same as in Lemma \ref{lemma cube AQ AR-1}.
Using (\ref{eq identity varphi_j f}), (\ref{eq baisic geo}), Lemma \ref{lemma cube AQ AR-1} and $\gamma \in\mathcal S (\rn)$, 
we have that for $j\in \mathbb Z$, $Q\in \Q_j $, $x\in Q$, any $y \in A ^{-j} [0,1)^n$,
\begin{align*}
	\nonumber
	|A_Q (\varphi_j * \vec f ) (x) |^r & \lesssim   \sum_{R \in \Q_j }   |A_Q  \varphi_j * \vec f (x_R +y ) |^r   \frac{1}{  (1 + |\det A|^j \rho_A  ( x- x_R -y ) ) ^{Mr} } \\
	\nonumber
	& \approx   \sum_{R \in \Q_j }   |A_Q  \varphi_j * \vec f (x_R +y ) |^r   \frac{1}{  (1 + |\det A|^j \rho_A  ( x_Q- x_R  ) ) ^{Mr} } \\
	& \lesssim \sum_{R \in \Q_j }   |A_R \varphi_j * \vec f (x_R +y) |^r   \frac{1}{  (1 + |\det A|^j \rho_A  ( x_Q- x_R ) ) ^{(M -\Delta) r} } .
\end{align*}
Taking the average over $y \in  A^{-j}  [0,1)^n$, we obtain
\begin{align}\label{eq AQ r trick}
	\nonumber
	|A_Q (\varphi_j * \vec f ) (x) |^r &  \lesssim   \sum_{R \in \Q_j }   |\det A|^j \int_R  |A_R \varphi_j * \vec f (z) |^r   \frac{1}{  (1 + |\det A|^j \rho_A  ( x_Q- x_R ) ) ^{(M -\Delta) r} }  \d z \\
	&\approx  |\det A|^j\int_\rn  |\mathbb A_j \varphi_j * \vec f (z) |^r   \frac{1}{  (1 + |\det A|^j \rho_A  ( x- z ) ) ^{(M -\Delta) r} }  \d z ,
\end{align}
where $\mathbb A_j$ is the same as in (\ref{eq def A_j}).
Taking the supremum over $x\in Q$, we obtain
\begin{equation*} 
	\sup_{ \mathbb A , \varphi, Q  } 	|A_Q (\varphi_j * \vec f ) (x) |^r  \lesssim |\det A|^j\int_\rn    \frac{ |\mathbb A_j \varphi_j * \vec f (z) |^r}{  (1 + |\det A|^j \rho_A  ( x- z ) ) ^{(M -\Delta) r} }  \d z .
\end{equation*}
For each $j \in \mathbb Z$, 
let 
\begin{equation*}
	g_j := \sum_{Q \in \Q_j } \sup_{ \mathbb A , \varphi, Q  } \vec f \chi_Q  \quad \operatorname{and} \quad 	h_j = \mathbb A_j \varphi_j *\vec f .
\end{equation*}
By Lemma \ref{lemma r trick all  tau}, we obtain
\begin{align*}
	\left\|  \sup_{ \mathbb A , \varphi} \vec f      \right\|_{\dot a_{p,q}^{s,\tau} (A) } =  \left\| 	\{ |\det A|^{js} g_j  \}_{j\in\mathbb Z}   \right\|_{L \dot A_{p,q}^\tau }   \lesssim \left\| 	\{ |\det A|^{js} h_j  \}_{j\in\mathbb Z}   \right\|_{L \dot A_{p,q}^\tau }
=  \left\|   \vec f     \right  \|_{\dot A_{p,q}^{s,\tau} (\{A_Q\}_{Q\in \Q} ,A) }  .
\end{align*}

This finishes the proof of Theorem \ref{theorem sup A_Q less inf}.
\end{proof}

\begin{lemma} \label{lemma AQ les W}
		 Let $A $ be a dilation. 
Let $s\in \mathbb R$, $\tau \in [0,\infty) $, $p\in  (0,\infty)$,
$q \in (0,\infty]$.
	 Let  $W \in \A_p$ and $\mathbb A: =\{ A_Q\}_{Q\in \Q}$ be a sequence of reducing operators of order $p$ for $W$.
	Let $\varphi\in \mathcal S (\rn)$ satisfy (\ref{eq varphi supp}) and (\ref{eq varphi > 0}).
	Then for any $\vec f \in (\S_\infty^\prime (\rn))^m$,	
	\begin{equation} \label{eq A_Q less W}
		\| \vec f\|_{  \dot A_{p,q}^{s,\tau} (\{A_Q\}_{Q\in \Q} ,A) }  \lesssim \| \vec f\|_{\dot A_{p,q}^{s,\tau} (W ,A)} .
	\end{equation}
\end{lemma}

\begin{proof}
	 We consider the following two cases on $p$.
	
	Case $p \in (0,1]$. 
	By   Lemma \ref{lemma reverse A_Q W} (i), for any $j \in \mathbb Z$ and a.e. $x\in \rn$,
	\begin{equation*}
		|\mathbb A_j (x) \varphi_j *\vec f (x) | \le \| \mathbb A_j (x) W^{-1/p} (x) \|  | W^{1/p} (x)   \varphi_j *\vec f (x) |  \lesssim  | W^{1/p} (x)   \varphi_j *\vec f (x) | ,
	\end{equation*}
where $\mathbb A_j$   is the same as in (\ref{eq def A_j}). This finishes the proof of (\ref{eq A_Q less W}) in this case.
	
	Case $p \in (1,\infty)$.  From (\ref{eq AQ r trick}) and (\ref{eq baisic geo}), we have that for $j\in \mathbb Z$, $Q\in \Q_j $, $x\in Q$, 
	\begin{align*}
		|A_Q \varphi_j *\vec f (x) |^r 
		&\lesssim 		
		  \sum_{R \in \Q_j }   
		 |\det A|^j  \int_R 
		 | A_R \varphi_j * \vec f (z ) |^r   \frac{  1}{  (1 + |\det A|^j \rho_A  ( x- z) ) ^{(M -\Delta) r} }   \d z \\
		 	&\le		
		  \sum_{R \in \Q_j }   |\det A|^j 
		 \int_R  \|A_R W^{-1/p} (z)\| ^r
		   \frac{ |  W^{1/p} (z) \varphi_j * \vec f (z ) |^r }{  (1 + |\det A|^j \rho_A  ( x- z) ) ^{(M -\Delta) r} }    \d z  \\
		   & =  \int_\rn  |\det A|^j     \| \mathbb A_j (z) W^{-1/p} (z)\| ^r
		   \frac{ |  W^{1/p} (z) \varphi_j * \vec f (z ) |^r }{  (1 + |\det A|^j \rho_A  ( x- z) ) ^{(M -\Delta) r} }    \d z ,
	\end{align*}
where $\Delta $ is the same as in Lemma \ref{lemma cube AQ AR-1}.
Let $r$  be sufficiently small such that we can use  Lemma \ref{lemma reverse A_Q W} (iv).
By H\"older's inequality, we get
\begin{align*}
	|A_Q \varphi_j *\vec f (x) |^r  & \lesssim \left(  |\det A|^j  \int_\rn  \frac{ \| \mathbb A_j (z) W^{-1/p} (z)\| ^{rp'} }{ (1 + |\det A|^j \rho_A  ( x- z) ) ^{(M -\Delta) r}  }  \d z \right)^{1/p'}   \\
	&\quad \times \left( |\det A|^j  \int_\rn \frac{ |  W^{1/p} (z) \varphi_j * \vec f (z ) |^{rp} }{  (1 + |\det A|^j \rho_A  ( x- z) ) ^{(M -\Delta) r  } }    \d z \right)^{1/p}.
\end{align*}
Let $ (M -\Delta) r >1 $.
Using Lemma \ref{lemma reverse A_Q W} (iv),
\begin{align*}
&	 \left(  |\det A|^j  \int_\rn  \frac{ \| \mathbb A_j (z) W^{-1/p} (z)\| ^{rp'} }{ (1 + |\det A|^j \rho_A  ( x- z) ) ^{(M -\Delta) r}  }  \d z \right)^{1/p'}   \\
& =  \left(  |\det A|^j  \sum_{R \in \Q_j }  \int_R \frac{ \|  A_R W^{-1/p} (z)\| ^{rp'} }{ (1 + |\det A|^j \rho_A  ( x- z) ) ^{(M -\Delta) r}  }   \d z \right)^{1/p'} \\
& \lesssim \left(  \sum_{R \in \Q_j }  (1 + |\det A|^j \rho_A  ( x- x_R) ) ^{ - (M -\Delta) r}    \right)^{1/p'}  <\infty .
\end{align*}
Hence 
\begin{align*}
	|A_Q \varphi_j *\vec f (x) |^{rp}  & \lesssim   |\det A|^j  \int_\rn \frac{ |  W^{1/p} (z) \varphi_j * \vec f (z ) |^{rp} }{  (1 + |\det A|^j \rho_A  ( x- z) ) ^{(M -\Delta) r  } }    \d z .
\end{align*}
Let $r$ be sufficiently small such that   $rp \in (0,  \min\{1, p,q\} ) $.
Let $g_j (x) := |A_Q \varphi_j *\vec f (x) | $and $ h_j (x) : =|  W^{1/p} (x) \varphi_j * \vec f (x ) | $.
By Lemma \ref{lemma r trick all  tau}, we obtain
\begin{align*}
	\| \vec f\|_{  \dot A_{p,q}^{s,\tau} (\{A_Q\}_{Q\in \Q} ,A) }  =  \left\| 	\{ |\det A|^{js} g_j  \}_{j\in\mathbb Z}   \right\|_{L \dot A_{p,q}^\tau }   \lesssim \left\| 	\{ |\det A|^{js} h_j  \}_{j\in\mathbb Z}   \right\|_{L \dot A_{p,q}^\tau }
	=   \| \vec f\|_{  \dot A_{p,q}^{s,\tau} (W ,A) } .
\end{align*}
This finishes the proof of (\ref{eq A_Q less W}) in this case and hence Lemma \ref{lemma AQ les W}.
\end{proof}

For $j\in\mathbb Z $ and $f \in L^1_{\operatorname{loc}}$, define 
\begin{equation*}
	E_j (f) : = \sum_{Q \in \Q_j}  \chi_Q \fint_Q f (x) \d x.
\end{equation*}

Similar to \cite[Lemma 3.6]{FR21} and \cite[Theorem 3.7]{FR21}, respectively, we have the following two lemmas.

\begin{lemma}\label{lemma alpha g L1 le g L1}
	Suppose  $\alpha = \{ \alpha_j\}_{j\in\mathbb Z} $ is a sequence of non-negative measurable functions
	on $\rn$ such that
	\begin{equation*}
		\sup_{Q \in \Q} \fint_Q \sup_{j\in \mathbb Z:  j \ge - \scale (Q) } \alpha_j (x) \d x <\infty .
	\end{equation*}
	Then, for any sequence $\{ g_j \}_{j\in\mathbb Z} $  of functions on $\rn$  such that for every  $j \in \mathbb Z$, $g_j$ is 	constant on each  cube $ Q$  with $ |Q| = |\det A|^{-j} $ , we have
	
	\begin{equation*}
		\|\sup_{j\in \mathbb Z}  |\alpha_j g_j | \|_{L^1} \lesssim 	\|\sup_{j\in \mathbb Z}  | g_j | \|_{L^1} .
	\end{equation*}
\end{lemma}

\begin{lemma} \label{lemma gamma_j bound Lp ell q}
	Suppose $\{\gamma_j\}_{j \in \mathbb Z}$ is a sequence of non-negative measurable 	functions on $\rn$.
	
	{\rm (i)}
	 Suppose  $0 < q \le p <\infty$, and $\{\gamma_j\}_{j \in \mathbb Z}$ satisfies
	\begin{equation} \label{eq gamma_j 1}
		\sup_{Q \in \Q _j} \fint_Q \gamma_j^{p (1+\delta)} \d x \le c
	\end{equation} 
for some $c,\delta>0 $, independent of  $j \in \mathbb Z$. Then there exists  $C > 0$ such that for any
sequence  $\{ f_j \} _{j \in \mathbb Z  } $ of measurable functions on  $\rn$,
\begin{equation*}
	\| \{ \gamma_j E_j (f_j) \} \|_{ L^p (\ell^q)} \le  C \| \{  E_j (f_j) \} \|_{ L^p (\ell^q)} .
\end{equation*}

	{\rm (ii)} Suppose  $1 < p < \infty, 1  \le q \le \infty$,  and   $ \{ \gamma_j\}_{  j \in \mathbb Z}  $ satisfies
\begin{equation} \label{eq gamma_j P subset Q}
	\sup_{Q \in \Q} \fint_Q \sup_{j\in \mathbb Z:  j \ge - \scale (Q) }  \gamma_j^{p (1+\delta)} \d x <\infty .
\end{equation}
for some $\delta >0 $. Then there exists  $C > 0$ such that for any
sequence  $\{ f_j \} _{j \in \mathbb Z  } $ of measurable functions on  $\rn$,
\begin{equation*}
	\| \{ \gamma_j E_j (f_j) \} \|_{ L^p (\ell^q)} \le C \| \{  E_j (f_j) \} \|_{ L^p (\ell^q)} .
\end{equation*}
\end{lemma}

\begin{corollary} \label{cor gamma j E j Lp ell q}
		 Let $A $ be a dilation. 
Let  $p\in  (0,\infty)$,
$q \in (0,\infty]$.
	 Let  $W \in \A_p$ and $\{ A_Q\}_{Q\in \Q}$ be a sequence of reducing operators of order $p$ for $W$.		
	For $j\in \mathbb Z$, set
	\begin{equation} \label{eq gamma_j W a_q^ -1}
		\gamma_j = \sum_{Q\in \Q}  \|W^{1/p } A_Q^{-1} \|\chi_Q .
	\end{equation}
Then for any sequence $\{ f_j\}_{j \in \mathbb Z} $  of nonnegative measurable functions on $\rn$ or for any $\{ f_j\}_{j \in \mathbb Z}  \subset L^1_{\operatorname{loc}}$,
 \begin{equation} \label{eq gamma_j Lp ell q}
	\| \{ \gamma_j E_j (f_j)\}_{ L^p (\ell^q)}  \lesssim 	\| \{  E_j (f_j)\}_{ L^p (\ell^q)}  .
\end{equation}
\end{corollary}
\begin{proof}
	Suppose first that  $0 < q\le p < \infty$. Then (\ref{eq gamma_j 1}) holds for some $ \delta > 0$ by Lemma \ref{lemma reverse A_Q W}. Then (\ref{eq gamma_j Lp ell q}) follows from  Lemma \ref{lemma gamma_j bound Lp ell q} (i). 
	
	If $1<p <\infty $, $ 1\le q\le \infty$,   then (\ref{eq gamma_j P subset Q})  holds by Lemma \ref{lemma reverse A_Q W}. Using  Lemma \ref{lemma gamma_j bound Lp ell q} (ii),  we obtain (\ref{eq gamma_j Lp ell q}) in this case.
 
It remains to prove (\ref{eq gamma_j Lp ell q})  for  $0< p \le 1 $, $ p \le q \le \infty$. Pick $ \epsilon >0 $ sufficiently small that  $p/\epsilon > 1$ (and hence, $q/\epsilon > 1$). Then
\begin{equation*}
	\| \{ \gamma_j E_j (f_j)\} \|_{ L^p  (  \ell^q ) } ^\epsilon  =  \| \{ \gamma_j ^\epsilon | E_j (f_j)|^\epsilon \} \|_{ L^{p/\epsilon}  (  \ell^{q/\epsilon} ) }  .
\end{equation*}

By Lemma \ref{lemma reverse A_Q W}, the sequence $\{  \gamma_j ^\epsilon \}_{j\in \mathbb Z} $ satisfies
\begin{equation*}
	\sup_{Q\in \Q}  \fint_Q \sup_{P\in Q ^\sharp, x\in P}	( \gamma_j^\epsilon) ^{ p (1+\delta) /\epsilon} \d x <\infty .
\end{equation*}
Thus, by    Lemma \ref{lemma gamma_j bound Lp ell q} (ii), we obtain
\begin{equation*}
  \| \{ \gamma_j ^\epsilon | E_j (f_j)|^\epsilon \} \|_{ L^{p/\epsilon}  (  \ell^{q/\epsilon} ) }  \lesssim   \| \{  | E_j (f_j)|^\epsilon \} \|_{ L^{p/\epsilon}  (  \ell^{q/\epsilon} ) }   =   \| \{ | E_j (f_j)| \} \|_{ L^{p}  (  \ell^{q} ) } ^\epsilon ,
\end{equation*}
as desired.
\end{proof}

\begin{remark} \label{remark gamma j Ej Lp Q}
		In Corollary \ref{cor gamma j E j Lp ell q}, if, for any $i \in \mathbb Z $, let $f_i = g$ if $i=j$ and let $f_i = 0$  if $ i \neq j$,
	where $g$ is a nonnegative measurable function on $\rn$  or $g \in L^1_{\operatorname{loc}}$, then, for any $j\in \mathbb Z$,
	\begin{equation*}
		\|   \gamma_j E_j (g)  \|_{L^p }  \lesssim  	\| E_j (g)  \|_{L^p},
	\end{equation*}  
	where the  implicit positive constant is independent of $g$ and $j$.
	
	Fix a dilated cube $Q\in \Q$. Let $h = g \chi_Q$ and  $j \in \mathbb Z $. Then 
	\begin{align*}
		\|   \gamma_j E_j (g)  \|_{L^p (Q) } =  \|   \gamma_j E_j (h)  \|_{L^p  } \lesssim  \|    E_j (h)  \|_{L^p  }  =  \|    E_j (g)  \|_{L^p  (Q)}, 
	\end{align*}
where the  implicit positive constant is independent of $g$, $j$ and $Q$.
\end{remark}

\begin{corollary} \label{cor gamma_j  W A_Q LF}
		 Let $A $ be a dilation. 
		Let  $W \in \A_p$ and $\{ A_Q\}_{Q\in \Q}$ be a sequence of reducing operators of order $p$ for $W$.
	Let $\tau \in [0,\infty) $, $p\in  (0,\infty)$,
	$q \in (0,\infty]$.		
	For $j\in \mathbb Z$, set
	\begin{equation*}
		\gamma_j = \sum_{Q\in \Q}  \|W^{1/p } A_Q^{-1} \|\chi_Q .
	\end{equation*}
Then for both choices of $ L \dot A _{ p,q }^\tau \in \{ L \dot B _{ p,q }^\tau, L \dot F _{ p,q }^\tau \}$,
\begin{equation*} 
	\| \{ \gamma_j E_j (f_j)\} \|_{ L \dot A _{ p,q }^\tau }  \lesssim 	\| \{  E_j (f_j)\}  \|_{  L \dot A_{ p,q }^\tau }  .
\end{equation*}
\end{corollary}

\begin{proof}
	From Remark \ref{remark gamma j Ej Lp Q}, it suffices to prove 
	\begin{equation*} 
		\| \{ \gamma_j E_j (f_j)\} \|_{ L \dot F _{ p,q }^\tau }  \lesssim 	\| \{  E_j (f_j)\}  \|_{  L \dot F_{ p,q }^\tau }  .
	\end{equation*}
	By Lemma \ref{lemma use covered Q}, for any $Q\in \Q$, there exists $Q^\ast :=  \bigcup_{ \ell=1 }^{C_\Q}  Q_\ell  $ such that $Q^\sharp \subset Q^\ast $ and $|Q|\approx |Q^\ast| $. Here $Q ^\sharp : = \{P\in\Q :  P \preceq Q, |P\cap Q| >0 \} $.
Then
\begin{align*}
 \left \| \left(  \sum_{j \ge - \scale (Q)  }  |\gamma_j E_j (f_j)|^q \right)^{1/q}  \right\| _{L^p (Q)} 
	&\le \left \| \left(  \sum_{j \ge - \scale (Q)  }  |\gamma_j E_j (f_j \chi_{Q^\ast })|^q \right)^{1/q}  \right\| _{L^p } \\
	&\lesssim  \left\| \left(  \sum_{j \ge - \scale (Q)  }  |E_j (f_j )|^q \right)^{1/q}  \right\| _{L^p (Q^\ast) } \\
		&\lesssim \sum_{i=1}^{C_\Q} \left \| \left(  \sum_{j \ge - \scale (Q)  }  |E_j (f_j )|^q \right)^{1/q}  \right\| _{L^p (Q_i ) } \\
		&\lesssim |Q|^\tau \| \{  E_j (f_j) \} \|_{ L \dot F _{ p,q }^\tau } .
\end{align*}
Hence
\begin{align*}
		\| \{ \gamma_j E_j (f_j) \} \|_{ L \dot F _{ p,q }^\tau } = \sup_{  Q\in \Q } |Q|^{-\tau} \left \| \left(  \sum_{j \ge - \scale (Q)  }  |\gamma_j E_j (f_j)|^q \right)^{1/q}  \right\| _{L^p (Q)}  \lesssim \| \{  E_j (f_j) \} \|_{ L \dot F _{ p,q }^\tau } .
\end{align*}
The proof is finished.
\end{proof}

\begin{theorem} \label{theorem f W lesssim A_Q}
	 Let $A $ be a dilation. 
	 Let $s\in \mathbb R$, $\tau \in [0,\infty) $, $p\in  (0,\infty)$,
	 $q \in (0,\infty]$.
		Let  $W \in \A_p$ and $\{ A_Q\}_{Q\in \Q}$ be a sequence of reducing operators of order $p$ for $W$.
	 Then 
	\begin{equation*}
		\| \vec f\|_{\dot A_{p,q}^{s,\tau} (W ,A) } \lesssim 	\| \vec f\|_{\dot A_{p,q}^{s,\tau} (\{A_Q\}_{Q\in \Q} ,A) } .
	\end{equation*}
\end{theorem}

\begin{proof}
	For $j\in \mathbb Z$, set
\begin{equation*}
	\gamma_j (x) = \sum_{Q\in \Q}  \|W^{1/p } (x) A_Q^{-1} \|\chi_Q  (x).
\end{equation*}
Define 
\begin{equation*}
	h_j (x) = |\det A|^{ j s} |W^{1/p} (x) \varphi_j * \vec f(x) | ,
\end{equation*}
and 
\begin{equation*}
	k_j (x) = \sum_{Q\in \Q_j} |\det A|^{ j s}  \chi_Q  (x) \left(  \sup_{y\in Q}  |A_Q \varphi_j * \vec f (y) | \right) .
\end{equation*} 
Note that each $k_j $ is a constant on $Q \in \Q_j$ and thus $E_j k_j = k_j$. Then
\begin{equation*}
	h_j  (x)\le \gamma_j (x)  k_j (x).
\end{equation*}
Hence by Corollary \ref{cor gamma_j  W A_Q LF},
\begin{align*}
	\| \vec f\|_{\dot F_{p,q}^{s,\tau} (W ,A)}   = \| \{ h_j \} \|_{  L \dot F _{ p,q }^\tau } \le  \| \{ \gamma_j k_j \} \|_{  L \dot F _{ p,q }^\tau }  \lesssim \| \{ k_j \} \|_{  L \dot F _{ p,q }^\tau }  \lesssim \| \vec f\|_{\dot F_{p,q}^{s,\tau} (\{A_Q\}_{Q\in \Q} ,A)},
\end{align*} 
where the last step is by Theorem \ref{theorem sup A_Q less inf}.
 \end{proof}

Theorems  \ref{theorem sup A_Q less inf},  \ref{theorem f W lesssim A_Q},  and  Lemma \ref{lemma AQ les W}  lead to Theorem \ref{theorem W app a app A_Q}.

\subsection{Properties of sequence spaces}

The following result gives a characterization of the  $\dot f ^{s,\tau}_{p,q}(A)$-norm via sequences of sparse sets.
\begin{lemma} \label{lemma sparse char dot f}
Let $A $ be a dilation. 
 Let $s\in \mathbb R$, $\tau \in [0,\infty) $, $p\in  (0,\infty)$,
$q \in (0,\infty]$.
Suppose that, for any $Q\in \Q$, $E_Q \subset Q$  is a measurable set with $|E_Q| \ge \delta |Q|$ for some $\delta \in (0,1)$. Then for $w =  \{ w_Q\}_{Q\in \Q} \subset \mathbb C$,
\begin{equation*}
	\| w \|_{\dot f ^{s,\tau}_{p,q}(A)  } \approx 	\left\| \left\{ |\det A|^{j (s +1/2)} \sum_{Q\in \Q_j} w_Q \chi_{E_Q} \right \}_{j\in \mathbb Z } \right\|_{ L \dot F _{ p,q }^\tau }  .
\end{equation*}
\end{lemma}

\begin{proof}
	Since  $E_Q \subset Q $, one direction is trivial. For another direction, we use
	the fact that for any  $ \beta >0   $,
	\begin{equation*}
		\chi_Q \lesssim \delta^{ -1/\beta}  ( \M_{\rho_A} ( \chi_{E_Q} ^ \beta ) )^{1/\beta} .
	\end{equation*}
Let $ \beta \in (0, \min\{ 1,p,q\} )$. Using Lemma \ref{lemma M rho A FS} and $E_Q \subset Q$, we obtain 
\begin{align*}
		\| w \|_{\dot f ^{s,\tau}_{p,q}(A)  }
		 & \lesssim  \delta^{ -1/\beta} \sup_{P\in \Q} |P|^{-\tau} \left\| \left(  \sum_{j\ge -\scale (P)}  |\det A|^{j q (s +1/2)} \sum_{Q\in \Q_j} w_Q ^q ( \M_{\rho_A} ( \chi_{E_Q} ^ \beta ) )^{q/\beta} \right) ^{1/q}  \right\|_{L^p } \\
		 & \lesssim \sup_{P\in \Q} |P|^{-\tau} \left\| \left(  \sum_{j\ge -\scale (P)}  |\det A|^{j q (s +1/2)} \sum_{Q\in \Q_j} w_Q ^q  \chi_{E_Q} \right) ^{1/q}  \right\|_{L^p } \\
		 & \le 	\left\| \left\{ |\det A|^{j (s +1/2)} \sum_{Q\in \Q_j} w_Q \chi_{E_Q} \right \}_{j\in \mathbb Z } \right\|_{ L \dot F _{ p,q }^\tau }.
\end{align*}
Thus we obtain the desired result.
\end{proof}

\begin{theorem} \label{theorem seq W approx A_Q}
	 Let $A $ be a dilation. 	
	 Let $s\in \mathbb R$, $\tau \in [0,\infty) $, $p\in  (0,\infty)$,
	$q \in (0,\infty]$.
	Let  $W \in \A_p$ and $\{ A_Q\}_{Q\in \Q}$ be a sequence of reducing operators of order $p$ for $W$.
	Then for any $\vec t =\{\vec t _Q \}_{Q\in \Q} \subset \mathbb C^m$
	\begin{equation*}
		\| \vec t \|_{ \dot a_{p,q}^{s,\tau} (W,A) }  \approx 	\| \vec t\|_{\dot a_{p,q}^{s,\tau} (\{A_Q\}_{Q\in \Q},A)  } ,
	\end{equation*}
where the positive equivalence constants are independent of $\vec t$.	
\end{theorem}

\begin{proof}
	Using (\ref{eq reducing matrix}), we obtain 	
	\begin{equation*}
		\| \vec t \|_{\dot b_{p,q}^{s,\tau} (\{A_Q\}_{Q\in \Q} ,A) } \approx	\| \vec t\|_{\dot b_{p,q}^{s,\tau} (W ,A) } .
	\end{equation*}	
It remains to show  $	\| \vec t \|_{ \dot f_{p,q}^{s,\tau} (\{A_Q\}_{Q\in \Q} ,A) } \approx	\| \vec t\|_{\dot f_{p,q}^{s,\tau} (W ,A) } $.
We first prove 
\begin{equation} \label{eq seq A_Q le W}
		\| \vec t \|_{ \dot f_{p,q}^{s,\tau} (\{A_Q\}_{Q\in \Q} ,A) } \lesssim	\| \vec t\|_{\dot f_{p,q}^{s,\tau} (W ,A) } .
\end{equation}
To do so, we consider it into  two cases  on $p$.

	Case $ p \in (0,1] $. 
	By Lemma \ref{lemma reverse A_Q W}, for any $j \in \mathbb Z$, and a.e. $x\in \rn$,
	\begin{equation*}
		| A_Q  \vec t_{Q}|  \le \|   A_Q  W^{-1/p} (x) \|  | W^{1/p} (x)  \vec t_{Q}|   \lesssim | W^{1/p} (x)  \vec t_{Q}|   .
	\end{equation*}
Thus we obtain (\ref{eq seq A_Q le W}) in the case $p \in (0,1]$.
	
	Case $p \in (1,\infty)$. By Lemma \ref{lemma reverse A_Q W} (iv), there exists  a constant $C>0$ such that
	\begin{equation} \label{eq A_Q W^ -1/p le 1}
		\sup_{Q\in \Q} \fint_Q \|A_Q  W^{-1/p} (x) \| \d x \le C.
	\end{equation}
For any $Q\in \Q$,  let  $E_Q : = \{x \in Q: \|A_Q  W^{-1/p} (x) \| \le   2C \}.$
By Chebyshev's inequality, and (\ref{eq A_Q W^ -1/p le 1}), we infer that, for any $Q\in \Q$,
\begin{equation*}
	|Q \backslash E_Q |\le \frac{1}{2C} \int_{Q \backslash E_Q } \|A_Q  W^{-1/p} (x) \| \d x \le \frac{1}{2|Q|}.
\end{equation*}
Hence $ |E_Q |\ge |Q|/2 $. By Lemma \ref{lemma sparse char dot f},
\begin{align*}
	\| \vec t \|_{ \dot f_{p,q}^{s,\tau} (\{A_Q\}_{Q\in \Q} ,A)  } &  =  	\left\| \left\{ |\det A|^{j (s +1/2)} |\mathbb A_j  \vec s_{Q,j} |\right \}_{j\in \mathbb Z } \right\|_{ L \dot F _{ p,q }^\tau }
	\\
	& \approx 	\left\| \left\{ |\det A|^{j (s +1/2)}    \sum_{Q\in \Q_j}  |A_Q \vec t_Q | \chi_{E_Q} \right \}_{j\in \mathbb Z } \right\|_{ L \dot F_{ p,q }^\tau }
	\\
	&\lesssim \left\| \left\{ |\det A|^{j (s +1/2)}    \sum_{Q\in \Q_j}  |W^{1/p}\vec t_Q | \chi_{E_Q} \right \}_{j\in \mathbb Z } \right\|_{ L \dot F_{ p,q }^\tau }
	\\
		& \le 	\left\| \left\{ |\det A|^{j (s +1/2)}    \sum_{Q\in \Q_j}  |W^{1/p} \vec t_Q |  \right \}_{j\in \mathbb Z } \right\|_{ L \dot F_{ p,q }^\tau } = 	\| \vec t \|_{ \dot f_{p,q}^{s,\tau } (W ,A)  } ,
\end{align*}
where $\mathbb A_j$ is  the same as in (\ref{eq def A_j}).
Thus we obtain (\ref{eq seq A_Q le W}) in the case $p \in (1,\infty)$.

 Next we prove
\begin{equation} \label{eq seq W le A_Q}
		\| \vec t\|_{\dot f_{p,q}^{s,\tau} (W ,A) }  \lesssim \| \vec t \|_{ \dot f_{p,q}^{s,\tau} (\{A_Q\}_{Q\in \Q} ,A) }.
\end{equation}
For each $j \in \mathbb Z$, set
\begin{equation*}
	f_j (x): = \sum_{Q \in \Q_j} |\det A|^{j (s +1/2)} |A_Q \vec t_Q | \chi_Q (x)
\end{equation*}
and 
\begin{equation*}
	g_j (x):= \sum_{Q \in \Q_j} |\det A|^{j (s +1/2)} |W^{1/p} (x) \vec t_Q | \chi_Q  (x) .
\end{equation*}
Note that $f_j$ is a constant on each $Q\in \Q_j$.
Then 
\begin{equation*}
	g_j (x) \le \sum_{Q \in \Q_j} |\det A|^{j (s +1/2)} |W^{1/p} (x) A_Q^{-1} \| \| A_Q \vec t_Q | \chi_Q  (x) = \gamma_j  (x) f_j (x),
\end{equation*}
where $\gamma_j$  is the same as in (\ref{eq gamma_j W a_q^ -1}). From this and Corollary \ref{cor gamma_j  W A_Q LF}, we obtain
\begin{align*}
		\| \vec t\|_{\dot f_{p,q}^{s,\tau} (W ,A) } 
		 = \| \{ g_j \}_{j \in \mathbb Z} \|_{ L \dot F_{p.q}^\tau } \le \| \{\gamma_j   f_j \}_{j \in \mathbb Z} \|_{ L \dot F_{p.q}^\tau } \lesssim \| \{  f_j \}_{j \in \mathbb Z} \|_{ L \dot F_{p.q}^\tau } 
	= \| \vec t \|_{ \dot f_{p,q}^{s,\tau} (\{A_Q\}_{Q\in \Q} ,A) }.
\end{align*}
This finishes the proof of (\ref{eq seq W le A_Q})  and Theorem \ref{theorem seq W approx A_Q}.
\end{proof}

\subsection{The $\varphi$-transform}
Suppose that $\varphi, \psi \in \S (\rn)$ are such that supp $\F \varphi$ and supp $\F \psi$ are compact and bounded away from the origin. 
The $\varphi$-transform was first introduced in \cite{FJ90}.
\begin{definition}
	The $\varphi$-transform $S_\varphi$, often called the analysis transform, is the map taking each $\vec f \in (\S_\infty ^\prime (\rn) )^m $ to the sequence $S_\varphi \vec f = \{ (S_\varphi \vec f )_Q \}_{Q \in \Q} $ defined by $(S_\varphi \vec f )_Q = \langle  \vec f , \varphi_Q \rangle$. This is well defined, since $\int_\rn x^\alpha \varphi_Q (x) \d x =0 $ for any multi-index $\alpha$. Here,  we follow the pairing convention which is consistent with the usual scalar product in $L^2 (\rn)$, i.e., $\langle  \vec f , \varphi \rangle = f (\bar \varphi)$ for $\vec f \in (\S ^\prime  (\rn) )^m$ and $ \varphi \in \S (\rn)$.
	
	The inverse $\varphi$-transform, $T_\psi$, often called the synthesis transform, is the map taking the sequence $ \vec s =\{ \vec s_Q\}_{Q\in \Q}$ to  $T_\psi \vec s = \sum_{Q\in \Q} \psi_Q  \vec s_Q $.
\end{definition}

For $\psi \in \mathcal S (\rn)$, let
$\|\psi\|_{\S_N} := \sup_{x\in \rn} \sup_{|\gamma|\le N}  (1+|x|)^N |\partial ^\gamma \psi (x)| $.

\begin{lemma}[(3.18), \cite{Bo07}] \label{lemma psi_Q varphi_P}
	Let $\varphi, \psi \in \S_\infty  (\rn)$.	
	For any $L>0$, there exist positive constants $N \in\mathbb N$ and $ C $ such that for all $P,Q\in \Q$,
	\begin{equation*}
		|\left\langle  \psi_Q , \varphi_P \right\rangle | \le C \|\psi\|_{\S_N} \|\varphi\|_{\S_N}  \left( 1+ \frac{\rho_A (x_Q -x_P  )}{ |P| \vee |Q|  } \right) ^{-L}  \min \left( \frac{|Q|}{|P|},   \frac{|P|}{|Q|} \right)^L.
	\end{equation*}
	Here the constant $C$ depends only on $L$.
\end{lemma}

The following lemma shows that the inverse $\varphi$-transform $T_\psi$ is well defined for any $ \vec s \in \dot a_{p,q}^{s,\tau} (\{A_Q\}_{Q\in \Q},A)   $.

\begin{lemma} \label{lemma well defined inverse varphi transform}
 Let $A $ be a dilation. 
Let  $W \in \A_p$ and $\{ A_Q\}_{Q\in \Q}$ be a sequence of reducing operators of order $p$ for $W$.
 Let $s\in \mathbb R$, $\tau \in [0,\infty) $, $p\in  (0,\infty)$,
$q \in (0,\infty]$. 
Let
\begin{equation*}
		L > \max\{ \tilde  d / p'  - \tau  - s + 1/2 + 1/p , d/p  + \tau  + s + 1/2 - 1/p, \Delta + 1 \}
\end{equation*}
where $\tilde d, d, \Delta $  are the same as in Lemma \ref{lemma cube AQ AR-1}.
Then there exist positive constants $N \in \mathbb N$ and $ C $ such that for any $\varphi, \psi \in \S_\infty  (\rn)$
\begin{equation*}
		\sum_{Q\in \Q} |  \langle \psi_Q, \varphi\rangle | | \vec s_Q| \le C  \| \vec s\|_{\dot a_{p,q}^{s,\tau} (W,A)}   \|\psi\|_{\S_N} \|\varphi\|_{\S_N}.
\end{equation*}
\end{lemma}

\begin{proof}
	Recall that 
	\begin{equation*}
		\| \vec s\|_{\dot f_{p,q}^{s,\tau} (\{A_Q\}_{Q\in \Q},A)} = \sup_{P \in \Q} |P|^{-\tau} \left\| \left(\sum_{ j \ge -\scale (P)} |\det A|^{jq (s+1/2)}  \sum_{Q \in \Q_j} |A_Q \vec s_Q |^q  \chi_Q   \right)^{1/q}  \right\|_{L^p (P)}    
	\end{equation*}
and 
	\begin{equation*}
	\| \vec s\|_{\dot b_{p,q}^{s,\tau} (\{A_Q\}_{Q\in \Q},A)} = \sup_{P \in \Q} |P|^{-\tau}  \left(\sum_{ j \ge -\scale (P)} |\det A|^{jq (s+1/2)}  \left\| \sum_{Q \in \Q_j} |A_Q \vec s_Q |  \chi_Q    \right\|_{L^p (P)} ^q    \right)^{1/q} .
\end{equation*}
	By Theorem \ref{theorem seq W approx A_Q}, we obtain, for each $Q\in \Q$, 
	\begin{align*}
		| \vec s_Q |  & \le \|A_Q ^{-1} \|  |A_Q \vec s_Q| \le  \|A_Q ^{-1} \|  |Q|^{\tau +  s +1/2  -1/p}   \| \vec s\|_{\dot a_{p,q}^{s,\tau} (\{A_Q\}_{Q\in \Q},A)} 
		\\
		&   \approx \|A_Q ^{-1} \|  |Q|^{\tau +  s +1/2  -1/p}  \| \vec s\|_{\dot a_{p,q}^{s,\tau} (W,A)} .
	\end{align*}
Hence for any $\varphi \in \S_\infty (\rn)$, 
\begin{equation*}
	\sum_{Q\in \Q} |  \langle \psi_Q, \varphi\rangle | | \vec s_Q| \lesssim  \| \vec s\|_{\dot a_{p,q}^{s,\tau} (W,A)}  \sum_{Q\in \Q}  |Q|^{ \tau + s+ 1/2 -1/p } \|A_Q ^{-1} \|   | \langle \psi_Q, \varphi \rangle|.
\end{equation*}
From Lemma \ref{lemma cube AQ AR-1}, we have
\begin{equation*}
	\| A_Q^{-1} \| \le \|A_{  [0,1)^n } ^{-1}\| \|    A_{  [0,1)^n }  A_Q^{-1} \|  \lesssim     \max \left( |Q|  ^{d/p}  , |Q| ^{ - \tilde  d/p'}  \right) \left( 1 +  \frac{ \rho_A ( c_{Q}   ) }{ |Q|  \vee 1 } \right)^{ \Delta }.
\end{equation*}
From Lemma \ref{lemma psi_Q varphi_P},
  we have that for any $\varphi,\psi \in  \S_\infty (\rn) $  and  $Q \in \Q $,
  \begin{align*}
  	| \langle \psi_Q, \varphi \rangle| = 	| \langle \psi_Q,\varphi_{[0,1)^n } \rangle| \lesssim  \|\psi\|_{\S_N} \|\varphi\|_{\S_N}  \left( 1+ \frac{\rho_A (x_Q   )}{ 1 \vee |Q|  } \right) ^{-L}  \min \left( |Q|,   |Q|^{-1}  \right)^L .
  \end{align*}
Hence
\begin{align*}
		\sum_{Q\in \Q} |  \langle \psi_Q, \varphi\rangle | | \vec s_Q| &  \lesssim  \| \vec s\|_{\dot a_{p,q}^{s,\tau} (W,A)}   \|\psi\|_{\S_N} \|\varphi\|_{\S_N}  \sum_{Q\in \Q}  |Q|^{ \tau + s+ 1/2 -1/p } \\
		& \quad \times   \min \left( |Q|  ^{d/p- L}  , |Q| ^{  L-\tilde  d/p'}  \right) \left( 1 +  \frac{ \rho_A ( c_{Q}   ) }{ |Q|  \vee 1 } \right)^{ -L + \Delta }   .
\end{align*}
Since 	$L > \max\{ \tilde  d / p'  - \tau  - s + 1/2 + 1/p , d/p  + \tau  + s + 1/2 - 1/p, \Delta + 1 \}$, we have
\begin{align*}
&	\sum_{j\ge 0} | \det A |^{ (-j) ( \tau + s+ 1/2 -1/p)  }  |\det A |^{-j (L - \tilde  d / p')}  \sum_{Q\in \Q_j }    \left( 1 +   \rho_A ( c_{Q} )  \right)^{ -L + \Delta } \\
& \lesssim 	\sum_{j\ge 0} | \det A |^{ (-j) ( \tau + s+ 1/2 -1/p)  }  |\det A |^{-j (L - \tilde  d / p')}   \sum_{k \in \mathbb Z^n} ( 1+ |\det A|^{-j}  \rho_A (k) ) ^{-L+\Delta} \\     
& \lesssim 	\sum_{j\ge 0} | \det A |^{ (-j) ( \tau + s+ 1/2 -1/p)  }  |\det A |^{-j (L - \tilde  d / p')}   |\det A|^j     <\infty,
\end{align*}
and 
\begin{align*}
 &	\sum_{j< 0} | \det A |^{ (-j) ( \tau  + s+ 1/2 -1/p)  }  |\det A |^{-j (d/p - L )}  \sum_{Q\in \Q_j }    \left( 1 +   |\det A|^j  \rho_A ( c_{Q} )  \right)^{ -L + \Delta } \\
 & \lesssim 	\sum_{j< 0} | \det A |^{ (-j) ( \tau  + s+ 1/2 -1/p)  }  |\det A |^{-j (d/p - L )}  \sum_{k \in \mathbb Z^n} (1+ \rho_A  (k)) ^{-L+\Delta} \\ 
	&\lesssim \sum_{j< 0} | \det A |^{ (-j) ( \tau  + s+ 1/2 -1/p)  }  |\det A |^{-j (d/p - L )}  <\infty .
\end{align*}
Thus the proof is complete.
\end{proof}

\begin{lemma}[Lemma 3.3, \cite{BH06}] \label{lemma seq peetre besov}
	Let $A $ be a dilation. Let $s\in \mathbb R$, $q\in (0,\infty] $, $p\in  (0,\infty)$. Then for any $r>0$
	and all $  u  = \{ u_Q\}_{Q\in \Q}$,
	
	{\rm (i)} if
	  $\lambda  >\max ( 1, r/q, r/p ) $, there exists a constant $C>0$ such that 
	\begin{equation*}
			\| u  \|_{ \dot f_{p,q}^{s} (A) } \le 	\| u ^\ast_{ r, \lambda} \|_{ \dot f_{p,q}^{s} (A) } \le C 	\| u  \|_{ \dot f_{p,q}^{s} (A) } ;
	\end{equation*}

	{\rm (ii)} if
  $\lambda  >\max ( 1,  r/p ) $, there exists a constant $C>0$ such that 
\begin{equation*}
	\| u  \|_{ \dot b_{p,q}^{s} (A) } \le 	\| u ^\ast_{ r, \lambda} \|_{ \dot b_{p,q}^{s} (A) } \le C	\| u  \|_{ \dot b_{p,q}^{s}(A) } .
\end{equation*}
\end{lemma}

\begin{remark}
	From the proof of \cite[Lemma 3.3]{BH06}, it is easy to obtain result of anisotropic Besov spaces.
\end{remark}

\begin{lemma} \label{lemma unweighted sequence morrey peetre}
	 Let $A $ be a dilation. 
	Let  $W \in \A_p$ and $\{ A_Q\}_{Q\in \Q}$ be a sequence of reducing operators of order $p$ for $W$.
 Let $s\in \mathbb R$, $\tau \in [0,\infty) $, $p\in  (0,\infty)$,
$q \in (0,\infty]$.  
	Then the following statements hold:
	
	{\rm (i)} 
		 if  $\lambda  >\max ( 1, r/q, r/p ) $, then 
	\begin{equation*}
		\| u  \|_{ \dot f_{p,q}^{s,\tau} (A) } \le 	\| u ^\ast_{ r, \lambda} \|_{ \dot f_{p,q}^{s,\tau} (A) } \lesssim 	\| u  \|_{ \dot f_{p,q}^{s,\tau}  (A) } ;
	\end{equation*}

	{\rm (ii)} if
	$\lambda  >\max ( 1,  r/p ) $, then 
\begin{equation*}
	\| u  \|_{ \dot b_{p,q}^{s,\tau}  (A) } \le 	\| u ^\ast_{ r, \lambda} \|_{ \dot b_{p,q}^{s,\tau}  (A) } \lesssim 	\| u  \|_{ \dot b_{p,q}^{s,\tau}  (A) } .
\end{equation*}	
\end{lemma}

\begin{proof}
	Recall that 
		\begin{equation*}
		(u ^\ast_{ r, \lambda}  )_Q :=  \left( \sum_{R\in \Q, |R|=|Q| }  \frac{|u_R |^r   }{ (1+ |Q|^{-1} \rho_A (x_Q -x_R )  ) ^\lambda }   \right)^{1/r} .
	\end{equation*}

Fix a dilated cube $P \in \Q$.  Let  $ \tilde P := \{ x\in \rn : \rho_A (x - c_P)  < |\det A|^3 |P|  \} $. 
For $ Q\in \Q$, let $w_Q = u_Q$ if $Q \subset \tilde P$ and $w_Q =0 $ otherwise, and let $v_Q =  u_Q - w_Q$. Set $w= \{ w_Q\}_{Q\in \Q} $ and $ v = \{ v_Q\}_{Q\in \Q}$.
Then for all $Q$, we  have
\begin{equation*}
	(  u ^\ast_{  r , \lambda} )_Q^{r }  =  	(  w ^\ast_{  r , \lambda} )_Q^{r }   + 	(  v^\ast_{  r , \lambda} )_Q^{r }   .
\end{equation*}
	Using Lemma \ref{lemma seq peetre besov}, if $\lambda  >\max ( 1, r/q, r/p ) $,
we have
\begin{align*}
	I_P &: = |P|^{-\tau} \left\| \left( \sum _{ j \ge -\scale (P)} |\det A|^{j q(s +1/2)} \sum_{Q \in \Q_j}   	(  w ^\ast_{  r , \lambda} )_Q ^q  \chi_Q      \right) ^{1/q}      \right\|_{L^p (P)} \\
	&\le |P|^{-\tau} \| w ^\ast_{ r, \lambda} \|_{ \dot f_{p,q}^{s} (A) }  \\
	&\lesssim |P|^{-\tau} \| w  \|_{ \dot f_{p,q}^{s} (A) }  \\
	& \lesssim \| u  \|_{ \dot f_{p,q}^{s,\tau} }.
\end{align*}

On the other hand, let $ \tilde {  \tilde P } := \{ x\in \rn : \rho_A (x - c_P)  < |\det A|  |P|  \}.$
Let $P(k)  $  be the dilate cube such that $|P(k)| =|P| $ and $ x_P - x_{ P (k) }  = A^{\scale (P)} k $ for $k \in \mathbb Z^n$.

For $j \in \{  -\scale (P), -\scale (P) + 1, \ldots \}$, $k \in \mathbb Z^n \backslash \{0\}$,
 let 
 \begin{equation*}
 	A (j,k,P) = \{ R\in \Q_j :R \preceq P(k) , | R \cap P(k) |>0 ,  R \cap \tilde {  \tilde P }  = \emptyset \}.
 \end{equation*}
Then for $ R \in  A (j,k,P)$ and $Q \preceq P, Q \in \Q_j$, we have
\begin{equation*}
	1+ |R|^{-1} \rho_A (x_Q - x_R)  \approx 1 +   |P| |R|^{-1} \rho_A (k)  \approx |P| |R|^{-1} \rho_A (k).
\end{equation*}
Let $0 <a \le r$ and let $0< a < \min \{ p,q\} $. 
If $r \ge  \min \{ p,q\}$, since $\lambda  >\max ( 1, r/q, r/p )   = \frac{r}{ \min (r, q,p ) }$, choose $a$ sufficiently close $\min \{ p,q\}$ such that  $\lambda \ge r /a$.
If  $r <   \min \{ p,q\}$, since $\lambda  >\max ( 1, r/q, r/p )   = \frac{r}{ \min (r, q,p ) }$, choose $a$ sufficiently close $r  $  such that $a < \min \{ p,q\}$  and  $\lambda \ge r /a$.

Thus, for $j \in \{  -\scale (P), -\scale (P) + 1, \ldots \}$, $k \in \mathbb Z^n \backslash \{0\} $, we obtain
\begin{align*}
	&	\sum_{R \in 	A (j,k,P)}  \frac{|v_R |^r   }{ (1+ |R|^{-1} \rho_A (x_Q -x_P )  ) ^\lambda } \\
	& \lesssim (  |P| |R|^{-1} \rho_A (k))^{-\lambda} \sum_{R \in 	A (j,k,P)}  |v_R |^r  \\
	& \le (  |P| |R|^{-1} \rho_A (k))^{-\lambda} \left(  \sum_{R \in 	A (j,k,P)}  |v_R |^a  \right)^{r/a}  \\
	& = (  |P| |R|^{-1} \rho_A (k))^{-\lambda}  |\det A|^{ j   r /a  }  \left(  \sum_{R \in 	A (j,k,P)} \int_R |v_R |^a \chi_R \d x  \right)^{r/a} \\
	& \lesssim  (|P| |R|^{-1} \rho_A (k))^{-\lambda}  |\det A|^{ j  r /a  } |P| ^{r/a} \M_{\rho_A}    \left( \sum_{R \in 	A (j,k,P)} |v_R |^a \chi_R  \right) ^{r/a} (x_{P(k)})  \\
	& =  |P|^{ - \lambda +r/a }    |\det A|^{ -j \lambda + j  r /a }  \rho_A (k)^{-\lambda}  \M_{\rho_A}    \left( \sum_{R \in 	A (j,k,P)} |v_R |^a \chi_R  \right) ^{r/a} (x_{P(k)})   \\
	& \le   \rho_A (k)^{-\lambda}  \M_{\rho_A}    \left( \sum_{R \in 	A (j,k,P)} |v_R |^a \chi_R  \right) ^{r/a} (x_{P(k)}) 
	.
\end{align*}

Hence
\begin{align*}
|Q|^{ -s -1/2 }	(v ^\ast_{ r, \lambda}  )_Q  & :=  \left( \sum_{R\in \Q, |R|=|Q| }   \frac{ |R|^{ (-s -1/2) r } |v_R |^r   }{ (1+ |R|^{-1} \rho_A (x_Q -x_P )  ) ^\lambda }   \right)^{1/r} \\
	& \lesssim  \left( \sum_{k \in \mathbb Z^n \backslash \{0\} }  \sum_{R \in 	A (j,k,P)} \frac{ |R|^{ (-s -1/2) r } |v_R |^r   }{ (1+ |Q|^{-1} \rho_A (x_Q -x_P )  ) ^\lambda }   \right)^{1/r} \\
	& \lesssim   \left( \sum_{k \in \mathbb Z^n \backslash \{0\}}   \rho_A (k)^{-\lambda}  \M_{\rho_A}    \left( \sum_{R \in 	A (j,k,P)} |R|^{ (-s -1/2) a}  |v_R |^a \chi_R  \right) ^{r/a} (x_{P(k)})  \right)^{1/r}   .
\end{align*}
By Minkowski's inequality, Lemma \ref{lemma M rho A FS} and $\lambda >1$,
\begin{align*}
J_P := 	& |P|^{-\tau}  \left\| \left( \sum _{ j \ge -\scale (P)} |\det A|^{j q(s +1/2)} \sum_{Q \in \Q_j}   (v ^\ast_{r, \lambda } )_Q ^q  \chi_Q      \right) ^{1/q}      \right\|_{L^p (P)} \\
	&\lesssim |P|^{-\tau}     \left\| \left( \sum _{ j \ge -\scale (P)}     \left( \sum_{k \in \mathbb Z^n}   \rho_A (k)^{-\lambda}  \M_{\rho_A}    \left( \sum_{R \in 	A (j,k,P)} |R|^{ (-s -1/2) a}  |v_R |^a \chi_R  \right) ^{r/a}   \right)^{q/r}          \right) ^{1/q}      \right\|_{L^p (P)} \\
	&= |P|^{-\tau}  \left\| \left( \sum _{ j \ge -\scale (P)}     \left( \sum_{k \in \mathbb Z^n}   \rho_A (k)^{-\lambda}  \M_{\rho_A}    \left( \sum_{R \in 	A (j,k,P)} |R|^{ (-s -1/2) a}  |v_R |^a \chi_R  \right)^{r/a}  \right)^{q/r}          \right) ^{r/q}      \right\|_{L^{p/r} (P)} ^{1/r} \\
		&\le |P|^{-\tau}  \left( \sum_{k \in \mathbb Z^n}   \rho_A (k)^{-\lambda}  \left\| \left( \sum _{ j \ge -\scale (P)}     \left(   \M_{\rho_A}    \left( \sum_{R \in 	A (j,k,P)} |R|^{ (-s -1/2) a}  |v_R |^a \chi_R  \right) ^{r/a}  \right)^{q/r}          \right) ^{r/q}      \right\|_{L^{p/r} (P)}  \right) ^{1/r} \\
			&\lesssim |P|^{-\tau}  \left( \sum_{k \in \mathbb Z^n}   \rho_A (k)^{-\lambda}  \left\| \left( \sum _{ j \ge -\scale (P)}        \sum_{R \in 	A (j,k,P)} |R|^{ (-s -1/2)q}  |v_R |^q \chi_R             \right) ^{r/q}      \right\|_{L^{p/r}}  \right) ^{1/r} \\
				&\lesssim |P|^{-\tau}  \left( \sum_{k \in \mathbb Z^n}   \rho_A (k)^{-\lambda}  |P|^{\tau r } \|u\|_{\dot f_{p,q}^{s,\tau} (A) } ^r  \right) ^{1/r} \\
				& \lesssim \|u\|_{\dot f_{p,q}^{s,\tau} (A) }  .
\end{align*}

Therefore,
\begin{equation*}
	\| u ^\ast_{ r, \lambda} \|_{ \dot f_{p,q}^{s,\tau} (A) } \lesssim   \sup_{P \in \Q}  ( I_P +J_P ) \lesssim 	\| u  \|_{ \dot f_{p,q}^{s,\tau} (A) } .
\end{equation*}

(ii) The proof of (ii) is similar and we omit it here. 
\end{proof}

\begin{lemma} \label{lemma matrix weighted sequence morrey peetre}
	Let $A $ be a dilation. 
	Let  $W \in \A_p$ and $\{ A_Q\}_{Q\in \Q}$ be a sequence of reducing operators of order $p$ for $W$.
 Let $s\in \mathbb R$, $\tau \in [0,\infty) $, $p\in  (0,\infty)$,
$q \in (0,\infty]$.  
	Then

	{\rm (i)} 	if  $\lambda  >\max ( 1, r/q, r/p ) $, then 
	\begin{equation*}
		\| \vec u  \|_{ \dot f_{p,q}^{s,\tau}  (W, A)} \approx \left \| \left( \{  |A_Q \vec u_Q \}_{Q\in \Q} \right) ^\ast_{ r, \lambda}     \right \|_{ \dot f_{p,q}^{s,\tau} (A) } ;
	\end{equation*}
	
	{\rm (ii)} if
	$\lambda  >\max ( 1,  r/p ) $, then 
	\begin{equation*}
	\| \vec u  \|_{ \dot b_{p,q}^{s,\tau}  (W, A)} \approx \left \| \left( \{  |A_Q \vec u_Q \}_{Q\in \Q} \right) ^\ast_{ r, \lambda}     \right \|_{ \dot b_{p,q}^{s,\tau} (A) } .
\end{equation*}	
\end{lemma}
\begin{proof}
	By Theorem \ref{theorem seq W approx A_Q} and  Lemma \ref{lemma unweighted sequence morrey peetre}, we have
	\begin{equation*}
			\| \vec u  \|_{ \dot f_{p,q}^{s,\tau}  (W, A)} \approx 	\| \vec u  \|_{ \dot f_{p,q}^{s,\tau}  ( \{ A_Q\}_{Q\in \Q} , A )}  = \| \{  |A_Q \vec u_Q \}_{Q\in \Q}            \|_{ \dot f_{p,q}^{s,\tau} (A) }  = \left \| \left( \{  |A_Q \vec u_Q \}_{Q\in \Q} \right) ^\ast_{ r, \lambda}     \right \|_{ \dot f_{p,q}^{s,\tau} (A) } .
	\end{equation*}
The proof of (ii) is similar.
Thus we finish the proof of Lemma \ref{lemma matrix weighted sequence morrey peetre}.
\end{proof}

\begin{theorem} \label{theorem wavelet trans}
		Let $A $ be a dilation. 
	 Let $s\in \mathbb R$, $\tau \in [0,\infty) $, $p\in  (0,\infty)$,
	$q \in (0,\infty]$.   
	Let  $W \in \A_p$.	
	Let $\varphi, \psi \in \mathcal S (\rn)$ satisfy (\ref{eq varphi supp}) and (\ref{eq varphi > 0}).
	Let $\tilde \varphi (x): = \overline{\varphi (-x)}$.
	Then the operators
	\begin{equation*}
		S_\varphi :  \dot A_{p,q}^{s,\tau}  (W, A, \tilde \varphi) \to \dot a_{p,q}^{s,\tau}  (W, A)  \; \operatorname{and } \;  T_\psi : \dot a_{p,q}^{s,\tau}  (W, A) \to \dot A_{p,q}^{s,\tau}  (W, A,  \varphi) 
	\end{equation*}
are bounded. Moreover, if $\varphi$ and $\psi $  satisfy (\ref{eq varphi psi =1}), then $T_\psi \circ S_\varphi $  is the identity on $ \dot A_{p,q}^{s,\tau}  (W, A, \tilde \varphi)$. 
\end{theorem}

\begin{proof}
	Let $\mathbb A:= \{ A_Q\}_{Q\in \Q}$ be a sequence of reducing operators of order $p$ for $W$.
	For any $\vec f \in  \dot A_{p,q}^{s,\tau}  (W, A) $, let 
\begin{equation*}
	\sup_{ \mathbb A ,\tilde \varphi} \vec f := \left\{  \sup_{ \mathbb A , \tilde \varphi, Q  } \vec f \right\}_{Q \in \Q}, 
\end{equation*}
where 
\begin{equation*}
	\sup_{ \mathbb A ,  \tilde \varphi, Q  } \vec f := \sup_{y\in Q} |A_Q   \tilde \varphi_Q * \vec f (y) |  =  |Q|^{1/2} \sup_{y\in Q} |A_Q  \tilde  \varphi_{ -\scale (Q) } * \vec f (y) |  .
\end{equation*}
For any $Q \in \Q_j$,
\begin{equation*}
	|A_Q ( S_\varphi \vec f )_Q | \le | A_Q \langle \vec f , \varphi_Q \rangle | = |Q|^{1/2} | A_Q   \tilde \varphi_{ j}  *\vec f  ( x_Q) | \le  \sup_{ \mathbb A ,  \tilde \varphi, Q  }  \vec f  .
\end{equation*}
By Theorems \ref{theorem W app a app A_Q} and \ref{theorem seq W approx A_Q}, we obtain 
\begin{equation*}
	\| S_\varphi \vec f \|_{  \dot a_{p,q}^{s,\tau}  (W, A) } \approx  	\| S_\varphi \vec f \|_{  \dot a_{p,q}^{s,\tau}  (  \{ A_Q \}_{ Q \in \Q}, A) } \le \left\|  \sup_{ \mathbb A ,  \tilde \varphi, Q  } \vec f  \right\|_{  \dot a_{p,q}^{s,\tau}  ( A) }  \approx \left\|   \vec f  \right\|_{  \dot A_{p,q}^{s,\tau}  (W, A, \tilde \varphi ) } .
\end{equation*}

Now we prove the boundedness of $T_\psi$.
Since the supports of $ \F \varphi$, $\F \psi$ are bounded and bounded away from the origin, there exists an integer $M$  such that supp  $ \F \varphi_j \cap $  supp $\F \psi_i = \emptyset$ for $|i-j| >M$. 
For any  $j \in \mathbb Z $, $Q \in \Q_j$  and $x\in Q$, 
\begin{align*}
	|  A_Q \varphi_j * (T_\psi  \vec t ) ( x) |   & = \left|  \sum_{ i  = j-M }^{j+M}   \sum_{P \in \Q_i}  A_Q \vec t_P  ( \varphi_j * \psi_P )  (x) \right| \\
	&  \le  \sum_{ i  = j-M }^{j+M}   \sum_{P \in \Q_i}  \|A_Q A_P^{-1} \|   |   A_P \vec t_P  |    |( \varphi_j * \psi_P )  (x) | .
\end{align*}
From \cite[p 1482]{BH06}, for any $ L>1$, there is a constant $C = C (L) >0$  such that 
\begin{equation*}
|	\varphi_j * \psi_P  (x)| \le C |P|^{-1/2}  (  1+ \rho_A ( A^i (x- x_P)  ) ) ^{-L}
\end{equation*}
for $P \in \Q_i$  and $ |i-j| \le M$.  
By Lemma  \ref{lemma cube AQ AR-1}, for any $j \in \mathbb Z $, $ |i-j|\le M $, $Q \in \Q_j$, $R\in \Q_i$, 
\begin{align*}
		\| A_{Q} A_{P}  ^{-1} \|  
		& \lesssim  \left( 1 +  \frac{ \rho_A ( c_{Q} -  c_{P}  ) }{  |P| } \right)^{ \Delta } ,
\end{align*}
where $\Delta$ is the same as in Lemma  \ref{lemma cube AQ AR-1}.
Given any $x\in \rn$,    for each $i \in \{ j-M , j-M +1, \ldots, j+M  \}$,   let  $Q^{ (i)} $ be the unique dilated cube such that $x \in Q^{(i)}$ and $Q^{ (i)}  \in \Q_i $.
Thus, for $x\in Q  \in \Q_j $,
\begin{align*}
	|  A_Q \varphi_j * (T_\psi  \vec t ) ( x) |  
	&  \le  \sum_{ i  = j-M }^{j+M}   \sum_{P \in \Q_i}  \|A_Q A_P^{-1} \|   |   A_P \vec t_P  |    |( \varphi_j * \psi_P )  (x) | \\
	& \lesssim  \sum_{ i  = j-M }^{j+M}   \sum_{P \in \Q_i}  \left( 1 +  \frac{ \rho_A ( c_{Q} -  c_{P}  ) }{  |P| } \right)^{ \Delta }  |   A_P \vec t_P  |  |P|^{-1/2}  (  1+ \rho_A ( A^i (x- x_P)  ) ) ^{-L} \\
	& \approx  \sum_{ i  = j-M }^{j+M}   \sum_{P \in \Q_i} |P|^{-1/2} \frac{  |  A_P \vec t_P   | }{ (  1+ \rho_A ( A^i (x_Q - x_P)  ) ) ^{L - \Delta }  } \\
	& \lesssim  \sum_{ i  = j-M }^{j+M}  |\det A|^{ i /2 }   \left (  \left( \{  |A_R \vec t_R \}_{R\in \Q} \right) ^\ast_{ 1, L-\Delta}    \right)_{Q^{ (i)}}  .
\end{align*}
By choosing $  L -\Delta >  \max \{ 1, 1/ q, 1/p \} = 1/\min\{1,p,q \}  $, Lemma \ref{lemma matrix weighted sequence morrey peetre} yields
\begin{align*}
		\| T_\psi \vec t \|_{ \dot A_{p,q}^{s,\tau}  ( \{ A_Q\}_{ Q\in \Q}, A, \varphi )} \lesssim   \left\| \left( \{  |A_R \vec t_R \}_{R\in \Q} \right) ^\ast_{ 1, L-\Delta}  \right\|_{ \dot a_{p,q}^{s,\tau} } \lesssim   \| \vec t  \|_{ \dot a_{p,q}^{s,\tau}  (W, A)} .
\end{align*}
By Theorem \ref{theorem W app a app A_Q}, we deduce that
\begin{equation*}
    \| T_\psi \vec t \|_{ \dot A_{p,q}^{s,\tau}  ( W, A, \varphi)} \approx 	\| T_\psi \vec t \|_{ \dot A_{p,q}^{s,\tau}  ( \{ A_Q\}_{ Q\in \Q}, A,\varphi  )} \lesssim \| \vec t  \|_{ \dot a_{p,q}^{s,\tau}  (W, A)} .
\end{equation*}
This finishes the proof of boundedness of  $T_\psi$.

Finally, if $\varphi$ and $\psi $  satisfy (\ref{eq varphi psi =1}),
then by Lemma \ref{lemma identity}, we see that
 $T_\psi \circ S_\varphi $  is the identity on $ \dot A_{p,q}^{s,\tau}  (W, A, \tilde \varphi)$. 
\end{proof}

The following result shows that $ \dot A_{p,q}^{s,\tau}  ( W, A) $ is independent of the choice of $\varphi$.
\begin{theorem} \label{theorem independent varphi}
		Let $A $ be a dilation. 
		 Let $s\in \mathbb R$, $\tau \in [0,\infty) $, $p\in  (0,\infty)$,
		$q \in (0,\infty]$.  
	Let  $W \in \A_p$ and $\{ A_Q\}_{Q\in \Q}$ be a sequence of reducing operators of order $p$ for $W$.
	Then $ \dot A_{p,q}^{s,\tau}  ( W, A)$ and  $\dot A_{p,q}^{s,\tau}  (  \{ A_Q\}_{ Q\in \Q}, A)$ are independent of the choice of $\varphi$.
\end{theorem}

\begin{proof}
	Let $\varphi^{(1)}, \varphi^{(2)},  \psi^{(2)} \in \mathcal S(\rn)$ satisfy  (\ref{eq varphi supp}) and  (\ref{eq varphi > 0})  and assume both $\psi^{(2)}$ and
	$\varphi^{(2)}$ satisfy (\ref{eq varphi psi =1}). Then, from both Lemma \ref{lemma identity} and Theorem \ref{theorem wavelet trans}, we infer that, for
	any $\vec f \in \dot A_{p,q}^{s,\tau}  ( W, A, \varphi^{(2)}  )$,
	\begin{align*}
		\| \vec f\|_{\dot A_{p,q}^{s,\tau}  ( W, A, \varphi^{(1)}  )} & = \| T_{\widetilde { \psi ^{(1)}  }  }  \circ S_{ \widetilde {   \varphi^{(2)} } }  \vec f \|_{ \dot A_{p,q}^{s,\tau}  ( W, A, \varphi^{(1)}  ) }  \\
		& \lesssim  \| S_{\widetilde {   \varphi^{(2)} }}  \vec f \|_{ \dot a_{p,q}^{s,\tau}  ( W ) }  \lesssim 	\| \vec f\|_{\dot A_{p,q}^{s,\tau}  ( W, A, \varphi^{(2)}  )} .
	\end{align*}
By symmetry, we also obtain the reverse inequality. 
By Theorem \ref{theorem W app a app A_Q}, $\dot A_{p,q}^{s,\tau}  (  \{ A_Q\}_{ Q\in \Q}, A)$ is also independent of the choice of $\varphi$.
This finishes the proof of Theorem \ref{theorem independent varphi}.
\end{proof}

\begin{proposition} \label{prop F embed S prime}
		Let $A $ be a dilation. 
 Let $s\in \mathbb R$, $\tau \in [0,\infty) $, $p\in  (0,\infty)$,
$q \in (0,\infty]$.  
	Let  $W \in \A_p$. Then 
	$\dot A_{p,q}^{s,\tau}  ( W, A)  \subset ( \S_\infty ^\prime  (\rn) )^m $.
\end{proposition}

\begin{proof}
	Let $\varphi, \psi \in \S (\rn)$  satisfy  (\ref{eq varphi supp}),  (\ref{eq varphi > 0}) and (\ref{eq varphi psi =1}). By  Lemmas \ref{lemma identity} and \ref{lemma well defined inverse varphi transform},
	  and Theorem \ref{theorem wavelet trans}, there exists  constant $N \in \mathbb N $  such that
	\begin{align*}
		| \langle \vec f , \phi \rangle |   = 	| \langle T_\psi \circ S_\varphi \vec f , \phi \rangle | \le 
		\sum_{Q\in \Q}   | (S_\varphi \vec f )_Q |         |\langle \psi_Q , \phi \rangle | \lesssim \| S_\varphi \vec f \|_{\dot a_{p,q}^{s,\tau} (W,A)}   \|\phi\|_{\S_N} \lesssim \| \vec f \|_{\dot A_{p,q}^{s,\tau} (W,A)}   \|\phi\|_{\S_N} .
	\end{align*}
Thus the proof is finished.
\end{proof}

Using Proposition \ref{prop F embed S prime} and an argument similar to that used in the proof of \cite[Proposition 2.3.1]{Gra142}, we obtain the following result; we omit the details.

\begin{proposition} \label{prop complete}
		Let $A $ be a dilation. 
 Let $s\in \mathbb R$, $\tau \in [0,\infty) $, $p\in  (0,\infty)$,
$q \in (0,\infty]$. 
	Let  $W \in \A_p$. 
	Then $\dot A_{p,q}^{s,\tau}  ( W, A) $ is a  complete quasi-norm space.
\end{proposition}

Finally, we have the following lifting property. 
Let $\sigma\in \mathbb R$, $A$ be a dilation and $A^\ast$ be its conjugate transpose.
The lifting operator $ \dot I _\sigma  $  is definded by 
\begin{equation} \label{eq def I sigma}
	\dot I_\sigma f (x) : = \F^{-1} ( \rho_{A^\ast} (\cdot) ^\sigma \F f  ) (x) .
\end{equation}

\begin{proposition}\label{prop shfit}
	Let $A $ be a dilation. 
 Let $s\in \mathbb R$, $\tau \in [0,\infty) $, $p\in  (0,\infty)$,
$q \in (0,\infty]$.  
Let  $W \in \A_p$. 
Let $\sigma\in \mathbb R$. 
Then for any $\vec f \in  (\S_\infty ^\prime (\rn) )^m$, 
\begin{equation*}
	\|\dot I_\sigma \vec f \|_{\dot A_{p,q}^{s - \sigma ,\tau}  ( W, A )  }  \approx \| \vec f \|_{\dot A_{p,q}^{s  ,\tau}  ( W, A )  } .
\end{equation*}
\end{proposition}

\begin{proof}
	Let $A$ be a dilation and $A^\ast$ be its conjugate transpose. Since $\sigma (A ^\ast)  =   \overline { \sigma ( A) } $.  $A^\ast$  is also a dilation. 
	Let $\psi = \F ^{-1}  ( \rho_{A^\ast} (\cdot)^\sigma \F \varphi ) $.
	Let 
	$\psi_j (x) = |\det A|^{j} \psi( A ^j x) $. Then  $\F \psi_j (\xi) = \F \psi ( (A^\ast)^{-j}  \xi)$ and supp $\F \psi_j \subset (A^\ast)^{j} [-\pi,\pi]^n  $. This shows $\psi$ also satisfies (\ref{eq varphi supp}).  
	By Lemma \ref{lemma rho A and |x|}, for any fixed $\xi \in \rn \backslash\{0\} $, there exists $ C >0$ 
	\begin{align*}
		| \F \psi (  (A^\ast) ^j \xi ) | &  =  | \rho_{A^*} ( (A^\ast) ^j \xi )^\sigma    \F \varphi (  (A^\ast) ^j \xi ) | = |\det A^*|^{j\sigma} | \rho_{A^*} ( \xi )^\sigma    \F \varphi (  (A^\ast) ^j \xi ) | \\
		&  \ge C  |\xi|^{c\sigma} |\det A^*|^{j\sigma} |    \F \varphi (  (A^\ast) ^j \xi ) | .
	\end{align*}
	Thus 
	\begin{equation*}  
		\sup_{j \in \mathbb Z}  | \F \psi (  (A^\ast) ^j \xi ) | \ge C
		\sup_{j \in \mathbb Z}  | \F \varphi (  (A^\ast) ^j \xi ) | >0.
	\end{equation*}
Thus $\psi$ satisfies  (\ref{eq varphi > 0}).	
	Note that
	\begin{align*}
		\psi_j * \vec f  & = \F^{-1}  (\F \psi_j \F \vec f  ) = \F^{-1}  (\F \psi ( (A^\ast)^{-j}  \cdot ) \F \vec f  )
		 \\ 
		 &=        \F^{-1}  ( \rho_{A^\ast}  ( (A^\ast)^{-j} \cdot   ) ^\sigma   \F \varphi ( (A^\ast)^{-j}  \cdot ) \F \vec f  ) \\
		 &=  |\det A^\ast |^{-j \sigma }  \F^{-1}  ( \rho_{A^\ast}  (  \cdot   ) ^\sigma   \F \varphi ( (A^\ast)^{-j}  \cdot ) \F \vec f  ) \\
		 & = |\det A  |^{-j \sigma }  \F^{-1}  ( \rho_{A^\ast}  (  \cdot   ) ^\sigma   \F \varphi_j (  \F \vec f  ) .
	\end{align*}
From this, we obtain
\begin{equation*}
	\| \dot I_\sigma \vec f \|_{ \dot A_{p,q}^{s - \sigma ,\tau}  ( W, A, \varphi) } \lesssim 	\|  \vec f \|_{ \dot A_{p,q}^{s  ,\tau}  ( W, A , \psi ) } .
\end{equation*}
 By Theorem \ref{theorem independent varphi}, we obtain
\begin{equation*}
		\| \dot I_\sigma \vec f \|_{ \dot A_{p,q}^{s - \sigma ,\tau}  ( W, A, \varphi) } \lesssim 	\| \vec f \|_{ \dot A_{p,q}^{s  ,\tau}  ( W, A, \varphi) } .
\end{equation*}
On the other hand, we obtain
\begin{equation*}
	\|  \vec f \|_{ \dot A_{p,q}^{s  ,\tau}  ( W, A) }  = 	\| \dot I_{ - \sigma} \dot I_\sigma  \vec f \|_{ \dot A_{p,q}^{s  ,\tau}  ( W, A) }   \lesssim 	\|  \dot I_\sigma \vec f \|_{ \dot A_{p,q}^{s -\sigma ,\tau}  ( W, A) } .
\end{equation*}
This finishes the proof.
\end{proof}

\section{Matrix-weighted anisotropic Triebel-Lizorkin spaces for $p=\infty$} \label{sec p infty}

The Triebel-Lizorkin spaces for $p=\infty$ were introduced in \cite{FJ90} and anisotropic Triebel-Lizorkin spaces for $p=\infty$ with $\rho_A$-doubling measures were introduced in \cite{Bo07}. In this section, we introduce matrix-weighted anisotropic  Triebel-Lizorkin spaces $\dot F _{ \infty ,q } ^{s} (\{A_Q\}_{Q\in \Q},A  ) $ for the endpoint exponent $p=\infty$ and obtain some results corresponding to Section \ref{sec BF spaces}.

\begin{definition}
		Let $A $ be a dilation. 
	Let $s\in \mathbb R$, $p\in (0,\infty)$,
	$q \in (0,\infty]$.  
	Let  $W \in \A_p$  and $\{ A_Q\}_{Q\in \Q}$ be a sequence of reducing operators of order $p$ for $W$. 
		Let $\varphi\in \mathcal S (\rn)$ satisfy (\ref{eq varphi supp}) and (\ref{eq varphi > 0}).
	The homogeneous averaging matrix-weighted anisotropic Triebel-Lizorkin space $\dot F _{ \infty ,q } ^{s} (\{A_Q\}_{Q\in \Q},A  ) $  is the set of all $\vec f\in (\S_\infty ^\prime (\rn) )^m$  such that
		\begin{equation}\label{eq def F A_Q A p infty}
		\|\vec f\|_{\dot F _{ \infty ,q } ^{s} (\{A_Q\}_{Q\in \Q},A  )  } :=  \sup_{P \in \Q} \left( \fint_P   \sum_{j = - \scale (P)} ^\infty  |\det A|^{jsq}  |\mathbb A_j \varphi_j *\vec f (x)| ^q \d x   \right) ^{1/q}<\infty ,
	\end{equation}
where $\mathbb A_j$ is  the same as in (\ref{eq def A_j}).

		The homogeneous averaging matrix-weighted anisotropic Triebel-Lizorkin space $\dot f _{ \infty, q} ^{s} (\{A_Q\}_{Q\in \Q},A  ) $  is the set of all $\vec t = \{ \vec t_Q\}_{ Q\in \Q} \subset (\mathbb C)^m$  such that
	\begin{equation} \label{eq def f A_Q A p infty}
		\|\vec t\|_{\dot f _{ \infty, q} ^{s} (\{A_Q\}_{Q\in \Q},A  ) } :=  \sup_{P \in \Q} \left( \fint_P   \sum_{  Q\in \Q, \scale (Q) \le  \scale (P)}   |Q|^{  (- s -1/2 )q }  |A_Q \vec t_Q| ^q  \chi_Q (x) \d x   \right) ^{1/q}<\infty .
	\end{equation}

	Naturally, if $q=\infty$, then (\ref{eq def F A_Q A p infty})  and (\ref{eq def f A_Q A p infty}) are interpreted as
	\begin{equation*}
	\|\vec f\|_{\dot F _{ \infty , \infty} ^{s} (\{A_Q\}_{Q\in \Q},A  )  } :=  \sup_{j\in \mathbb Z}   |\det A|^{js }  \|\mathbb A_j \varphi_j *\vec f  \|_{L^\infty}  , \quad 
	  \|\vec t\|_{\dot f _{ \infty, \infty } ^{s} (\{A_Q\}_{Q\in \Q},A  ) } :=  \sup_{P \in \Q}  |Q|^{  - s -1/2 } |A_Q \vec t_Q| .
\end{equation*}	
\end{definition}

\begin{remark}
	(i)
	Let  $ A_Q \equiv 1$ for $Q\in \Q$. Then we denote $ \dot f_{\infty,q}^{s} (A) := \dot f _{ \infty, q} ^{s} (\{A_Q\}_{Q\in \Q},A  )  $ for brevity.
	
	(ii) By definition, when $p=q$ and $\tau = 1/q $, we have
	\begin{equation*}
		\dot F _{ \infty ,q } ^{s} (\{A_Q\}_{Q\in \Q},A  )  = \dot F_{q,q}^{s,1/q}  ( \{ A_Q\}_{ Q\in \Q}, A) \; \operatorname{and} \; \dot f _{ \infty, q} ^{s} (\{A_Q\}_{Q\in \Q},A  )  = \dot f_{q,q}^{s,1/q}  ( \{ A_Q\}_{ Q\in \Q}, A) 
	\end{equation*}
where $ \dot F_{q,q}^{s,1/q}  ( \{ A_Q\}_{ Q\in \Q}, A)$ and $\dot f_{q,q}^{s,1/q}  ( \{ A_Q\}_{ Q\in \Q}, A) $  are the Triebel-Lizorkin-type spaces in Definition \ref{def 8 func spaces}.
	
	(iii) To emphasize the dependence on $\varphi$, we will use the notation $\dot F _{ \infty ,q } ^{s} (\{A_Q\}_{Q\in \Q},A ,\varphi )$ for (\ref{eq def F A_Q A p infty}). Later we will show that this definition is independent of $\varphi$.	
\end{remark}

\subsection{A characterization of the homogeneous  anisotropic Triebel-Lizorkin space}
The following characterization is the main result of this subsection, which will be used to show the relation between $	\dot F _{ \infty ,q } ^{s} (\{A_Q\}_{Q\in \Q},A  ) $  and $ \dot F_{q,q}^{s,1/q}  ( \{ A_Q\}_{ Q\in \Q}, A)$ when $p \neq q$.

\begin{proposition}  \label{prop seq p infty char}
	Let $A$ be a dilation. Let $s\in \mathbb R$, $q \in (0,\infty]$.  For each $ u \in (0,\infty) $, $r = \{ r_Q\}_{Q\in \Q} $, 
	\begin{equation}  \label{eq f s q infty equ}
		\sup_{P \in \Q} \left( \fint_P   \left(   \sum_{ \scale(Q) \le \scale (P) }   |Q|^{ - (s -1/2 )q }  |r_Q| ^q   \chi_Q (x) \right)^{u/q}  \d x   \right) ^{1/u}  \approx 	\| r\|_{  \dot f_{\infty,q}^{s} (A)  } .
	\end{equation}
\end{proposition}

We use the idea of  \cite[Corollary 5.7]{FJ90} to prove  Proposition \ref{prop seq p infty char}. We first need some technical lemmas. 
The following lemma is immediately obtained from the definition of $\dot f _{ \infty, q} ^{s} (\{A_Q\}_{Q\in \Q},A  )$
\begin{lemma} \label{lemma p = infty EQ}
	Let $A $ be a dilation. 
	Let $s\in \mathbb R$, $p\in (0,\infty)$,
	$q \in (0,\infty]$.  
	Let  $W \in \A_p$ and $\{ A_Q\}_{Q\in \Q}$ be a sequence of reducing operators of order $p$ for $W$. 	
	Let $\epsilon > 0$. Suppose that for each dilated cube $Q$ there is a set $E_Q \subset Q$ with $|E_Q| /|Q| >\epsilon$. Then for $r = \{ r_Q\} _{Q\in \Q}  \subset \mathbb C $,
	\begin{align*}
		& \sup_{P \in \Q} \left( \fint_P   \sum_{ \scale (Q) \le  \scale (P)}   |Q|^{ - (s -1/2 )q }  |\mathbb A_j \vec t_Q| ^q  \chi_{E_Q} (x) \d x   \right) ^{1/q} \le 	\|\vec t\|_{\dot f _{ \infty, q} ^{s} (\{A_Q\}_{Q\in \Q},A  ) } \\
		& \le \frac{1}{\epsilon}\sup_{P \in \Q} \left( \fint_P   \sum_{ \scale (Q) \le  \scale (P)}   |Q|^{ - (s -1/2 )q }  |\mathbb A_j \vec t_Q| ^q  \chi_{E_Q} (x) \d x   \right) ^{1/q} .
	\end{align*}
\end{lemma}

We now consider an operator $m^{s, q} $ on sequences.
First, for a sequence $r = \{r_Q\}_{Q \in \Q}$, we define
\begin{equation*}
	G^{s , q}( r ) (x) : = \left( \sum_{P \in \Q}    ( |P|^{-s -1/2 }  | r_P| )^q \chi_P (x)  \right)^{1/q}
\end{equation*}

and
\begin{equation*}
	G^{s , q}_Q ( r ) (x) : = \left( \sum_{P \in \Q, \scale (P) \le \scale (Q), P \cap Q \neq \emptyset }    ( |P|^{-s -1/2}  | r_P| )^q \chi_P (x)  \right)^{1/q} .
\end{equation*}

We let $m^{s, q}_Q (r) $ denote the ``$\frac{1}{4}$-median'' of $	G^{s , q}_Q ( r ) $ on $Q$, i.e.,
\begin{equation} \label{eq def m alpha q Q}
	m^{s, q}_Q (r) : = \inf\{ \epsilon : |\{ x \in Q : 	G^{s , q}_Q ( r ) (x)> \epsilon \}| < |Q|/4 \}.
\end{equation}
We also set
\begin{equation*}
	m^{s, q} (r)  (x)  :  =\sup_Q  m^{s, q}_Q (r) \chi_Q (x) .
\end{equation*}

\begin{proposition}\label{prop f approx Lp}
	Let $A$ be a dilation. Let $s\in \mathbb R$, 
	$p, q \in (0,\infty]$.  
	Then 
	for $r = \{ r_Q\} _{Q\in \Q}  \subset \mathbb C $,
	\begin{equation*}
		\| r\|_{  \dot f_{p,q}^{s} (A)  } \approx  \| m^{s, q} (r) \|_{L^p} .
	\end{equation*}
\end{proposition}
\begin{proof}
	Observe that for $\beta >0$,
	\begin{equation} \label{eq m alpha q subset}
		\{ x\in \rn : m^{s, q} (r)  (x)  > \beta   \}  \subset  \left\{x\in \rn :  \M_{\rho_A} (  \chi_{ \{  y: G^{s , q}( r ) (y)  > \beta \}  } )   (x) \ge\frac{ 1 }{ 4 C_{n,A} }  \right \} .
	\end{equation}
	Indeed,  suppose that $ m^{s, q} (r)  (x_0) > \beta $. Then there exists a dilated cube $x_0 \in  Q\in \Q$ such that $ m^{s, q}_Q (r)   > \beta $. By the definition of $ m^{s, q}_Q (r)$, we get 
	\begin{equation*}
		|\{ x \in Q : 	G^{s , q}_Q ( r ) (x)> \beta \}| \ge  |Q|/4 .
	\end{equation*}
	Note that for all $x \in Q$,  $	G^{s , q}_Q ( r ) (x)  \le 	G^{s , q} ( r ) (x) $. Hence 
	\begin{equation*}
		\{ x \in Q : 	G^{s , q}_Q ( r ) (x)> \beta \}  \subset \{ x \in Q : 	G^{s , q} ( r ) (x)> \beta \} .
	\end{equation*}
	and 
	\begin{equation*}
		| \{ x \in Q : 	G^{s , q} ( r ) (x)> \beta \}  | \ge  | \{ x \in Q : 	G^{s , q}_Q ( r ) (x)> \beta \} | \ge |Q|/4.
	\end{equation*}
	For this $Q$, select a $\rho_A$-ball $B_Q \in \B$ such that $Q\subset B_Q$ and $ |B_Q|  \le C_{n,A} |Q|  $ where the constant $C_{n,A}$ only depends on $n$ and $A$. Then 
	\begin{align*}
		\M_{\rho_A} (  \chi_{ \{  y: G^{s , q}( r ) (y)  > \beta  \} } )   (x)  & = \sup_{B\in\B} \frac{1}{|B|} \int_B   \chi_{ \{ y: G^{s , q}( r ) (y)  > \beta \}  } (z)  \d z  \\
		& \ge \frac{1}{|B_Q|} \int_{B_Q}   \chi_{ \{ y: G^{s , q}( r ) (y)  > \beta \}  } (z)  \d z  \\
		& \ge \frac{1}{C_{n,A} |Q|} \int_{Q}   \chi_{ \{ y: G^{s , q}( r ) (y)  > \beta \}  } (z)  \d z \\
		&\ge \frac{1}{4C_{n,A}},
	\end{align*}
	which leads to (\ref{eq m alpha q subset}).
	Since $\M_{\rho_A} $  is of weak-type (1,1)(Lemma \ref{lemma M weak 1,1 strong p,p}), we obtain 
	\begin{equation*}
		|\{ x\in \rn : m^{s, q} (r)  (x)  >  \beta   \} |  \le c | \left\{ x\in \rn :   G^{s , q}( r ) (x)  > \beta      \right \} |,
	\end{equation*}
	for all $\beta >0$, and, hence
	\begin{equation*}
		\| m^{s, q} (r) \|_{L^p} \lesssim \|  G^{s , q}( r ) \|_{L^p} = c \|r\|_{  \dot f_{p,q}^{s} (A) } ,
	\end{equation*}
	for $0<p <\infty $. When $p =\infty$, use Chebyshev's inequality to get that
	\begin{equation*}
		|	\{ x \in Q : 	G^{s , q}_Q ( r ) (x)> \beta\} | \le \frac{1}{\beta ^q } \int_Q  (	G^{s , q}_Q ( r ) (x) )^q \d x  \le \frac{|Q|}{\beta ^q}  \| r\|_{ \dot f_{\infty ,q}^{s} (A) } ^q < \frac{1}{4} |Q|
	\end{equation*}
	if $ \beta > 4^{1/q} \| r\|_{ \dot f_{\infty ,q}^{s} (A) } $. Hence,  $\| m^{s, q} (r) \|_{L^\infty}    \le c \| r\|_{ \dot f_{\infty ,q}^{s} (A) }  $.
	
	Next we prove the converse inequality. Define the extended integer-valued stopping time $\nu(x)$, for $x\in \rn$, by
	\begin{equation*}
		\nu(x) : = \inf \left\{\nu \in \mathbb Z:   \left(  \sum_{ |Q| \le |\det A|^{-\nu} }  (  |Q| ^{-s -1/2} |r_Q| )^q  \chi_Q (x)    \right)^{1/q}    \le    m^{s, q} (r)  (x) \right\}.
	\end{equation*}
	By (\ref{eq def m alpha q Q}), 
	we have
	\begin{equation*}
		|\{ x \in Q : 	G^{s , q}_Q ( r ) (x)>  m^{s, q}_Q (r) + \epsilon \}| < |Q|/4 .
	\end{equation*}
	Letting $\epsilon \to 0^+$, we get 
	\begin{equation*}
		|\{ x \in Q : 	G^{s , q}_Q ( r ) (x) \le  m^{s, q}_Q (r)  \}| \ge  \frac{ 3 |Q| }{4} .
	\end{equation*}
	For all $x \in  \{ y \in Q : 	G^{s , q}_Q ( r ) (y) \le  m^{s, q}_Q (r)  \}$, since $ m^{s, q}_Q (r)  \le m^{s, q} (r) (x) $, we obtain   $G^{s , q}_Q ( r ) (x) \le m^{s, q} (r) (x) $. Thus
	\begin{equation*}
		\{ x\in Q:  	G^{s , q}_Q ( r ) (x) \le  m^{s, q} (r)  (x) \}  \ge \{ x\in Q:  	G^{s , q}_Q ( r ) (x) \le  m^{s, q}_Q (r)  \} \ge \frac{ 3 |Q| }{4} .
	\end{equation*}
	
	Set $ X := \{ x\in Q:   |\det A|^{ - \nu (x)}      \ge |Q| \}  $ and $Y := \{ x\in Q:  	G^{s , q}_Q ( r ) (x) \le  m^{s, q} (r)  (x) \}$, Next we prove that $X=Y$.

	Suppose that $x \in X$. That means $x\in Q $ and    $ |\det A|^{ - \nu (x)}      \ge |Q|$.
	Hence 
	\begin{align*}
		G^{s , q}_Q ( r ) (x) & = \left( \sum_{P \in \Q, \scale (P) \le \scale (Q), P \cap Q \neq \emptyset }    ( |P|^{-s -1/2}  | r_P| )^q \chi_P (x)  \right)^{1/q} \\
		&  \le  \left( \sum_{P \in \Q, \scale (P) \le - \nu(x)  }    ( |P|^{-s -1/2}  | r_P| )^q \chi_P (x)  \right)^{1/q} \\
		&  \le    m^{s, q} (r)  (x) .
	\end{align*}
	Thus $x\in Y$ and $X \subset Y$. On the other hand, suppose that $x\in Y$.  We proceed by contradiction.
	Suppose that $  |\det A|^{ - \nu (x)}    < |Q| $. Thus, $\scale (Q) \ge -\nu (x) +1$ and 
	\begin{equation*}
		\left(  \sum_{ |P| \le |\det A|^{-\nu(x) + 1} }  (  |P| ^{-s -1/2} |r_P| )^q  \chi_P (x)    \right)^{1/q}     >     m^{s, q} (r)  (x) .
	\end{equation*}
	But for $x\in Q$, we have
	\begin{align*}
		G^{s , q}_Q ( r ) (x) & = \left( \sum_{P \in \Q, \scale (P) \le \scale (Q), P \cap Q \neq \emptyset }    ( |P|^{-s -1/2}  | r_P| )^q \chi_P (x)  \right)^{1/q} \\
		& \ge \left( \sum_{P \in \Q, \scale (P) \le -\nu(x) + 1 , P \cap Q \neq \emptyset }    ( |P|^{-s -1/2}  | r_P| )^q \chi_P (x)  \right)^{1/q} \\
		& = \left( \sum_{P \in \Q, \scale (P) \le -\nu(x) + 1  }    ( |P|^{-s -1/2}  | r_P| )^q \chi_P (x)  \right)^{1/q} \\
		& >  m^{s, q} (r)  (x) .
	\end{align*}
	This  contradicts 	$G^{s , q}_Q ( r ) (x)  \le m^{s, q} (r)  (x) $. Hence $ |\det A|^{ - \nu (x)}   \ge |Q| \implies  x\in X  \implies  Y \subset X$. Thus $X=Y$.
	
	Now set
	\begin{equation*}
		E_Q := \{ x\in Q:   |\det A|^{ - \nu (x)}      \ge |Q| \} = \{ x\in Q:  	G^{s , q}_Q ( r ) (x) \le  m^{s, q} (r)  (x) \} 
	\end{equation*}
	for each $Q \in \Q$. Then $ |E_Q| \ge \frac{3|Q|}{4}$	
	and 
	\begin{equation} \label{eq le m alpha q}
		\left(  \sum_{ Q \in \Q  }  (  |Q| ^{-s -1/2} |r_Q| )^q  \chi_{E_Q} (x)    \right)^{1/q}   \le \left(  \sum_{ |Q| \le |\det A|^{-\nu (x)} }  (  |Q| ^{-s -1/2} |r_Q| )^q  \chi_Q (x)    \right)^{1/q}   \le   m^{s, q} (r)  (x) 
	\end{equation}
	for each $x\in \rn$.

	By Lemma \ref{lemma sparse char dot f}, for $0<p<\infty$, $\|r\|_{  \dot f_{p,q}^{s} (A) } \le c   \| m^{s, q} (r) \|_{L^p}$. Similarly, (\ref{eq le m alpha q}) and Lemma \ref{lemma p = infty EQ} yield $\|r\|_{  \dot f_{\infty,q}^{s} (A) } \le c   \| m^{s, q} (r) \|_{L^\infty}. $
\end{proof}

\begin{proposition} \label{prop p = infty char E_Q}
	Let $A$ be a dilation. Let $s\in \mathbb R$, $ q \in (0,\infty]$. 
	Then $r = \{ r_Q\}_{Q \in \Q}  \in \dot f_{\infty,q}^{s} (A) $  if and only if for each $Q$ there is a subset $E_Q \subset Q$ with $ |E_Q| / |Q| > 1/2$ (or any other, fixed, number $0<\epsilon <1$)  such that
	\begin{equation} \label{eq E_Q Lp}
		\left\|   \left(  \sum_{Q \in \Q}   \left(  |Q|^{ -s -1/2 } |s_Q| \right)^q \chi_{E_Q}   \right)^{1/q}    \right\|_{L^\infty} <\infty .
	\end{equation}
	Moreover, the infimum of this expression over all such collections $\{E_Q\} _{Q\in \Q} $  is equivalent to $ \dot f_{\infty,q}^{s} (A) $. 
\end{proposition}

\begin{proof}
	Let $r \in \dot f_{\infty,q}^{s} (A)$. Then the $E_Q$'s chosen in the proof of Proposition \ref{prop f approx Lp} above yield (\ref{eq E_Q Lp}). 
	The converse follows from Lemma \ref{lemma p = infty EQ}.
\end{proof}

\begin{proof}[Proof of Proposition \ref{prop seq p infty char}]
	Case $q =\infty$. First note that
	\begin{equation*}
		\| r\|_{  \dot f_{\infty,\infty }^{s} (A)  }  
		=  \sup_{Q\in \Q} |Q|^{ - (s -1/2 )  }  |r_Q| .
	\end{equation*}
	Then the left-hand side of  (\ref{eq f s q infty equ}) is 
	\begin{align*}
		\sup_{P \in \Q} \left( \fint_P   \left(   \sup_{ \scale(Q) \le \scale (P)}    |Q|^{ - (s -1/2 ) }  |r_Q|   \chi_Q (x) \right)^{u}  \d x   \right) ^{1/u}  =  \sup_{Q\in \Q} |Q|^{ - (s -1/2 )  }  |r_Q|  .
	\end{align*}
	
	Case  $ 0<q<\infty  $ and  $\infty > u \ge q $. By H\"older's inequality, the right-hand side of (\ref{eq f s q infty equ}) is dominated by the left. On the other hand, if $P$ is fixed dilated cube and $E_Q$ are the subsets given by Proposition \ref{prop f approx Lp}, then by  Lemma \ref{lemma sparse char dot f}, we obtain
	\begin{align*}
		\fint_P  \left(   \sum_{ \scale(Q) \le \scale (P) }   |Q|^{ - (s -1/2 )q }  |r_Q| ^q   \chi_Q (x) \right)^{u/q}  \d x  \le c 	\fint_P  \left(   \sum_{ \scale(Q) \le \scale (P) }   |Q|^{ - (s -1/2 )q }  |r_Q| ^q   \chi_{E_Q} (x) \right)^{u/q}  \d x  .
	\end{align*}
	Now this is clearly less than
	\begin{equation*}
		c   \left\| \left( \sum_{ \scale(Q) \le \scale (P)}   |Q|^{ - (s -1/2 )q }  |r_Q| ^q   \chi_{E_Q}  \right)^{1/q}  \right\|_{L^\infty}^u,
	\end{equation*}
	and by Proposition \ref{prop p = infty char E_Q} this can be estimated by $c \| r\|_{ \dot f_{\infty,q}^{s} (A)}^u$.
	
	If $ 0<q<\infty  $ and  $ 0< u < q $, then H\"older's inequality shows that  the left-hand side of (\ref{eq f s q infty equ}) is dominated by the right. To prove the converse inequality, use Chebyshev's inequality and get
	\begin{align*}
		|	\{ x \in Q : 	G^{s , q}_Q ( r ) (x)> \beta\} | & \le \frac{1}{\beta ^u } \int_Q  (	G^{s , q}_Q ( r ) (x) )^u \d x  \\
		& \le \frac{1}{\beta ^u}  \int_Q   \left(   \sum_{ \scale(R) \le \scale (Q) }   |R|^{ - (s -1/2 )q }  |r_R| ^q   \chi_R (x) \right)^{u/q}  \d x    < \frac{1}{4} |Q|.
	\end{align*}
	if
	\begin{equation*}
		\beta > 4^{1/u} \sup_{P \in \Q} \left( \fint_P   \left(   \sum_{ \scale(Q) \le \scale (P) }   |Q|^{ - (s -1/2 )q }  |r_Q| ^q   \chi_Q (x) \right)^{u/q}  \d x   \right) ^{1/u} .
	\end{equation*} 
	Hence 
	\begin{equation*}
		\| m^{s, q} (r) \|_{L^\infty}  \le c \sup_{P \in \Q} \left( \fint_P   \left(   \sum_{ \scale(Q) \le \scale (P) }   |Q|^{ - (s -1/2 )q }  |r_Q| ^q   \chi_Q (x) \right)^{u/q}  \d x   \right) ^{1/u}  .
	\end{equation*}
	By Proposition \ref{prop f approx Lp}, we have
	\begin{equation*}
		\| r\|_{  \dot f_{\infty,q}^{s} (A)  } \approx  	\| m^{s, q} (r) \|_{L^\infty}   \lesssim \sup_{P \in \Q} \left( \fint_P   \left(   \sum_{ \scale(Q) \le \scale (P) }   |Q|^{ - (s -1/2 )q }  |r_Q| ^q   \chi_Q (x) \right)^{u/q}  \d x   \right) ^{1/u}  .
	\end{equation*}
	This completes the proof.
\end{proof}

\subsection{Properties of matrix-weighted anisotropic Triebel-Lizorkin spaces $p =\infty$}
Now we study some properties of $\dot F_{\infty,q}^{s} (\{A_Q\}_{Q\in \Q} ,A)$  and $ \dot f_{\infty,q}^{s} (\{A_Q\}_{Q\in \Q} ,A)$.

\begin{lemma} [Lemma 3.10, \cite{Bo07}] \label{lemma peetre p infty}
	Let $A $ be a dilation. 
		Let $s\in \mathbb R$, 
	$q \in (0,\infty]$.  (Note that Lebesgue measure is $\rho_A$-doubling measure with a constant $\beta =1$).
	Then for any $u=\{u_Q\}_{Q\in \Q} \subset \mathbb C$, $r>0$  and $\lambda > \max (1, r/q)$, 
	\begin{equation*}
		\|  u  \|_{ \dot f _{ \infty, q} ^{s} (A  )  } \approx \left \|  u^\ast_{ r, \lambda}     \right \|_{  \dot f _{ \infty, q} ^{s} (A  )} .
	\end{equation*}	
\end{lemma}
By Lemma \ref{lemma peetre p infty}, we obtain the following result.

\begin{lemma} \label{lemma f infty TL peetre seq}
	Let $A $ be a dilation. 	Let $s\in \mathbb R$, $p\in (0,\infty)$,
	$q \in (0,\infty]$.  
	Let  $W \in \A_p$ and $\{ A_Q\}_{Q\in \Q}$ be a sequence of reducing operators of order $p$ for $W$.
		Then for any $\vec u=\{\vec u_Q\}_{Q\in \Q} \subset \mathbb C^m$, $r>0$  and $\lambda > \max (1, r/q)$, 
	\begin{equation*}
		\| \vec u  \|_{ \dot f _{ \infty, q} ^{s} (\{A_Q\}_{Q\in \Q},A  )  } \approx \left \| \left( \{  |A_Q \vec u_Q | \}_{Q\in \Q} \right) ^\ast_{ r, \lambda}     \right \|_{ \dot f _{ \infty, q} ^{s} (A  ) } .
	\end{equation*}	 
\end{lemma}

\begin{theorem} \label{theo A_Q sup le inf} 
	Let $A $ be a dilation. 	Let $s\in \mathbb R$, $p\in (0,\infty)$,
	$q \in (0,\infty]$.
	Let  $W \in \A_p$ and $\mathbb A := \{ A_Q\}_{Q\in \Q}$ be a sequence of reducing operators of order $p$ for $W$. 
	Let $\varphi\in \mathcal S (\rn)$ satisfy (\ref{eq varphi supp}) and (\ref{eq varphi psi =1}).
	Then for any $\vec f \in (\S_\infty^\prime (\rn))^m $,
	\begin{equation} \label{eq F infty le f infty}
		\| \vec f\|_{ \dot F_{\infty,q}^{s} (\{A_Q\}_{Q\in \Q} ,A)  }  \le \left\|  \sup_{ \mathbb A , \varphi} \vec f      \right \|_{\dot f_{\infty,q}^{s} ( A) }  \lesssim 	\| \vec f\|_{  \dot F_{\infty,q}^{s} (\{A_Q\}_{Q\in \Q} ,A)  } .
	\end{equation}
\end{theorem}
\begin{proof}
	The first inequality is trivial. To prove the second inequality, we consider it into two cases on $q$.
	
		Case $q\in (0,\infty)$. 	
	Fix $r \in  (0,  \min\{1,p,q\} )$. Let $M > \Delta + 1 /r $ where $\Delta$ is the same as in Lemma \ref{lemma cube AQ AR-1}.	
	From (\ref{eq AQ r trick}), we have that for $x\in Q \in \Q_j$, 
\begin{equation*}
		|A_Q (\varphi_j * \vec f ) (x) |^r  \lesssim  |\det A|^j\int_\rn  |\mathbb A_j \varphi_j * \vec f (z) |^r   \frac{1}{  (1 + |\det A|^j \rho_A  ( x- z ) ) ^{(M -\Delta) r} }  \d z 
\end{equation*}
where  $\mathbb A_j$ is the same as in (\ref{eq def A_j}).
For each $j \in \mathbb Z$, 
let 
\begin{equation*}
	g_j := \sum_{Q \in \Q_j } \sup_{ \mathbb A , \varphi, Q  } \vec f \chi_Q  \quad \operatorname{and} \quad 	h_j = \mathbb A_j \varphi_j *\vec f .
\end{equation*}
By Lemma \ref{lemma r trick all  tau}, we obtain
\begin{align*}
\left\|  \sup_{ \mathbb A , \varphi} \vec f      \right\|_{\dot f_{\infty,q}^{s} ( A) } &  = 	\left\|  \sup_{ \mathbb A , \varphi} \vec f      \right\|_{\dot f_{q,q}^{s,1/q} ( A) } =  \left\| 	\{ |\det A|^{js} g_j  \}_{j\in\mathbb Z}   \right\|_{L \dot F_{q,q}^{1/q} }  \\
&  \lesssim \left\| 	\{ |\det A|^{js} h_j  \}_{j\in\mathbb Z}   \right\|_{L \dot A_{q,q}^{1/q} }
	=  \left\|   \vec f     \right  \|_{\dot F_{\infty,q}^{s} (\{A_Q\}_{Q\in \Q} ,A) }  .
\end{align*}
	
	Case $q=\infty$. It easy to see that
	\begin{align*}
		\sup_{ \mathbb A , \varphi, Q  } 	|A_Q (\varphi_j * \vec f ) (x) | \le \| \mathbb A_j \varphi_j * \vec f \|_{L^\infty} .
	\end{align*}
where   $\mathbb A_j$ is the same as in (\ref{eq def A_j}).
 Then 
\begin{align*}
	\|  \sup_{ \mathbb A , \varphi} \vec f       \|_{\dot f_{\infty,\infty}^{s} ( A) } =  	\| \vec f\|_{  \dot F_{\infty, \infty }^{s} (\{A_Q\}_{Q\in \Q} ,A)  }  = \| \sup_{j\in \mathbb Z}  \mathbb A_j \varphi_j * \vec f  \|_{L^\infty} .
\end{align*}
Thus the proof is complete.
\end{proof}

The following lemma shows that the inverse $\varphi$-transform $T_\psi$ is well defined for any $ \vec s \in \dot f_{\infty,q}^{s} (\{A_Q\}_{Q\in \Q},A)   $.

\begin{lemma}
	Let $A $ be a dilation. 	Let $s\in \mathbb R$, $p\in (0,\infty)$,
	$q \in (0,\infty]$.
	Let  $W \in \A_p$ and $\{ A_Q\}_{Q\in \Q}$ be a sequence of reducing operators of order $p$ for $W$.
Let
\begin{equation*}
	L > \max\{ \tilde  d / p'   - s +  1/2  , d/p    + s + 1/2 , \Delta + 1 \}
\end{equation*}
where $\tilde d, d, \Delta $  is the same as in Lemma \ref{lemma cube AQ AR-1}.
	Then there exist constants $N, C > 0$ such that for
	\begin{equation*}
		\sum_{Q\in \Q} |  \langle \psi_Q, \phi\rangle | | \vec s_Q| \le C  \| \vec s\|_{\dot f_{\infty,q}^{s} (\{A_Q\}_{Q\in \Q},A)  }   \|\psi\|_{\S_N} \|\varphi\|_{\S_N}.
	\end{equation*}
\end{lemma}

\begin{proof}
 The proof  is similar to that of  Lemma \ref{lemma well defined inverse varphi transform} and we omit it here.
\end{proof}

\begin{theorem} \label{theorem wavelet trans p infty TL}
	Let $A $ be a dilation. 
	Let $s\in \mathbb R$, $p\in  (0,\infty)$,
	$q \in (0,\infty]$.  
	Let  $W \in \A_p$ and $\{ A_Q\}_{Q\in \Q}$ be a sequence of reducing operators of order $p$ for $W$.
		Let $\varphi, \psi \in \mathcal S (\rn)$ satisfy (\ref{eq varphi supp}) and (\ref{eq varphi > 0}).
			Let $\tilde \varphi (x): = \overline{\varphi (-x)}$.
		Then the operators
	\begin{equation*}
		S_\varphi :  \dot F_{ \infty, q} ^{s} (\{A_Q\}_{Q\in \Q},A,  \tilde \varphi  ) \to \dot f _{ \infty, q} ^{s} (\{A_Q\}_{Q\in \Q},A  ) \; \operatorname{and } \;  T_\psi : \dot f _{ \infty, q} ^{s} (\{A_Q\}_{Q\in \Q},A  ) \to  \dot F_{ \infty, q} ^{s} (\{A_Q\}_{Q\in \Q},A, \varphi  )
	\end{equation*}
	are bounded.	 Moreover, if $\varphi$ and $\psi $  satisfy (\ref{eq varphi psi =1}), then $T_\psi \circ S_\varphi $  is the identity on	 $  \dot F_{ \infty, q} ^{s} (\{A_Q\}_{Q\in \Q},A , \tilde \varphi )$.  
\end{theorem}

\begin{proof}
	We first prove the boundedness of $S_\varphi:  \dot F_{ \infty, q} ^{s} (\{A_Q\}_{Q\in \Q},A ,\tilde \varphi ) \to \dot f _{ \infty, q} ^{s} (\{A_Q\}_{Q\in \Q},A  ) $.
	For any $\vec f \in  \dot F_{ \infty, q} ^{s} (\{A_Q\}_{Q\in \Q},A, \tilde \varphi  ) $, let 
	\begin{equation*}
		\sup_{ \mathbb A ,\tilde \varphi} \vec f := \left\{  \sup_{ \mathbb A , \tilde \varphi, Q  } \vec f \right\}_{Q \in \Q}, 
	\end{equation*}
	where 
	\begin{equation*}
		\sup_{ \mathbb A ,  \tilde \varphi, Q  } \vec f := \sup_{y\in Q} |A_Q   \tilde \varphi_Q * \vec f (y) |  =  |Q|^{1/2} \sup_{y\in Q} |A_Q   \tilde \varphi_{ -\scale (Q) } * \vec f (y) |  .
	\end{equation*}
	For any $Q \in \Q_j$,
	\begin{equation*}
		|A_Q ( S_\varphi \vec f )_Q | \le | A_Q \langle \vec f , \varphi_Q \rangle | = |Q|^{1/2} | A_Q   \tilde \varphi_{ j}  *\vec f  ( x_Q) | \le \sup_{ \mathbb A ,  \tilde \varphi, Q  } \vec f  .
	\end{equation*}
	By Theorem \ref{theo A_Q sup le inf},
	\begin{equation*}
		\| S_\varphi \vec f \|_{  \dot f _{ \infty, q} ^{s} (\{A_Q\}_{Q\in \Q},A  )  } \le \left\|  \sup_{ \mathbb A ,  \tilde \varphi, Q  } \vec f  \right\|_{ \dot f _{ \infty, q} ^{s} (\{A_Q\}_{Q\in \Q},A  )  }  \approx \left\|   \vec f  \right\|_{   \dot F_{ \infty, q} ^{s} (\{A_Q\}_{Q\in \Q},A , \tilde \varphi )} .
	\end{equation*}
	
	Now we prove the boundedness of $T_\psi : \dot f _{ \infty, q} ^{s} (\{A_Q\}_{Q\in \Q},A  )  \to
	 \dot F_{ \infty, q} ^{s} (\{A_Q\}_{Q\in \Q},A,\varphi   ) $. 	 
	 Given any $x\in \rn$,    for each $i \in \{ j-M , j-M +1, \ldots, j+M  \}$,   let  $Q^{ (i)} $ be the unique dilated cube such that $x \in Q^{(i)}$ and $Q^{ (i)}  \in \Q_i $.
	  From the proof of Theorem \ref{theorem wavelet trans}, we have
	 for $x\in Q  \in \Q_j $,
	 \begin{align*}
	 	|  A_Q \varphi_j * (T_\psi  \vec t ) ( x) |  
	  \lesssim  \sum_{ i  = j-M }^{j+M}  |\det A|^{ i /2 }   \left (  \left( \{  |A_R \vec t_R \}_{R\in \Q} \right) ^\ast_{ 1, L-\Delta}    \right)_{Q^{ (i)}}  .
	 \end{align*}
where $\Delta$ is the same as in Lemma \ref{lemma cube AQ AR-1}.	 
	By choosing $  L -\Delta > 1/\min\{1,q \}  $, Lemma \ref{lemma f infty TL peetre seq} yields
	\begin{align*}
		\| T_\psi \vec t \|_{\dot F_{ \infty, q} ^{s} (\{A_Q\}_{Q\in \Q},A , \varphi   ) } \lesssim   \left\| \left( \{  |A_R \vec t_R \}_{R\in \Q} \right) ^\ast_{ 1, L-\Delta}  \right\|_{\dot f _{ \infty, q} ^{s} (\{A_Q\}_{Q\in \Q},A  ) } \lesssim   \| \vec t  \|_{ \dot f _{ \infty, q} ^{s} (\{A_Q\}_{Q\in \Q},A  ) } .
	\end{align*}

	Finally, if $\varphi$ and $\psi $  satisfy (\ref{eq varphi psi =1}), then, by Lemma \ref{lemma identity}, we find that $T_\psi \circ S_\varphi $  is the identity on	 $  \dot F_{ \infty, q} ^{s} (\{A_Q\}_{Q\in \Q},A , \tilde \varphi )$. 
\end{proof}

By an argument similar to that used in the proof of Theorem \ref{theorem independent varphi}, we obtain the following result that space $  \dot F_{ \infty, q} ^{s} (\{A_Q\}_{Q\in \Q},A  )$ is independent of the choice of $\varphi$; we omit the details.
\begin{theorem}\label{theorem indepen varphi p infty}
	Let $A $ be a dilation. 
	Let $s\in \mathbb R$, $p\in (0,\infty)$,
	$q \in (0,\infty]$.  
	Let  $W \in \A_p$ and $\{ A_Q\}_{Q\in \Q}$ be a sequence of reducing operators of order $p$ for $W$.
	Then $  \dot F_{ \infty, q} ^{s} (\{A_Q\}_{Q\in \Q},A  )$ is independent of the choice of $\varphi$.
\end{theorem}

The following conclusion is the main result of this section.
\begin{theorem}\label{theorem f infty f 1/p}
		Let $A $ be a dilation. 
	Let $s\in \mathbb R$, 
	$q \in (0,\infty]$, and  $p\in (0,\infty)$.  
	Let  $W \in \A_p$ and $\{ A_Q\}_{Q\in \Q}$ be a sequence of reducing operators of order $p$ for $W$.
	Then $  \dot F_{ \infty, q} ^{s} (\{A_Q\}_{Q\in \Q},A  ) = \dot F_{ p, q} ^{s, 1/p} (\{A_Q\}_{Q\in \Q},A  )$ and $  \dot f_{ \infty, q} ^{s} (\{A_Q\}_{Q\in \Q},A  ) = \dot f_{ p, q} ^{s, 1/p} (\{A_Q\}_{Q\in \Q},A  )$.
\end{theorem}
\begin{proof}
	We first show that  $ \dot f_{ \infty, q} ^{s} (\{A_Q\}_{Q\in \Q},A  ) =\dot f_{ p, q} ^{s, 1/p} (\{A_Q\}_{Q\in \Q},A  )$.
	For $\vec t =\{ \vec t_Q\}_{Q\in \Q} \subset \mathbb C^m$, define $ u : = \{u_Q \}_{ Q\in \Q}$ by setting $u_Q := |A_Q \vec t_Q|$ for each $Q\in \Q$.
	By Proposition \ref{prop seq p infty char}, we obtain
	\begin{equation*}
		\| \vec  t\|_{\dot f_{ \infty, q} ^{s} (\{A_Q\}_{Q\in \Q},A  ) } = \| u \|_{\dot f_{ \infty, q} ^{s} (A  ) } \approx  \| u \|_{\dot f_{p, q} ^{s, 1/p } (A  ) }   =  \| \vec  t\|_{\dot f_{ p, q} ^{s, 1/p} (\{A_Q\}_{Q\in \Q},A  )  }.
	\end{equation*}
	
Next we prove that $ \dot F_{ \infty, q} ^{s} (\{A_Q\}_{Q\in \Q},A  ) =\dot F_{ p, q} ^{s, 1/p} (\{A_Q\}_{Q\in \Q},A  ) $.
From Theorems \ref{theorem wavelet trans p infty TL} and \ref{theorem wavelet trans} and the above result,  we have
\begin{align} \label{eq F infty F 1/p}
	\nonumber 
	\| \vec f\|_{\dot F_{ \infty, q} ^{s} (\{A_Q\}_{Q\in \Q},A  )  } & = \left\|  T_\psi \circ S_\varphi \vec f \right\|_{\dot F_{ \infty, q} ^{s} (\{A_Q\}_{Q\in \Q},A  )} \lesssim \left\|   S_\varphi \vec f \right\|_{\dot f_{ \infty, q} ^{s} (\{A_Q\}_{Q\in \Q},A  )}  \\
	& \approx \left\|   S_\varphi \vec f \right\|_{\dot f_{ p, q} ^{s, 1/p} (\{A_Q\}_{Q\in \Q},A  )} \lesssim \left\|   \vec f \right\|_{\dot F_{ p, q} ^{s, 1/p} (\{A_Q\}_{Q\in \Q},A  )} .
\end{align}
Applying an argument similar to that used in the estimation of (\ref{eq F infty F 1/p}), we obtain the reverse inequality. Thus, $\dot F_{ \infty, q} ^{s} (\{A_Q\}_{Q\in \Q},A  ) =\dot F_{ p, q} ^{s, 1/p} (\{A_Q\}_{Q\in \Q},A  )  $.
\end{proof}

Applying Propositions \ref{prop F embed S prime}, \ref{prop complete} and \ref{prop shfit}, we obtain the following results; we omit the details.

\begin{proposition} 
	Let $A $ be a dilation. 
	Let $s\in \mathbb R$, $p \in (0,\infty)$,
	$q \in (0,\infty]$.  
	Let  $W \in \A_p$ and $\{ A_Q\}_{Q\in \Q}$ be a sequence of reducing operators of order $p$ for $W$. Then 
	$ \dot F_{ \infty, q} ^{s} (\{A_Q\}_{Q\in \Q},A  ) \subset ( \S_\infty ^\prime )^m $,
	and $\dot F_{ \infty, q} ^{s} (\{A_Q\}_{Q\in \Q},A  ) $  is a   complete quasi-norm space.
\end{proposition}

\begin{proposition}
	Let $A $ be a dilation. 
	Let $s\in \mathbb R$, $p \in (0,\infty)$,
	$q \in (0,\infty]$.  
Let  $W \in \A_p$ and $\{ A_Q\}_{Q\in \Q}$ be a sequence of reducing operators of order $p$ for $W$.
	Let $\sigma\in \mathbb R$. 
	Then for any $\vec f \in  (\S_\infty ^\prime (\rn) )^m$, 
	\begin{equation*}
		\|\dot I_\sigma \vec f \|_{\dot F_{ \infty, q} ^{s -\sigma} (\{A_Q\}_{Q\in \Q},A  )   }  \approx \| \vec f \|_{\dot F_{ \infty, q} ^{s} (\{A_Q\}_{Q\in \Q},A  )   } .
	\end{equation*}
where $\dot I_\sigma $  is defined in (\ref{eq def I sigma}).
\end{proposition}
Next we obtain an embedding 	$ \dot A_{ p, q} ^{s, \tau} (\{A_Q\}_{Q\in \Q},A  )  \hookrightarrow \dot F_{ \infty, \infty } ^{s +\tau -1/p } (\{A_Q\}_{Q\in \Q},A  ) $ by the following proposition.
\begin{proposition}
		Let $A $ be a dilation. 
	Let $s\in \mathbb R$, $\tau \in [0,\infty) $,  $p \in (0,\infty)$,
$q \in (0,\infty]$. 
	Let  $W \in \A_p$ and $\{ A_Q\}_{Q\in \Q}$ be a sequence of reducing operators of order $p$ for $W$.
	Then \begin{equation*}
	 \dot A_{ p, q} ^{s, \tau} (\{A_Q\}_{Q\in \Q},A  )  \hookrightarrow \dot F_{ \infty, \infty } ^{s +\tau -1/p  } (\{A_Q\}_{Q\in \Q},A  ) .
	\end{equation*}
\end{proposition}
\begin{proof}
	From (\ref{eq AQ r trick}), we have that for any $p \in (0,\infty) $,
\begin{equation*}
	|A_Q (\varphi_j * \vec f ) (x) |^p \lesssim   \sum_{R \in \Q_j }   |\det A|^j \int_R  |A_R \varphi_j * \vec f (z) |^p   \frac{1}{  (1 + |\det A|^j \rho_A  ( x_Q- x_R ) ) ^{(M -\Delta) p} }  \d z .
\end{equation*}
where $\Delta$ is the same as in Lemma \ref{lemma cube AQ AR-1} and $(M -\Delta) p > 1 $.

For $i \in \mathbb N $, let $  E_i := \{ R\in \Q_j :  i  <  |\det A|^{j} \rho_A  ( x_Q- x_R )   \le  i + 1 \} $ and $ E_0 : = \{ R\in \Q_j :    \rho_A  ( x_Q- x_R )  \le |\det A|^{-j}  \}$.
Then 
\begin{align*}
|A_Q (\varphi_j * \vec f ) (x) |^p  & 	\lesssim  \sum_{i =0 }^\infty  \sum_{R \in E_i }   |\det A|^j   (1+ i )^{- (M -\Delta) p }   \int_R  |A_R \varphi_j * \vec f (z) |^p     \d z  \\
	&\le \| \vec f\|_{  \dot A_{ p, q} ^{s, \tau} (\{A_Q\}_{Q\in \Q},A  ) } ^p  \sum_{i =0 }^\infty   \sum_{R \in E_i }   |\det A|^j   (1+ i )^{- (M -\Delta) p }    |\det A|^{-jsp}  |\det A|^{ - j \tau p }  \\
	& = \| \vec f\|_{  \dot A_{ p, q} ^{s, \tau} (\{A_Q\}_{Q\in \Q},A  ) } ^p  |\det A|^{j  (1+ p ( -s-\tau ) )  }   \sum_{i =0 }^\infty  \sum_{R \in E_i }       (1+ i )^{- (M -\Delta) p }     \\
	&\lesssim \| \vec f\|_{  \dot A_{ p, q} ^{s, \tau} (\{A_Q\}_{Q\in \Q},A  ) } ^p  |\det A|^{j  (1+ p ( -s-\tau ) )}  .
\end{align*}
Thus
\begin{align*}
	\| \vec f\|_{ \dot F_{ \infty, \infty } ^{s +\tau -1/p } (\{A_Q\}_{Q\in \Q},A  )  }   = \sup_{j\in\mathbb Z} \sup_{Q\in \Q_j}     |\det A|^{j  (s +\tau -1/p )   } |A_Q (\varphi_j * \vec f ) (x) | 
	 \lesssim \| \vec f\|_{  \dot A_{ p, q} ^{s, \tau} (\{A_Q\}_{Q\in \Q},A  ) }.
\end{align*}
Hence the proof is complete.
\end{proof}

\begin{theorem}\label{theorem a tau supercritical}
		Let $A $ be a dilation. Let $\mathbb A = \{ A_Q\}_{Q\in \Q}$ be a family of positive definite matrices.
	Let $s\in \mathbb R$, $p\in (0,\infty)$, 
	$q \in (0,\infty]$. Suppose that   $q\in (0,\infty)$, $\tau \in (1/p, \infty) <0$ or $ q=\infty $, $\tau \in [1/p, \infty)$.
	Then for $r=\{r_Q\}_{\Q\in \Q} \subset \mathbb  C $,
	\begin{equation*}
		\| r \|_{ \dot a_{ p, q} ^{s, \tau} ( A  ) } \approx \| r\|_{ \dot f_{\infty, \infty} ^{ s +\tau  - 1/p  } (A) } . 
	\end{equation*}
and  for $\vec t =\{ \vec t_Q\}_{\Q\in \Q} \subset  \mathbb  C^m $
\begin{equation*}
		\| \vec t \|_{ \dot a_{ p, q} ^{s, \tau} ( \{A_Q\}_{Q\in \Q}, A  ) } \approx \| \vec t \|_{ \dot f_{\infty, \infty} ^{ s +\tau  - 1/p  } (\{A_Q\}_{Q\in \Q} , A) } . 
\end{equation*}
\end{theorem}

\begin{proof}
	It suffices to show  $	\| r \|_{ \dot a_{ p, q} ^{s, \tau} ( A  ) } \approx \| r\|_{ \dot f_{\infty, \infty} ^{ s +\tau  - 1/p  } (A) } $.
	For $r = \{ r_{ Q,j } \}_{j \in \mathbb Z, Q\in \mathcal Q_j}   $, 
	since $r_P $  is constant on $P$, we have
	\begin{align*}
		\| r \|_{ \dot a_{ p, q} ^{s, \tau} ( A  ) } & =  	\left\| \left\{ |\det A|^{j (s +1/2)} |r_{Q,j} |\right \}_{j\in \mathbb Z } \right\|_{ L \dot A _{ p,q }^\tau }  \\
		& =\sup_{P\in \Q} |P|^{-\tau}   \left\|  \left\{ |\det A|^{j (s +1/2)} |r_{Q,j} |\right \}_{j\in \mathbb Z } \right  \|_{ L \dot A _{ p,q }  (\widehat P)}  \\
		& \ge \sup_{P\in \Q} |P|^{-\tau}  \|   |\det A|^{ -\scale (P) (s +1/2)}  r_{P}  \|_{ L ^p  ( P)}  \\
		& = \sup_{P\in \Q} |P|^{-\tau +1/p }  \|   |\det A|^{ -\scale (P) (s +1/2)}  r_{P}  \|_{ L ^\infty   ( P)}  \\
		& = \| r\|_{ \dot f_{\infty, \infty} ^{ s +\tau  - 1/p } (A) } .
	\end{align*}
For the other direction, note that 
\begin{align*}
& \left\|  \left\{ |\det A|^{j (s +1/2)} |r_{Q,j} |\right \}_{j\in \mathbb Z } \right\|_{ L \dot A _{ p,q }  (\widehat P)} \\
&  \le  \left\|  \left\{ |\det A|^{j (s +1/2)}   \|r_{Q,j}  \|_{\ell^\infty (Q\in \Q_j) } \right \}_{j \ge -\scale (P) } \chi_P \right\|_{ L \dot A _{ p,q }   } \\
 & \le  \| r\|_{ \dot f_{\infty, \infty} ^{ s +\tau  - 1/p  } (A) }  \left\|  \left\{ |\det A|^{j (1/p-\tau) }   \right \}_{j \ge -\scale (P) } \chi_P \right\|_{ L \dot A _{ p,q }   } .
\end{align*}
Suppose that   $q\in (0,\infty)$, $\tau \in (1/p, \infty) $ or $ q=\infty $, $\tau \in [1/p, \infty)$. Then 
\begin{align*}
	 \left\|  \left\{ |\det A|^{j (1/p-\tau) }   \right \}_{j \ge -\scale (P) } \chi_P \right\|_{ L \dot A _{ p,q }   }  = \| \chi_P \|_{L^p }  \left( \sum_{j \ge -\scale (P) }  |\det A|^{ jq (1/p-\tau)} \right)^{1/q} \approx  |P|^{ \tau }.
\end{align*}
Thus 
\begin{equation*}
	\| r \|_{ \dot a_{ p, q} ^{s, \tau} ( A  ) } = \sup_{P\in \Q} |P|^{-\tau}   \left\|  \left\{ |\det A|^{j (s +1/2)} |r_{Q,j} |\right \}_{j\in \mathbb Z } \right  \|_{ L \dot A _{ p,q }  (\widehat P)} \lesssim \| r\|_{ \dot f_{\infty, \infty} ^{ s +\tau  - 1/p  } (A) } .
\end{equation*}
Thus the proof is complete.
\end{proof}

\begin{corollary} \label{cor p infty with type W}
		Let $A $ be a dilation. 
	Let $s\in \mathbb R$, $\tau \in [0,\infty) $,  $p \in (0,\infty)$,
$q \in (0,\infty]$. 
Let  $W \in \A_p$ and $\{ A_Q\}_{Q\in \Q}$ be a sequence of reducing operators of order $p$ for $W$.
	Then we have the following identifications of spaces with equivalent quasi-norms:
	
	{\rm (i)} if $\tau = 1/p$, then $ \dot f_{ p, q} ^{s, \tau} (W, A  )  = \dot f_{ \infty, q} ^{s} (\{A_Q\}_{Q\in \Q},A  ) $;
	
	{\rm (ii)} if $\tau > 1/p$  or $ (\tau,q) = (1/p,\infty)$, then $ \dot a_{ p, q} ^{s, \tau} (W, A  )  = \dot f_{ \infty, \infty} ^{ s +\tau  - 1/p} (\{A_Q\}_{Q\in \Q},A  ) $.

\end{corollary}
\begin{proof}
	By Theorem \ref{theorem seq W approx A_Q}, we have $ \dot a_{ p, q} ^{s, \tau} (W, A  ) =  \dot a_{ p, q} ^{s, \tau} (\{A_Q\}_{Q\in \Q}, A  ) $ in both cases under consideration.
	 
	(i) From Theorem \ref{theorem f infty f 1/p}, we obtain $  \dot f_{ \infty, q} ^{s} (\{A_Q\}_{Q\in \Q},A  )  =  \dot f_{ p, q} ^{s, 1/p} (\{A_Q\}_{Q\in \Q},A  ) $. Thus $ \dot f_{ p, q} ^{s, \tau} (W, A  )  = \dot f_{ \infty, q} ^{s} (\{A_Q\}_{Q\in \Q},A  ) $.
	
	(ii) From Theorem \ref{theorem a tau supercritical}, we obtain $ \dot f_{\infty, \infty} ^{ s +\tau  - 1/p  } (\{A_Q\}_{Q\in \Q} , A) =  \dot a_{ p, q} ^{s, \tau} (\{A_Q\}_{Q\in \Q},  A  )$. Thus  $ \dot a_{ p, q} ^{s, \tau} (W, A  ) =\dot f_{\infty, \infty} ^{ s +\tau  - 1/p  } (\{A_Q\}_{Q\in \Q} , A) $.
\end{proof}

\begin{corollary}
	Let $A $ be a dilation. 
	Let $s\in \mathbb R$, $\tau \in [0,\infty) $,  $p \in (0,\infty)$,
	$q \in (0,\infty]$. 
	Let  $W \in \A_p$ and $\{ A_Q\}_{Q\in \Q}$ be a sequence of reducing operators of order $p$ for $W$.	
	Then we have the following identifications of spaces with equivalent quasi-norms:
	
	{\rm (i)} if $\tau = 1/p$, then $ \dot F_{ p, q} ^{s, \tau} (W, A  )  = \dot F_{ \infty, q} ^{s} (\{A_Q\}_{Q\in \Q},A  ) $;
	
	{\rm (ii)} if $\tau > 1/p$  or $ (\tau,q) = (1/p,\infty)$, then $ \dot A_{ p, q} ^{s, \tau} (W, A  )  = \dot F_{ \infty, \infty } ^{ s +\tau  - 1/p} (\{A_Q\}_{Q\in \Q},A  ) $.
	
\end{corollary}
\begin{proof}
	Let $\varphi,\psi$ satisfy (\ref{eq varphi supp}), (\ref{eq varphi > 0}) and (\ref{eq varphi psi =1}). By Theorems \ref{theorem wavelet trans}, \ref{theorem independent varphi}, we obtain
	\begin{equation*}
			S_\varphi :  \dot A_{p,q}^{s,\tau}  (W, A) \to \dot a_{p,q}^{s,\tau}  (W, A)  \; \operatorname{and } \;  T_\psi : \dot a_{p,q}^{s,\tau}  (W, A) \to \dot A_{p,q}^{s,\tau}  (W, A) 
	\end{equation*}
are bound and $T_\psi \circ S_\varphi $ is the identity on $\dot A_{p,q}^{s,\tau}  (W, A)$.Similarly, from Theorem \ref{theorem wavelet trans p infty TL}, \ref{theorem indepen varphi p infty}, we get
	\begin{equation*}
	S_\varphi :  \dot F_{ \infty, q} ^{s} (\{A_Q\}_{Q\in \Q},A  ) \to \dot f _{ \infty, q} ^{s} (\{A_Q\}_{Q\in \Q},A  ) \; \operatorname{and } \;  T_\psi : \dot f _{ \infty, q} ^{s} (\{A_Q\}_{Q\in \Q},A  ) \to  \dot F_{ \infty, q} ^{s} (\{A_Q\}_{Q\in \Q}, A )
\end{equation*}
are bounded and  $T_\psi \circ S_\varphi $  is the identity on	 $  \dot F_{ \infty, q} ^{s} (\{A_Q\}_{Q\in \Q},A  )$.  By these results and Corollary \ref{cor p infty with type W}, we have
\begin{equation*}
  	T_\psi  \circ   S_\varphi : \dot X \xrightarrow[]{S_\varphi} \dot x = \dot y \xrightarrow[]{T_\psi} \dot Y
\end{equation*}
is bounded whenever
\begin{equation*}
	\{ ( \dot X, \dot x  ) ,  (\dot Y , \dot y)  \}  \subset \left\{  \left(  \dot A_{p,q}^{s,\tau}  (W, A), \dot a_{p,q}^{s,\tau}  (W, A) \right ),  \left( \dot F_{ \infty, \infty} ^{ s +\tau  - 1/p} (\{A_Q\}_{Q\in \Q},A  ), \dot f_{ \infty, \infty} ^{ s +\tau  - 1/p} (\{A_Q\}_{Q\in \Q},A  )  \right)     \right\}
\end{equation*}
where $ \tau > 1/p$  or $ (\tau, q) = (1/p,\infty) $, or 
\begin{equation*}
	\{ ( \dot X, \dot x  ) ,  (\dot Y , \dot y)  \}  \subset \left\{  \left(  \dot F_{p,q}^{s,\tau}  (W, A), \dot f_{p,q}^{s,\tau}  (W, A) \right ),  \left( \dot F_{ \infty, q} ^{ s } (\{A_Q\}_{Q\in \Q},A  ), \dot f_{ \infty, q} ^{ s} (\{A_Q\}_{Q\in \Q},A  )  \right)     \right\} .
\end{equation*}
On the other hand, $T_\psi \circ S_\varphi $   is the identity on each such $\dot X$. It follows that the identity is bounded from $ \dot X$ to $\dot Y$ for each pair $  (\dot X, \dot Y) $  as above. Since the roles of $\dot X$  and $\dot Y$  are exchangeable, we have $\dot X \subset \dot Y \subset \dot X $, and hence $ \dot X = \dot Y $. Thus the proof is finished.
\end{proof}

\noindent\textbf{Author Contributions}\quad 
All authors developed and discussed the results and contributed to the final
manuscript.

\medskip

\noindent\textbf{Data Availability}\quad Data sharing is not applicable 
to this article as no data sets were generated or analyzed.

\section*{Declarations}

\noindent\textbf{Conflict of interest}\quad All authors state no conflict of interest.

\medskip

\noindent\textbf{Informed Consent}\quad Informed consent has been obtained 
from all individuals included in this research work.


\providecommand{\bysame}{\leavevmode\hbox to3em{\hrulefill}\thinspace}
\providecommand{\MR}{\relax\ifhmode\unskip\space\fi MR }
\providecommand{\MRhref}[2]{%
	\href{http://www.ams.org/mathscinet-getitem?mr=#1}{#2}
}
\providecommand{\href}[2]{#2}

\bigskip

\noindent   Tengfei Bai, Pengfei Guo

\medskip

\noindent College of Mathematics and Statistics, Hainan Normal University, Haikou, Hainan 571158,
China

\medskip

\noindent Jingshi Xu (Corresponding author)

\medskip

\noindent 
School of Mathematics and Computing Science, Guangxi Colleges and Universities Key Laboratory of Data Analysis and Computation, Guilin University of Electronic Technology, Guilin, 541004, China 

\noindent Center for Applied Mathematics of Guangxi (GUET), Guilin, 541004, China

\smallskip

\noindent {\it E-mails}:
\texttt{202311070100007@hainnu.edu.cn} (T. Bai)

\noindent\phantom{{\it E-mails:}}
\texttt{guopf999@163.com} (P. Guo)

\noindent\phantom{{\it E-mails:}}
\texttt{jingshixu@126.com} (J. Xu)

\end{document}